\documentclass[12pt, reqno, a4paper]{amsart}
\usepackage{graphics,epsfig}
\usepackage{amsfonts,amsmath,amssymb,amsbsy,amsthm}
\usepackage{bm}
\usepackage{color}
\usepackage{float}
\usepackage{mathrsfs}
\usepackage{bbm}
\usepackage[normalem]{ulem}

\usepackage[left=2.2cm,right=2.2cm,top=3cm,bottom=3cm]{geometry}

\usepackage{environments}

\DeclareMathOperator{\divv}{div}
\renewcommand{\cases}[1]{\left\{ \begin{array}{rl} #1 \end{array} \right.}
\newcommand{\smfrac}[2]{{\textstyle \frac{#1}{#2}}}

\def\Xint#1{\mathchoice
{\XXint\displaystyle\textstyle{#1}}%
{\XXint\textstyle\scriptstyle{#1}}%
{\XXint\scriptstyle\scriptscriptstyle{#1}}%
{\XXint\scriptscriptstyle\scriptscriptstyle{#1}}%
\!\int}
\def\XXint#1#2#3{{\setbox0=\hbox{$#1{#2#3}{\int}$ }
\vcenter{\hbox{$#2#3$ }}\kern-.6\wd0}}
\def\mint{\Xint-}

\def\b{\big}
\def\B{\Big}
\def\bg{\bigg}

\def\sep{\,|\,}
\def\bsep{\,\b|\,}

\DeclareMathOperator*{\osc}{osc}

\def\R{\mathbb{R}}
\def\N{\mathbb{N}}
\def\Z{\mathbb{Z}}
\def\C{\mathbb{C}}

\def\WW{W}

\def\LL{L}

\def\dx{\,{\rm d}x}

\def\dr{\,{\rm d}r}
\def\dt{\,{\rm d}t}
\def\ds{\,{\rm d}s}
\def\dd{{\rm d}}
\def\pp{\partial}

\def\<{\langle}
\def\>{\rangle}

\def\mA{{\sf A}}
\def\mB{{\sf B}}
\def\mF{{\sf F}}

\def\mR{{\sf R}}

\def\mS{{\sf S}}

\def\bfg{{\bf g}}
\def\bfh{{\bf h}}

\def\bbA{\mathbb{A}}
\def\bbB{\mathbb{B}}
\def\bbC{\mathbb{C}}
\def\bbD{\mathbb{D}}

\newcommand{\Da}[1]{D_{\!#1}}
\newcommand{\Dc}[1]{\nabla_{#1}}
\def\D{\nabla}
\def\del{\delta}
\def\ddel{\delta^2}

\def\loc{{\rm loc}}

\def\eps{\varepsilon}

\renewcommand\a{{\rm a}}
\renewcommand\c{{\rm c}}
\newcommand\nn{{\rm nn}}
\newcommand\cn{{\rm cn}}

\newcommand{\supp}{{\rm supp}}

\newcommand{\weakto}{\rightharpoonup}

\def\L{{\Z^d}}

\def\Rg{\mathcal{R}}

\newcommand{\Lb}{{\Lambda^{\beta}}}
\newcommand{\Omb}{{\Omega^{\beta}}}

\def\Us{\mathscr{U}}

\def\Usc{\Us^{\rm c}}
\def\Ys{\mathscr{Y}}

\def\E{\mathscr{E}}
\def\Ea{\E^\a}
\def\Ec{\E^\c}
\def\Ech{\E^\c_h}

\def\Scb{\mS^\c}

\def\Pa{\mathscr{P}}

\def\Om{{\R^d}}
\def\Omh{{\Omega_h}}
\def\Omc{\Omega^\c}

\def\tily{\tilde y}
\def\tilz{\tilde z}
\def\tilu{\tilde u}
\def\tilv{\tilde v}

\def\bary{\bar y}
\def\barz{\bar z}
\def\barv{\bar v}
\def\barw{\bar w}
\def\baru{\bar u}

\def\Sa{\mS^\a}
\def\Sc{\mS^\c}

\def\zz{\bar{\zeta}}

\def\TT{\mathscr{T}}
\newcommand\Ts{\mathscr{T}}
\def\Th{\mathscr{T}_h}
\def\Nh{\mathscr{N}_h}

\def\ww{\omega}

\def\dW{\partial W}
\def\ddW{\partial^2 W}

\newcommand{\Y}{\mathscr{Y}}

\def\bl{\beta}

\def\Ebh{\mathscr{E}^\bl_h}

\newcommand{\T}{\mathcal{T}_h}

\newcommand{\ignore}[1]{}

\newcommand{\Real}{\mathbb{R}}

\definecolor{cocol}{rgb}{0.7, 0, 0}
\definecolor{bvcol}{rgb}{0.7,0,0.7}
\definecolor{ascol}{rgb}{0, 0.5, 0}

\begin{document}

\title[Analysis of Blended Atomistic/Continuum Coupling
Methods]{Analysis of Blended Atomistic/Continuum Hybrid Methods}

\author[X. Li]{Xingjie Helen Li}
\address{Xingjie Helen Li \\ The Division of Applied Mathematics at
  Brown University \\ 182 George St, Providence, RI, 02912 \\ USA}
\email{xingjie\_li@brown.edu}

\author[C. Ortner]{Christoph Ortner}
\address{Christoph Ortner\\ Mathematics Institute \\ Zeeman Building \\
  University of Warwick \\ Coventry CV4 7AL \\ UK}
\email{christoph.ortner@warwick.ac.uk}

\author[A. V. Shapeev]{\\Alexander V. Shapeev}
\address{Alexander V. Shapeev\\ School of Mathematics \\ University of Minnesota \\ 206 Church St. SE \\ Minneapolis, MN 55455, USA}
\email{alexander@shapeev.com}

\author[B. Van Koten]{Brian Van Koten}
\address{B. Van Koten\\ The University of Chicago \\ Department of Statistics \\
Eckhart Hall Room 108 \\ 5734 S University Ave \\ Chicago, IL 60637, USA}
\email{vankoten@galton.uchicago.edu}

\numberwithin{equation}{section}
\numberwithin{theorem}{section}

\date{\today}

\thanks{XHL was supported by an AMS-Simons Travel Grant. CO's work was
  supported by EPSRC grant EP/H003096, ERC Starting Grant 335120 and
  by the Leverhulme Trust through a Philip Leverhulme Prize.  AVS was
  supported by the AFOSR Award FA9550--12--1--0187.  }



\begin{abstract}
  We present a comprehensive error analysis of two prototypical
  atomistic-to-continuum coupling methods of blending type: the
  energy-based and the force-based quasicontinuum methods.

  Our results are valid in two and three dimensions, for finite range
  many-body interactions (e.g., EAM type), and in the presence of
  lattice defects (we consider point defects and dislocations). The
  two key ingredients in the analysis are (i) new force and energy
  consistency error estimates; and (ii) a new technique for proving
  energy norm stability of a/c couplings that requires only the
  assumption that the exact atomistic solution is a stable
  equilibrium.
\end{abstract}

\maketitle


\section{Introduction}
\label{sec:intro}
Atomistic-to-continuum coupling methods (a/c methods) are a class of
concurrent multi-scale schemes coupling molecular mechanics models of
atomistic processes with continuum mechanics models of long-ranged
elastic fields. A recent extensive overview and benchmark of a/c
schemes for material defect simulation is presented in
\cite{Miller:2008}. These schemes can, broadly, be categorised into
sharp-interface couplings and blending methods. Each of these
categories can further be divided into energy-based (conservative) and
force-based (non-conservative) a/c couplings. In the present paper we
develop a comprehensive error analysis of both energy-based and
force-based a/c couplings of blending type, which forms the
theoretical background for the optimised formulations in
\cite{2012-CMAME-optbqce, BQCFcomp}.

Precisely, we will consider (i) the B-QCE scheme formulated in
\cite{VkLusk2011:blended, 2012-CMAME-optbqce}, which is closely
related to methods proposed in \cite{XiBe:2004,
  bauman:applicationofArlequin,
  BadiaParksBochevGunzburgerLehoucq:2007}; and (ii) the B-QCF scheme
formulated in \cite{Lu.bqcf:2011, BQCF, BQCFcomp}, which is closely
related to methods proposed in
\cite{BadiaParksBochevGunzburgerLehoucq:2007,
  bauman:applicationofArlequin, xiao:bridgingdomain,
  fish:concurrentAtCcoupling, bridging,
  prudhomme:modelingerrorArlequin, seleson:bridgingmethods,
  XiBe:2004}. While our results are not be immediately applicable to
these related schemes \cite{XiBe:2004, bauman:applicationofArlequin,
  BadiaParksBochevGunzburgerLehoucq:2007, xiao:bridgingdomain,
  fish:concurrentAtCcoupling, bridging,
  prudhomme:modelingerrorArlequin, seleson:bridgingmethods}, we expect
that many of the techniques we develop can be employed to develop such
extensions.

In recent years a comprehensive numerical analysis theory of a/c
methods has begun to emerge, which is summarized in the review article
\cite{2013-atc.acta}. In one dimension, the foundations of this theory
are largely completed \cite{2013-atc.acta}. In two and three
dimensions only partial results exist to date: in
\cite{2012-MATHCOMP-qce.pair} sharp error bounds for an energy-based
coupling scheme are proven, in the presence of point defects. However,
the scheme itself is restricted to two dimensions and pair
interactions, and moreover, the analysis makes an assumptions on the
magnitude of the atomistic solution in order to establish stability of
the a/c scheme. In \cite{Lu.bqcf:2011} a sharp error estimate is
established, which is valid in two and three dimensions and for
general interatomic potentials; however, to establish stability of the
scheme it is assumed that the atomistic solution is globally smooth,
which therefore excludes the presence of lattice defects.

Our starting assumption is that the error analysis ought to be
performed in the energy-norm as this provides, to the best of our
knowledge, the only route at present to include crystal defects in the
analysis following \cite{2012-MATHCOMP-qce.pair, EhrOrtSha:defects,
  2013-atc.acta}.

Thus, there are two key difficulties in extending the one-dimensional
analysis in \cite{2013-atc.acta} (and references therein) to two and
three dimensions:
\begin{enumerate}
\item {\it Energy-norm consistency: } While consistency error
  estimates in $L^p$-type norms are readily obtained from elementary
  Taylor expansions, consistency error estimates in the negative
  energy norm are more difficult to obtain, since they require an
  analytically convenient ``weak form'' of the forces. The different
  interaction ranges of the continuum and atomistic models make this
  non-trivial as can, for example, be seen from the analysis in
  \cite{Or:2011a}, which develops such a ``weak form'' for energy-based
  sharp interface a/c couplings. In the present paper we draw from
  ideas in \cite{2012-ARMA-cb} to establish sharp consistency error
  estimates; see \S~\ref{sec:inter:bqce_stress} and
  \S~\ref{sec:int:bqcf_cons}.

\item {\it Stability: } A key observation in \cite{Lu.bqcf:2011} was
  that force-based blending (the B-QCF scheme) with a macroscopic
  blending width yields a ``universally stable'' a/c coupling in the
  terminology of \cite{OrtnerShapeevZhangV1}. However, stability is
  proven under conditions which, to our understanding, make it
  impossible to extend the analysis to situations with crystal
  defects, and the required blending width makes the scheme
  prohibitively expensive. In \cite{BQCF} it was then shown that the
  B-QCF scheme is also stable in a natural energy-norm, and that only
  a moderate blending width is required. However, this result required
  the assumption that a related B-QCE scheme is stable, which was
  still unknown.

  In the present work, we develop a new technique that allows us to
  prove stability of the B-QCE scheme; see
  \S~\ref{sec:int:stab_bqce}. After extending results from \cite{BQCF}
  and employing regularity estimates for the elastic fields generated
  by crystal defects \cite{EhrOrtSha:defects}, we are able to also
  conclude stability of the B-QCF scheme; see
  \S~\ref{sec:int:stab_bqcf}. Aside from technical conditions, our
  stability results only require the assumption that the atomistic
  equilibrium we are aiming to approximate is itself stable, but no
  assumptions on the magnitude or smoothness of the solution as in
  \cite{2012-MATHCOMP-qce.pair} or \cite{Lu.bqcf:2011} are required.
\end{enumerate}

The paper is structured as follows: In \S~\ref{sec:prelims} we
introduce a number of concepts that we require in order to formulate
the B-QCE and B-QCF schemes (\S~\ref{sec:defn_bqce} and
\S~\ref{sec:def_bqcf}), and to state the main results in
\S~\ref{sec:approx_error}. Our concluding remarks are also contained
in that section, in \S~\ref{sec:conclusion}.  In
\S~\ref{sec:intermediate} we present the key ideas and intermediate
results that are required to prove the main results. Finally, in
\S~\ref{sec:aux}--\S~\ref{sec:stab_prfs} we present the technical
details of the proofs.

\section{Prerequisites}
\label{sec:prelims}

\subsection{Generic notation}
\label{sec:prelim:notation}
Functions are normally maps from $\R^d \to \R^m$ or $\Z^d \to \R^m$
for some $d, m \in \{1, 2,3\}$. Vectors in $\R^d, \R^m$ or vectorial
functions are normally denoted by the symbols $y, z, u, v, w, f$. Lattice
sites, i.e. elements of $\Z^d$ are normally denoted by $\xi, \eta$,
while points in the continuous reference configuration are denoted by
$x \in \R^d$. We also identify $x$ with the identity map.

Matrices or matrix-valued functions are normally denoted by $\mA, \mB,
\mS, \mR$ and so forth. Tensors of fourth or higher rank are normally
denoted by $\bbA, \bbB, \bbC$, and so forth.

If a function $f : \R^d \to \R^m$ is (weakly) differentiable, then we
denote its jacobi matrix at $x$ by $\D f(x)$. If $f$ is scalar-valued,
then $\D^2 f(x)$ denotes the hessian matrix. In general, $\D^j f(x)$
denotes a tensor of order $m \times d \times \dots \times d$. Partial
derivatives with respect to some variable $s$ are denoted by
$\frac{\partial}{\partial s}$ or $\partial_s$. If $\alpha = (\alpha_1,
\dots, \alpha_d)$ is a multi-index, then $\partial_\alpha
= \partial_{x_{\alpha_1}} \cdots \partial_{x_{\alpha_d}}$.

Directional derivatives are denoted by $\D_\rho f := \D f \rho$, $\rho
\in \R^d$. If $\Rg \subset \R^d$ then we define a collection of
directional derivatives $\D_\Rg f(x) := (\D_\rho f(x))_{\rho \in
  \Rg}$.

Our use of tensor notation is intuitive and not crucial to follow the
main ideas. Nevertheless, for the sake of completeness we formally
define our notation.  The symbol $\otimes$ denotes the usual tensor
product: if $\bbA = (A_{i_1,\dots,i_r}) \in \R^{n_1\times\cdots\times
  n_r}$ and $\bbB = (B_{k_1,\dots,k_s}) \in \R^{m_1\times\cdots\times
  m_s}$, then $\bbA\otimes\bbB = (A_{i_1,\dots,i_r}B_{k_1,\dots k_s})
\in \R^{n_1\times \cdots \times n_r \times m_1 \times \cdots \times
  m_s}$. If $\bbA = (A_{i_1,\dots,i_{s+r}}) \in
\R^{m_1\times\cdots\times m_s \times n_1\times\cdots \times n_r}, \bbB
= (B_{j_1,\dots,j_r}) \in \R^{n_1\times\cdots\times n_r}$, then the
contraction operator is denoted by $(\bbA:\bbC)_{j_1,\dots,j_s} =
\sum_{i_1}^{n_1} \cdots \sum_{i_r}^{n_r}
A_{j_1,\dots,j_s,i_1,\dots,i_r} B_{i_1,\dots,i_r}$. In particular, if
$\bbA, \bbB$ have the same rank, then $\bbA:\bbB \in \R$ denotes the
euclidean inner product.

The symbol $\< \cdot, \cdot \>$ denotes an abstract duality pairing.
If $X, Y$ are normed linear spaces and $\mathscr{F} : X \to Y$ has
well-defined directional derivatives at a point $u \in X$, then we
denote the first of second derivatives, respectively, by
\begin{align*}
  \< \del \mathscr{F}(u), v \> &:= \lim_{t \to 0} t^{-1} \b(
  \mathscr{F}(u+tv) - \mathscr{F}(u)\b) \quad \text{and} \\
  \< \ddel \mathscr{F}(u) v, w \> &:= \lim_{t \to 0} t^{-1} \b\<
  \del\mathscr{F}(u+tw) - \del \mathscr{F}(u), v \b\>.
\end{align*}
Higher variations are defined recursively, e.g., $\< \del^3
\mathscr{F}(u) v_1, v_2, v_3 \> = \lim_{t \to 0} t^{-1} \< (\ddel
\mathscr{F}(u+tv_3) - \ddel \mathscr{F}(u)) v_1, v_2 \>$, whenever the
limit exists.

We use the standard definitions and notation $L^p, W^{k,p}, H^k$ for
Lebesgue and Sobolev spaces, and $\ell^p$ for sequence spaces on
$\Z^d$ or subsets thereof.

The closed ball with radius $r$ and center $x$ is denoted by
$B_r(x)$. Further, we set $B_r := B_r(0)$.

\subsection{Lattice functions and function spaces}
For $d \in \{2, 3\}, m \in \{1,2,3\}$, we denote the set of
vector-valued lattice functions by
\begin{displaymath}
  \Us := \Us(\Z^d)^m := \b\{ v : \Z^d \to \R^m \b\}.
\end{displaymath}
We interpret the lattice $\Z^d$ as the vertex set of a simplicial grid
$\TT$, as follows:
\begin{itemize}
\item in 2D, $\TT = \{ \xi + \hat{T}, \xi - \hat{T} \sep \xi \in \Z^2
  \}$ where $\hat{T} = {\rm conv}\{ 0, e_1, e_2\}$;
\item in 3D, $\TT = \{ \xi + \hat{T}_j \sep \xi \in \Z^3, j = 1,
  \dots, 6 \}$, where $\hat{T}_1, \dots, \hat{T}_6$ subdivide the cube
  $[0,1]^3$ as displayed in Figure \ref{fig:tet}.
\end{itemize}
\begin{figure}
  \begin{center}
    \includegraphics[height=4cm]{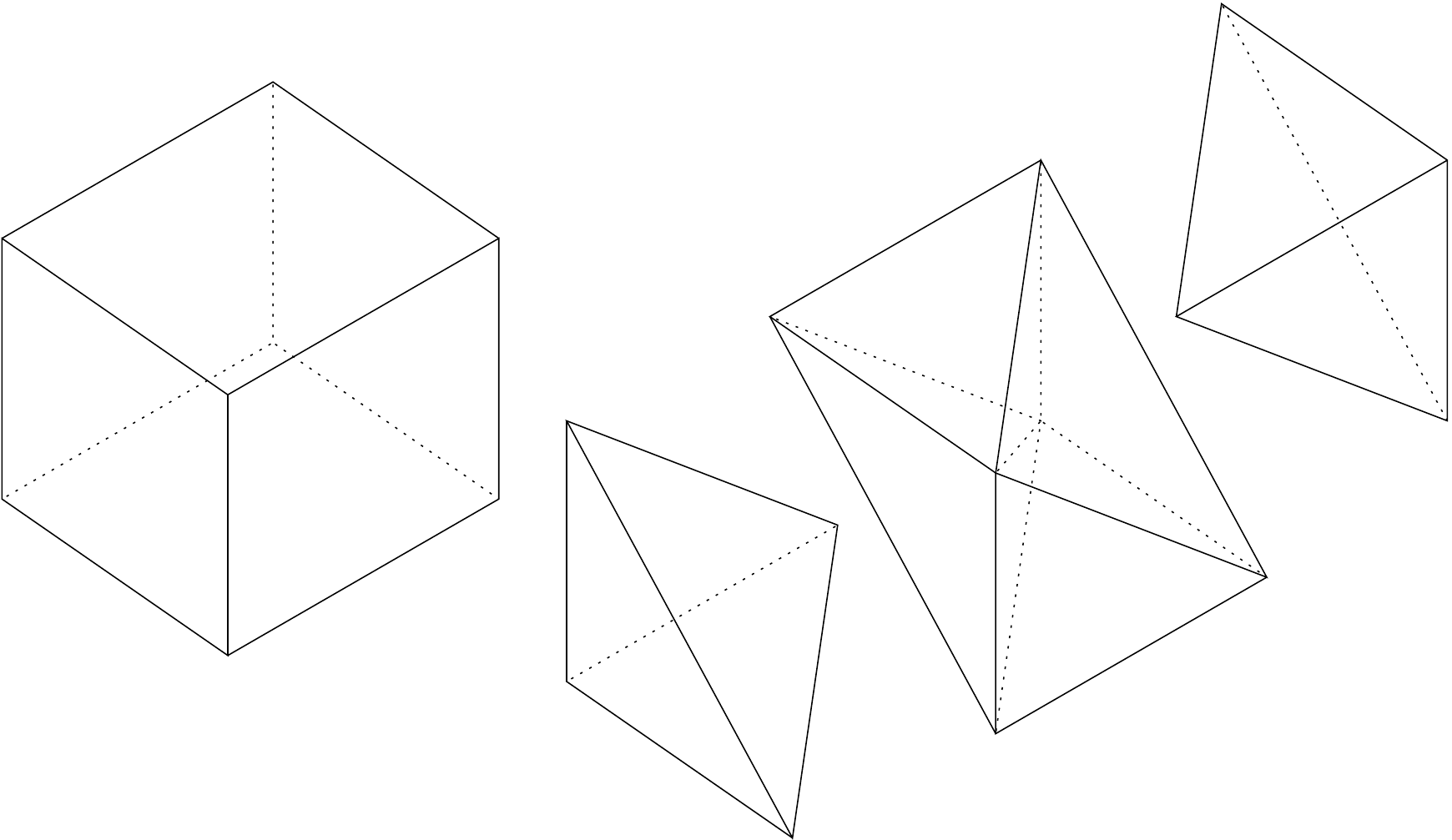}
    \caption{\label{fig:tet} Subdivision of the cube $[0,1]^3$ into 6
      tetrahedra $\hat{T}_1, \dots, \hat{T}_6$, so that the resulting
      partition $\TT$ is invariant under reflection about any lattice
      point $\xi \in \Z^3$.}
  \end{center}
\end{figure}

Let $\zz \in \WW^{1,\infty}(\R^d; \R)$ be the P1 nodal basis function
associated with the origin; that is $\zz$ is continuous and piecewise
affine with respect to $\TT$, $\zz(0) = 1$ and $\zz(\xi) = 0$
otherwise. We can then write the nodal interpolant as
\begin{equation}
  \label{eq:interp:S1_interp}
  \barv(x) := \sum_{\xi \in \Z^d} v(\xi) \zz(x - \xi), \qquad
  \text{for } v \in \Us.
\end{equation}
Clearly, $\barv \in W^{1,\infty}_{\rm loc}(\R^d)$ and $\barv(\xi) =
v(\xi)$ for all $\xi \in \L$.

Using the previous definition, we introduce the discrete homogeneous
Sobolev spaces
\begin{displaymath}
  \Us^{1,p} := \b\{ u \in \Us \bsep \D \baru \in \LL^p \b\}, \qquad
  \text{for } p \in [1,\infty],
\end{displaymath}
and the associated semi-norms $|u|_{\Us^{1,p}} := \| \D \baru
\|_{\LL^p}$.  This semi-norm fails to be a norm since it does not
penalize translations, but this issue will not enter our analysis.
For $p \in [1, \infty)$, the space of compact displacements,
\begin{displaymath}
  \Usc := \{u \in \Us \sep {\rm supp}(u) \mbox{ is compact}\}
\end{displaymath}
is dense in $\Us^{1,p}$ in the sense that, for each $u \in \Us^{1,p}$
there exists $u_n \in \Usc$ such that $\D \baru_n \to \D \baru$
strongly in $L^p$. \cite[Prop.~9]{OrtShap:Interp}.

\subsubsection{Smooth interpolant}
Since we will be primarily interested in approximation results, we
require some information about the {\em regularity} of lattice
functions.  Higher-order finite differences, a natural measure of
local smoothness of lattice functions, are cumbersome for our
analysis, hence we introduce a $C^{2,1}$-conforming multi-quintic
interpolant whose derivatives will provide equivalent information.  To
construct it we define the second-order nearest-neighbour finite
differences
\begin{align*}
  D_{i}^{\nn,0} u(\xi) &:= u(\xi), \\
  D_{i}^{\nn,1} u(\xi) &:= \smfrac12
    \b(u(\xi+e_i) - u(\xi-e_i)\b), \\
  D_{i}^{\nn,2} u(\xi) &:=   u(\xi+e_i) -
    2 u(\xi) + u(\xi-e_i),
\end{align*}
for $\xi \in \L, i \in \{1, \dots, d\}$.  For a multi-index
$\alpha\in\Z^d$, $|\alpha|_\infty\leq2, \alpha_i \geq 0,$ we define
\begin{displaymath}
  D_{\alpha}^\nn u(\xi) := D_{1}^{\nn,\alpha_1} \cdots D_{d}^{\nn,\alpha_d} u(\xi),
\end{displaymath}

The smooth interpolants are now defined through the following
lemma. Closely related and in some respects stronger results can be
found in \cite{Buffa_HPK_refinement2011, Shwab1998_spectral_book}, but
not of the specificity that we require (in particular not for $d =
3$).

\begin{lemma}
  \label{th:prelims:defn_smoothint}
  (a) For each $u \in \Us$ there exists a unique $\tilu \in
  C^{2,1}(\R^d;\R^m)$ such that $\tilu \in Q_5(\xi+(0, 1)^d)$ for all
  $\xi \in \L$ and
  $
    \partial_\alpha\tilu(\xi) = D_\alpha^\nn u(\xi)
  $
  for $\alpha\in\Z_+^d$, $|\alpha|_\infty \leq 2$, $\xi \in \L$.

  (b) Moreover, there exists a universal constant $C$ such that, for
  $p \in [1, \infty]$, $0 \leq j \leq 3$,
  \begin{equation}\label{eq:prelims:defn_smoothint}
    \| \nabla^j \tilu \|_{L^p(\xi+(0,1)^d)} \leq C \| D^j u
    \|_{\ell^p(\xi + \{-1,0,1,2\}^d)}.
  \end{equation}
  In particular, it follows that $\| \D \tilu \|_{L^p} \lesssim \| \D
  \baru \|_{L^p}$, where $D$ is the collection of first-order finite
  differences defined in \eqref{eq:defn_D}.
\end{lemma}
\begin{proof}
  The proof is given in \S~\ref{sec:analysis_smoothint}.
\end{proof}

\subsection{The atomistic model}
\label{sec: atomistic model}
We review an atomistic model from \cite{EhrOrtSha:defects} for a
defect in a homogeneous crystalline environment, which will form the
``exact problem'' that we will subsequently aim to approximate using
atomistic/continuum blending schemes.

We will consider atomistic models for two classes of crystallographic
defects: point defects and screw dislocations.


\subsubsection{Far-field boundary condition}
We fix domain and range dimensions $d \in \{2,3\}, m \in \{1,2,3\}$.
We call $\Z^d$ the \emph{reference configuration} and, with some abuse
of terminology, a map $y \in \Us$ a {\em deformed configuration} or
{\em deformation}. For example, if $d = m = 3$, then $y(\xi)$ is the
position of atom $\xi$.

We shall impose a {\em far-field boundary condition} $y(\xi) \sim
y_0(\xi)$ as $|\xi| \to \infty$, by specifying a {\em reference
  deformation} $y_0 : \R^d \to \R^m$ and admitting only deformations
from the space
\begin{displaymath}
  \Y := \b\{y \in \Us \bsep y = y_0 + u
  \mbox{ for some } u \in \Us^{1,2} \b\}.
\end{displaymath}
We explain how to choose $y_0$ to model various types of defects in
\S~\ref{sec:model:point} and \S~\ref{sec:model:screw} below. It will
later become important that $y_0$ is defined on all of $\R^d$.

For future reference, we extend the definition of the two lattice
interpolants as follows:
\begin{equation}
  \label{eq:prelims_defn_Iy}
  \bary := y_0 + \baru \quad \text{and} \quad
  \tily := y_0 + \tilu.
\end{equation}
(Strictly speaking, this represents a clash of notation.  However,
henceforth we will always apply the smooth interpolant to elements of
$\Ys$ or $\Us^{1,2}$ and therefore adopt the latest definition
\eqref{eq:prelims_defn_Iy}.)

\begin{remark}
  \label{rem:justification_ffbc}
  To justify how we impose the far-field boundary condition we note
  that, in all our model problems we will have that $y_0(\xi)$ scales
  linearly as $|\xi| \to \infty$, while $u \in \Us^{1,2}$ implies that
  $|u(\xi)| = o(|\xi|)$ \cite[Prop.~12]{OrtShap:Interp}. Thus, we have
  that $y(\xi) \sim y_0(\xi) + o(|y_0(\xi)|)$ as $|\xi| \to \infty$.

  The choice of the $\Us^{1,2}$ space for the relative displacements
  $u$ is due to the fact that these are precisely the ``finite-energy
  displacements''.
\end{remark}

\subsubsection{Energy difference functional}
\def\rcut{r_{\rm cut}}
We now define an energy (difference) functional on the space of
deformations.  First, we choose a finite {\em interaction range} $\Rg
\subset B_{\rcut} \cap \Z^d \setminus \{0\}$, where $\rcut > 0$ is a
{\em cut-off radius}, and we define the finite difference operator and
finite difference stencil
\begin{equation}
  \label{eq:defn_D}
  \begin{split}
    \Da{\rho} v(\xi) :=~& v(\xi+\rho) - v(\xi), \qquad \text{for } v \in
    \Us,  \quad \xi, \rho \in
    \Z^d, \quad \text{and} \\
    D v(\xi) :=~& \b( \Da{\rho} v(\xi) \b)_{\rho \in \Rg}, \qquad
    \text{for } \xi \in \Z^d.
  \end{split}
\end{equation}
We additionally make the technical assumption, without restriction of
generality, that $e_i\in\Rg$ for $i=1,\ldots,d$. Then, for $y \in \Y$,
we define an \emph{atomistic energy difference functional} of the form
\begin{equation}
  \label{def: energy difference functional}
  \Ea(y) := \sum_{\xi \in \Z^d} V(D y(\xi)) - V(D y_0(\xi)),
\end{equation}
where $V \in C^4( (\R^m)^\Rg)$ is a \emph{site potential}. If $y - y_0
\in \Usc$, then $\Ea(y)$ is well-defined, and we will show in Lemma
\ref{th:atmE_welldefined} (see also \S\ref{sec:model:point} and
\S\ref{sec:model:screw}) that, under natural conditions on $y_0$, $\E$
can be extended to $y \in \Ys$.

\def\domV{(\R^m)^\Rg}
\def\domEa{{\rm dom}(\Ea)}
\def\vsig{\varsigma}
\def\bfrho{{\bm \rho}}

We denote the partial derivatives of $V$ at a stencil $\bfg \in \domV$
by
\begin{align*}
  V_{,\rho}(\bfg) := \frac{\pp V(\bfg)}{\pp g_\rho} \in \R^m, \qquad
  V_{,\rho\vsig}(\bfg) := \frac{\pp^2 V(\bfg)}{\pp g_\rho \pp
    g_\vsig} \in \R^{m \times m},
\end{align*}
and so forth. For $\bfrho \in \Rg^j$ we also write
$V_{,\rho_1\cdots\rho_j} = V_{,\bfrho} \in \R^{m \times \cdots \times
  m}$. The first and second variations of $\E$, for test functions $v,
w \in \Usc$, and writing $V_{\xi,\bfrho} \equiv V_{,\bfrho}(Dy(\xi))$,
are given by
\begin{align*}
  \< \del\Ea(y), v \> &= \sum_{\xi \in \L} \< \del V(Dy(\xi)), Dv(\xi)
  \> = \sum_{\xi \in \L} \sum_{\rho \in \Rg} V_{\xi,\rho} \cdot
  D_\rho v(\xi), \quad \text{and} \\
  \< \ddel\Ea(y) v, w \> &= \sum_{\xi \in \L} \b\< \ddel V(Dy(\xi))
  Dv(\xi), Dw(\xi) \b\> = \sum_{\xi \in \L} \sum_{\rho,\vsig \in \Rg}
  D_\rho v(\xi) \cdot \b( V_{\xi,\rho\vsig} D_\vsig w(\xi) \b).
\end{align*}

We require throughout that $\Rg$ and $V$ are {\em point-symmetric}:
$-\Rg = \Rg$, and if $\bfg \in (\R^m)^\Rg$ and $\bfh =
(-g_{-\rho})_{\rho \in \Rg}$, then $V(\bfg) = V(\bfh)$. In particular,
this requirement implies that
\begin{equation}
  \label{eq:point_symmetry}
  V_{,-\bfrho}(\mF \Rg) = (-1)^{j} V_{,\bfrho}(\mF\Rg) \quad
  \text{for } \bfrho \in \Rg^j, j \geq 1, \quad \mF \in \R^{m \times d}.
\end{equation}

\begin{lemma}
  \label{th:atmE_welldefined}
  Suppose that $Dy_0 \in \ell^\infty(\L; (\R^m)^\Rg)$ and $\del\Ea(y_0)
  \in (\Us^{1,2})^*$, that is, $\< \del\Ea(y_0), v \> \leq c \| \D
  \barv \|_{L^2}$ for all $v \in \Usc$, then there exists a unique
  continuous and translation invariant extension of $u \mapsto \Ea(y_0
  +u), u \in \Usc$ to $u \in \Us^{1,2}$. The extended functional is
  four times continuously Fr\'echet differentiable in $\Us^{1,2}$.
\end{lemma}
\begin{proof}
  This result is a simplified variant of
  \cite[Thm. 2.3]{EhrOrtSha:defectsV1} or
  \cite[Thm. 2.8]{2012-ARMA-cb}.
\end{proof}

We now specify further details of the atomistic model for two
interesting situations: point defects and screw dislocations.

\subsubsection{Model for point defects}
\label{sec:model:point}
Strictly speaking, point defects occur only in 3D models, however we
also admit 2D toy models. Moreover, some combinations of topological
defects such as infinite vacancy-type dislocation loops or dislocation
dipoles with small separation distance may occasionally also be
treated as point defects, at least from an analytical perspective.

Thus, we admit $d \in \{2,3\}, m \in \{1,2,3\}$. We choose a
\emph{macroscopic strain} $\mA \in \R^{m \times d}$, non-singular, and
the far-field boundary condition $y_0(x) := \mA x$. (The matrix $\mA$
encodes the lattice structure, say $\mB \Z^d$, as well as an applied
macroscopic deformation $x \mapsto \mF x$; in this case $\mA = \mF
\mB$.)

Some point defects, such as Frenkel pairs, dislocation dipoles, can be
modeled as local (but not global) minimisers of $\E$. Other types of
point defects, such as vacancies, interstitials and impurities, can be
modeled (to some extent) by adding an external {\em defect potential}
$\Pa \in C^4(\Ys)$ to the total energy (see
\cite{EhrOrtSha:defectsV1}). We shall assume throughout that
\begin{enumerate}
\item[(A.P1)] $\Pa$ is localised: there exists $R_\Pa >0$ so that $\Pa$
  depends only on $(y(\xi); |\xi| \leq R_\Pa)$.
\item[(A.P2)] $\Pa$ is translation invariant: $\Pa(y) = \Pa(y+c)$, where
  $c(\xi) = c \in \R$.
\end{enumerate}
The total energy for point defects is then given by
\begin{displaymath}
  y \mapsto \Ea(y) + \Pa(y).
\end{displaymath}

\begin{remark}
  For slightly more complex defect geometries, such as multiple
  interstitials, it is convenient to augment the reference
  configuration, $\Z^d$, by a finite number of points. Conceptually,
  our analysis is easy to extend to such cases, but we keep our
  simplifying assumptions for the sake of a convenient notation. We
  refer to \cite{EhrOrtSha:defectsV1} for details of the ideas
  required to carry out this extension.
\end{remark}

\subsubsection{Model for screw dislocations}
\label{sec:model:screw}
\def\ulin{u^{\rm lin}}
Consider a straight screw dislocation in a Bravais lattice $\mB \Z^3$,
with Burgers vector $b \in \mB \Z^3$.  By rotating and dilating $\mB
\Z^3$, we may assume without loss of generality that $b = |b| e_3$ and
that $e_3$ is the shortest vector belonging to $\mB \Z^3$ which is
parallel to $b$. We assume, without loss of generality, that $|b| =
1$, i.e., $b = e_3$. In \cite{HudsonOrtner:disloc,
  EhrOrtSha:defectsV1} it is shown that a straight screw dislocation
can be modeled by an energy of the form \eqref{def: energy difference
  functional} with $m=3$ and $d=2$ and a reference deformation $y_0$
given by a linearised elasticity model. We briefly summarize the
construction:

We seek a reference deformation of the form $y_0(x) = \mA x +
\ulin(x)$, where $\mA \in \R^{3 \times 2}$, full rank. The matrix
$\mA$ incorporates the underlying lattice structure and any applied
macroscopic in- and anti-plane deformation, while $\ulin$ is the
displacement map according to linearised Cauchy--Born elasticity: Let
$W: \Real^{3 \times 2} \rightarrow \Real \cup \{\infty\}$ be the
\emph{Cauchy--Born strain energy density} defined by $W(\mF) =
V(\mF\Rg)$ (see \S~\ref{sec:prelims:cb} for more details), and let
$\bbC := \partial^2 W(\mA) \in \R^{3 \times 2 \times 3 \times 2}$ be
the corresponding linearised elasticity tensor. Then we require that
$\ulin \in C^\infty(\R^2 \setminus \Gamma; \R^3)$, where $\Gamma := \{
(x_1, 0) \sep x_1 \geq 0 \}$ is the ``glide plane'', and solves
\begin{equation}
  \label{eq:disl:eqn_ulin}
  \sum_{j = 1}^3 \sum_{\alpha,\beta = 1}^2
  \bbC_{i\alpha}^{j\beta} \partial_{x_\alpha} \partial_{x_\beta}
  \ulin_j(x) = 0 \quad \mbox{ for all } x \in \Real^2 \setminus \Gamma.
\end{equation}
In addition $\ulin$ must have Burgers vector $b$; that is, we require
\begin{equation}
  \label{eq:disl:bvec_ulin}
  y_0(x_1, 0-) - y_0(x_1, 0+) = \ulin(x_1, 0-) - \ulin(x_1, 0+) = b
  \qquad \text{for all } x_1 > 0,
\end{equation}
or in other words, $\int_C \nabla \ulin \cdot \dx = b$ for any closed
path $C$ winding once around $0$ in $\Real^2$.

In \cite[Sec. 12-3]{HirthLothe} and in
\cite[Sec. 2.4]{EhrOrtSha:defectsV1} it is shown that, if the
deformation $\mA x$ is strongly stable, i.e., there exists $c_0 > 0$
such that
\begin{equation}
  \label{eq:disl:stab_latt}
  \< \ddel \Ea(\mA x) v, v \> \geq c_0 \| \D \barv \|_{L^2}^2 \qquad
  \forall v \in \Usc,
\end{equation}
then a solution $\ulin \in C^\infty(\R^2 \setminus \{0\}; \R^3)$
satisfying \eqref{eq:disl:eqn_ulin} and \eqref{eq:disl:bvec_ulin}
exists, and moreover, that $\D \ulin \in C^\infty(\R^2 \setminus
\{0\}; \R^{3 \times 2})$ with
\begin{equation}
  \label{eq:regularity_ulin}
  |\D^j \ulin(x)| \leq C_j |x|^{-j} \qquad \text{ for }
  j \geq 1.
\end{equation}

In addition to the assumptions on $V$ made in
\S~\ref{th:atmE_welldefined} we require invariance under lattice slip
by a Burgers vector:
\begin{enumerate}
\item[(A.Vper) \hspace{-7mm}] \hspace{5mm} $V$ is periodic in the
  direction of $b$; that is, if $\bfg, \bfh \in (\R^3)^\Rg$ and
  $g_\rho - h_\rho \in b\Z$ for all $\rho \in \Rg$, then $V(\bfg) =
  V(\bfh)$.
\end{enumerate}

\begin{remark}
  1. Our assumptions on $y_0$ and $V$ are compatible with projecting
  a full 3D model; see \cite[Sec. 2.4]{EhrOrtSha:defectsV1} for the
  details.

  2. One may also formulate an anti-plane model. In this case, we set
  $m = 1$, $W : \R^d \to \R$ and $\ulin$ now solves a scaler elliptic
  equation; again see \cite{EhrOrtSha:defectsV1} for the details.
\end{remark}

\subsubsection{The atomistic variational problem}
\label{sec:atomistic_problem}
\def\ya{y^\a}
\def\barya{\bary^\a}
\def\tilya{\tily^\a}
\def\ua{u^\a}
\def\barua{\baru^\a}
\def\tilua{\tilu^\a}
\def\ca{\gamma^\a}
Throughout the remainder of the paper we assume that all assumptions
stated in \S~\ref{th:atmE_welldefined} hold. Moreover, we make one of
the following two sets of standing assumptions:
\begin{itemize}
\item[(pPt)] {\it Point defect problem: } $y_0 = \mA x$ for some $\mA$
  such that lattice stability \eqref{eq:disl:stab_latt} holds, and
  assumptions (A.P1), (A.P2) are satisfied.
\item[(pDs)] {\it Screw dislocation problem: } $y_0$ is given by
  \eqref{eq:disl:eqn_ulin}, \eqref{eq:disl:bvec_ulin} where $\mA$ is
  such that lattice stability \eqref{eq:disl:stab_latt} holds, and in
  addition assumption (A.Vper) is satisfied. We set $\Pa \equiv 0$.
\end{itemize}
Unless an argument applies equally to both cases (usually this is the
case), or it is clear from the context which of the two problems we
are considering, then we will always specify which set of assumptions
are are employing.

In either case, we seek to compute
\begin{equation}
  \label{eq:min_atm_exact}
  \ya \in \arg\min \b\{ \Ea(y) + \Pa(y) \bsep y \in \Ys \b\},
\end{equation}
in the sense of local minimality with respect to the metric ${\rm
  dist}(y, z) = \| \D \bary - \D\barz\|_{L^2}$.

As usual, we shall require stronger assumptions on the solution than
mere local minimality. Namely, we assume that $y^\a$ is a {\em
  strongly stable equilibrium}, by which we mean that there exists
$\ca > 0$ such that
\begin{equation}
  \label{eq:strong_stab_eq}
  \begin{split}
    \b\< \del \Ea(\ya) + \del\Pa(\ya), v \b\> &= 0 \qquad \forall v \in
    \Usc, \qquad \text{and} \\
    \b\< [\ddel\Ea(\ya) + \ddel\Pa(\ya)] v, v \b\> &\geq \ca \| \D \barv
    \|_{L^2}^2 \qquad \forall v \in \Usc.
  \end{split}
\end{equation}

The existence of a strongly stable equilibrium is a property of the
lattice and the interatomic potential (possibly even of the physical
material). Except in some special circumstances (e.g., when the
perturbation $\Pa$ is ``small'') it is difficult to establish under
the generic assumptions we are making.

However, given the existence of a strongly stable equilibrium, we can
estimate its {\em regularity} away from the defect core.

\begin{lemma}
  \label{th:regularity}
  Let either (pPt) or (pDs) be satisfied and let $\ya = y_0 + \ua$,
  $\ua \in \Us^{1,2}$, be a strongly stable equilibrium. Then, there
  exists $c > 0$ such that, for $j = 1, 2, 3$, and for a.e. $x$, $|x|
  \geq 2$,
  \begin{equation}
    \label{eq:regularity_ua}
    \b|\D^j \tilu^\a(x) \b| \leq \cases{
      c |x|^{1-d-j}, & \text{ case (pPt), } \\
      c |x|^{-j-1} \log |x|, & \text{ case (pDs), $d = 2$.}
    }
  \end{equation}
\end{lemma}
\begin{proof}
  The proof is a straightforward corollary of
  \cite[Thm. 3.1]{EhrOrtSha:defectsV1}.
\end{proof}

\subsection{The Cauchy--Born model}
\label{sec:prelims:cb}
The final concept we need to introduce before formulating a/c coupling
schemes is the Cauchy--Born model. The idea, briefly, is that if $y$
varies slowly then $D_\rho y(\ell) \approx \D_\rho \tily(\ell)$ and
hence $V(Dy(\ell)) \approx W(\D \tily(\ell))$, where the map $W :
\R^{m \times d} \to \R \cup \{+\infty\}$, $W(\mF) := V(\mF \Rg)$, is
called the {\em Cauchy--Born strain energy function}. In the absence
of defects, it is therefore reasonable to approximate the sum of site
energies with an integral over the energy density,
\begin{equation}
  \label{eq:cb_first}
  y \mapsto \int_{\R^d} \B( W(\D y) - W(\D y_0) \B) \dx.
\end{equation}
This model has been analyzed in considerable detail, e.g., in
\cite{BLBL:arma2002,E:2007a,2013:MakrSuli,2012-ARMA-cb}. Subject to
suitable technical conditions the results in these references
demonstrate that, if $y^\a$ is a ``sufficiently smooth'' stable
equilibrium of $\Ea$, then there exists a stable equilibrium $y^\c$ of
\eqref{eq:cb_first} such that
\begin{displaymath}
  \| \D y^\c - \D \barya \b\|_{L^2} \lesssim \| \D^3 \tily^\a \|_{L^2} +
  \| \D^2 \tily^\a \|_{L^4}^2.
\end{displaymath}
That is, the Cauchy--Born model is {\em second-order accurate.}

\section{Main Results}

\subsection{Formulation of the B-QCE and B-QCF methods}
We wish to approximate the atomistic model using a hybrid
atomistic/continuum description. The approximation is achieved in
three steps: 1. We replace the infinite domain with the finite
computational domain $\Omh$. 2. In those parts of $\Omh$ where the
Cauchy--Born approximation has sufficient accuracy we replace the
atomistic model with the Cauchy--Born model. 3. We restrict
deformations to a coarse-grained finite element space.

The key ingredient in this process is the coupling between the
atomistic and continuum models, which we achieve using a blending
formulation.

\subsubsection{Coarse-grained function spaces}
\label{sec:results:coarsening}
\def\Ri{R^{\rm i}}
\def\Ro{R^{\rm o}}
Let $\Omh$ be a polygonal (if $d = 2$) or polyhedral (if $d = 3$)
domain in $\R^d$. Let $\Ri > 0$ be maximal and $\Ro > 0$ be minimal
such that $B_{\Ri} \subset \Omh \subset B_{\Ro}$.

Let $\Th$ be a regular partition of $\Omh$ into closed triangles or
tetrahedra. For $T \in \Th$, let $h_T := {\rm diam}(T)$ and $r_T$ the
diameter of the largest ball contained in $T$. For $x \in \Omh$, let
$h(x) := \max_{T \in \Th, x \in T} h_T$. The associated space of P1
finite element functions is denoted by ${\rm P1}(\Th)$.  If $\Nh$
denotes the set of finite element nodes, then the nodal interpolant of
a function $v : \Nh \to \R^k$ is the unique function $I_h v \in {\rm
  P1}(\Th)$ such that $I_h v = v$ on $\Nh$.

For a function $v : \bigcup_{T \in \Th} {\rm int}(T) \to \R^k, k \in
\N$ let $Q_h v \in {\rm P0}(\Th)$ denote the piecewise constant
mid-point interpolant, $Q_h v(x) := v(x_T)$ for $x \in T\in \Th$,
where $x_T := \mint_T x \dx$.

Exploiting the structure $y = y_0 + u, u \in \Us^{1,2}$ of admissible
deformations, we define the coarse-grained displacement and
deformation spaces, respectively, by
\begin{align*}
  \Us_h &:= \b\{ u_h \in C(\R^d; \R^m) \bsep u_h|_{\Omh} \in {\rm
    P1}(\T),  u_h|_{\R^d \setminus \Omh} = 0 \b\} \quad \text{and} \\
  \Ys_h &:= \b\{ y_h = y_0 + u_h \bsep u_h \in \Us_h \b\}.
\end{align*}

\subsubsection{The B-QCE method}
\label{sec:defn_bqce}
Let $\beta \in C^{2,1}(\R^d)$ be a {\em blending function} then the
B-QCE energy difference functional is defined by
\def\Ebh{\mathscr{E}^{\bl}_h}
\def\domEbh{{\rm dom}(\Ebh)}
\begin{equation}
  \label{eq:def_Ebh}
\begin{split}
  \Ebh(y_h) &:= \sum_{\xi \in \L} \b(1-\beta(\xi)\b) \B(V\b(D
  y_h(\xi)\b) - V\b(Dy_0(\xi)\b)\B) \\
  & \qquad \qquad + \int_\Omh Q_h\B[ \beta \cdot \b(W(\D y_h) - W(\D y_0)\b) \B] \dx,
  \qquad \text{for } y_h \in \Ys_h.
\end{split}
\end{equation}
We assume that $1-\beta$ has compact support, hence the lattice sum is
finite, while the integral is taken over a finite domain; thus $\Ebh$
is well-defined. The application of the mid-point quadrature rule to
evaluate the integral makes \eqref{eq:def_Ebh} fully computable.

\def\bqce{{\rm bqce}}
In the B-QCE method we approximate the atomistic variational problem
\eqref{eq:min_atm_exact} with
\begin{equation}
  \label{eq:min_bqce}
  y_h^\bqce \in \arg\min \b\{ \Ebh(y_h) + \Pa(y_h) \bsep y_h \in \Ys_h \b\}.
\end{equation}

The B-QCE method, as we formulated it, was introduced for
one-dimensional lattices in~\cite{VkLusk2011:blended}, and was later
extended to two and three-dimensions in~\cite{2012-CMAME-optbqce} in a
formulation which differs only marginally from the one given in
\eqref{eq:def_Ebh}: in \cite{2012-CMAME-optbqce} the operator $Q_h$
defined a trapezoidal rule instead of a midpoint rule. As a matter of
fact, all of our results can be adapted to this case.

B-QCE shares many features with the bridging domain
method~\cite{XiBe:2004}, the Arlequin
method~\cite{bauman:applicationofArlequin}, and the AtC
coupling~\cite{BadiaParksBochevGunzburgerLehoucq:2007}.  The bridging
domain method and the Arlequin method differ from B-QCE primarily in
that they couple the atomistic and continuum degrees of freedom weakly
using Lagrange multipliers.  The AtC coupling is a very general
formulation which includes B-QCE and many other methods as special
cases.

\subsubsection{The B-QCF method}
\label{sec:def_bqcf}
While the B-QCE method blends atomistic and continuum energies the
B-QCF method blends atomistic and continuum forces. We first define
the Cauchy--Born finite element functional
\begin{equation}
  \label{eq:defn_cb_Yh}
  \Ech(y_h) := \int_{\Omh} Q_h \b[ W(\D y_h) - W(\D y_0) \b] \dx,
  \qquad \text{ for } y_h \in \Ys_h.
\end{equation}

\def\Fbh{\mathscr{F}^{\beta}_h}
\def\domFbh{{\rm dom}(\Fbh)}
Assume again that $\beta \in C^{2,1}(\R^2)$ is a blending function,
then the B-QCF operator is the nonlinear map $\Fbh : \Ys_h \to
\Us_h^*$, defined by
\begin{equation}
  \label{eq:defn_Fbh}
  \b\< \Fbh(y_h), v_h \b\> := \b\< \del\Ea(y_h), (1-\beta) v_h \b\> +
  \b\< \del\Ech(y_h), I_h[\beta v_h] \b\>,
\end{equation}
where $(1-\beta) v_h$ and $\beta v_h$ are defined in terms of
pointwise multiplication. $\Fbh$ is well-defined since $y_h$ and $v_h$
are defined as functions on all of $\R^d$ and $v_h$ has compact
support.

In the B-QCF method we approximate the atomistic variational problem
\eqref{eq:min_atm_exact} with the variational nonlinear system
\def\bqcf{{\rm bqcf}}
\begin{equation}
  \label{eq:eqn_bqcf}
  \b\< \Fbh\b(y_h^\bqcf\b) + \del\Pa\b(y_h^\bqcf\b), v_h \b\> = 0 \qquad
  \forall v_h \in \Us_h.
\end{equation}

\begin{remark}
  Suppose we define a blended a/c force via
  \begin{displaymath}
    F_\nu(y_h) := (1-\beta(\nu)) \frac{\partial \Ea(y_h)}{\partial y_h(\nu)}
    + \beta(\nu) \frac{\partial \Ech(y_h)}{\partial y_h(\nu)}
    \qquad \text{for } \nu \in \Nh \setminus \partial\Omh, \quad y_h \in \Ys_h,
  \end{displaymath}
  then $- \sum_{\nu \in \Nh \setminus \partial\Omh} F_\nu(y_h)
  v_h(\nu) = \< \Fbh(y_h), v_h \>$.  Thus, the nonlinear system
  $F_\nu(y_h^\bqcf) + \partial_{y_h(\nu)} \Pa(y_h^\bqcf) = 0$, $\nu
  \in \Nh \setminus \partial\Omh$, is equivalent to the variational
  form \eqref{eq:eqn_bqcf}.
\end{remark}

\bigskip

The B-QCF method \eqref{eq:eqn_bqcf} is essentially the same method as
those proposed in \cite{Lu.bqcf:2011, BQCFcomp}. It also has many
parallels with methods formulated in
\cite{BadiaParksBochevGunzburgerLehoucq:2007,bauman:applicationofArlequin,xiao:bridgingdomain,fish:concurrentAtCcoupling,bridging,prudhomme:modelingerrorArlequin,seleson:bridgingmethods,XiBe:2004}.

Both in \cite{Lu.bqcf:2011} and \cite{BQCFcomp} the main motivation of
force-blending was that stability of the scheme can be proben, while
the stability of sharp-interface force-based a/c couplings is entirely
open at this point \cite{doblusort:qcf.stab,qcf.stab,DobShapOrt:2011,
  LuMing12}

\subsection{Approximation Error Estimates}
\label{sec:approx_error}
\def\La{\Lambda^\a}
\def\Omo{\Omega^{\rm ext}}
\def\Oma{\Omega^\a}
\def\Rb{R^\beta}
\def\Ra{R^\a}
To formulate our approximation results, and for the subsequent
analysis, we require additional assumptions on the computational
domain and the mesh. See Figure \ref{fig:radii} for a visualisation of
the following definitions.

\begin{figure}
  \begin{center}
    \includegraphics[width=8cm]{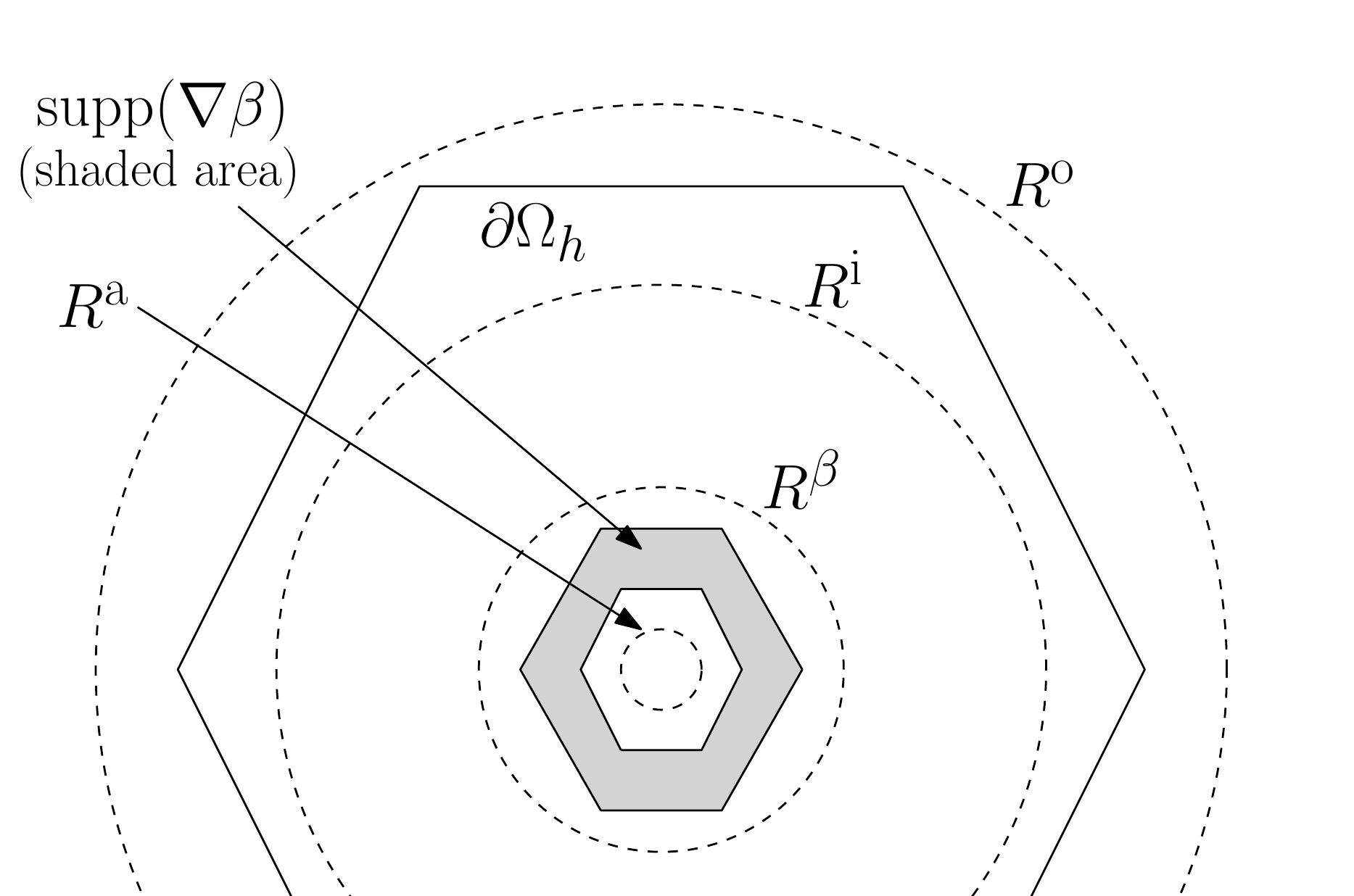}
    \caption{\label{fig:radii} Visualisation of the definitions and
      assumptions made in \S~\ref{sec:approx_error}.}
  \end{center}
\end{figure}

In addition to the radii $R_\Pa, \Ri, \Ro$ defined in
\S~\ref{sec:model:point} and in \S~\ref{sec:results:coarsening}, we
define $\Ra$ to be the largest and $R^\beta$ to be the smallest
numbers satisfying
\begin{displaymath}
  \supp(\beta) \supset B_{\Ra + 2 \rcut + \sqrt{d}}
  \qquad\text{and}\qquad
  \supp(1-\beta) \subset B_{R^\beta - 2\rcut-\sqrt{d}}.
\end{displaymath}
We specify {\em atomistic, blending, continuum and exterior regions}
\begin{align*}
   \Oma &:= {\rm supp}(1-\beta) +
B_{2\rcut+\sqrt{d}} \subset B_{\Rb}, \\
\Omb &:= {\rm supp}(\D\beta) + B_{2\rcut+\sqrt{d}} \subset B_{\Rb}
 \setminus B_{\Ra},  \\
 \Omc &:= {\rm supp}(\beta) \cap \Omh + B_{2\rcut+\sqrt{d}} \subset
 B_{\Ro} \setminus B_{\Ra}, \quad \text{and} \\
 \Omo &:= \R^2 \setminus B_{\Ri / 2}.
\end{align*}

Further, we
define {\em discrete atomistic and blending regions}
\begin{displaymath}
  \La := \L \cap ({\rm
  supp}(1-\beta)+\Rg) \quad \text{and} \quad
\Lb := \{\xi
\in \L : D\beta(\xi) \ne 0 \}.
\end{displaymath}
The fact that the various regions overlap is simply for the sake of
convenience of the analysis and notation.

We assume throughout that there exist fixed constants $C_{\T}$,
$C^\beta_1, C^\beta_2$ such that the following conditions are
satisfied:
\begin{align}
  \label{eq:blend_stab:radii}
  & R_\Pa \leq \Ra \leq R^\beta \qquad \text{and} \qquad
  R^\beta \leq C^\beta_1 \Ra;
  \\  &
  \label{eq:blend_stab:beta_bound}
  \beta\in C^{2,1}, \qquad 0\leq \beta \leq 1 \qquad \text{and} \qquad
  \|\nabla^j\beta\|_{L^\infty} \leq C^\beta_2 (\Ra)^{-j}, j = 1, 2,
  3; \\
  &\label{eq:blend_stab:tri} \text{$\Th$ is fully refined in $\Oma$}
  \qquad \text{and} \qquad \max_{T \in \Th} h_T / r_T \leq C_{\T}.
\end{align}
By \eqref{eq:blend_stab:tri} we mean that, if $T \in \T$ with $T \cap
\Oma \neq \emptyset$, then $T \in \TT$; as well as vice-versa.

In addition, only for $d = 2$ and only for the B-QCF method, we assume
that there are constants $C_{\Omega}, m_\Omega \geq 1$ such that
\begin{equation}
  \label{eq:bqcf_domain_bound}
  \Ro \leq C_\Omega (\Ra)^{m_\Omega}.
\end{equation}

The two main approximation parameters to define both the B-QCE and
B-QCF methods are the blending function $\beta$ and the finite element
mesh $\T$ (and through it, the computational domain $\Omh$). The
regions $\Oma, \Omb, \Omc, \Omo, \La, \Lb$ and the radii, $\Ra, \Rb,
\Ri, \Ro$ are derivative parameters. The constants $C_{\T}, C_j^\beta,
C_\Omega, m_\Omega$ in assumptions \eqref{eq:blend_stab:radii},
\eqref{eq:blend_stab:beta_bound}, \eqref{eq:bqcf_domain_bound} and
\eqref{eq:blend_stab:tri} are understood to be uniform in {\em all
  choices} of ($\beta$, $\T$) that may occur in our analysis.

Throughout the remainder of the paper, we will write ``$A=O(B)$'' or
``$|A|\lesssim B$'' if there exists a constant $C$ such that $|A| \leq
C B$, where $C$ is independent of the approximation parameters
$(\beta, \Th)$, but may depend on the constants $C_{\T}, C_j^\beta,
C_\Omega, m_\Omega$, or on any specified functions involved in the
estimate. (In particular, $C$ may depend on a solution $y^\a$ and on
derivatives $V_{,\bfrho}(\bfg)$ for $\bfg$ in some specified range,
cf. \S~\ref{sec:int:V_bounds}, but never on a test function.)


\subsubsection{Error estimates in terms of solution regularity}
\def\errba{{\bf E}^{\rm apx}}
\def\errcb{{\bf E}^{\rm cb}}
\def\errbqce{{\bf E}^{\rm int}}
\def\errff{{\bf E}^{\rm ff}}
For $y = y_0 + u \in \Ys$ we define the
\begin{align}
  \notag \text{best-approximation error } && \errba(y) &:= \| \D \baru
  \|_{L^2(\Omo)} + \| h \D^2
  \tilu \|_{L^2(\Omc)} + \|h^2 \D^3 \tilu \|_{L^2(\Omc)} \\
  \notag
  && & \qquad + \| h^2 \D^3 y_0 \|_{L^2(\Omc)} + \| h (\D y_0 - \mA)
  \otimes \D^2 y_0 \|_{L^2(\Omc)},  \\
  \label{eq:defn_errorterms}
  \text{Cauchy--Born model error } && \errcb(y) &:= \| \D^3 \tily
  \|_{L^2(\Omc)} + \| \D^2 \tily \|_{L^{4}(\Omc)}^2, \\
  \notag \text{and coupling error } && \errbqce(y) &:= \| \D^2 \beta
  \|_{L^2} + \| \D \beta \|_{L^\infty} \|\D^2 \tily \|_{L^2(\Omb)}.
\end{align}

The first term in $\errba$ measures the finite element coarsening
error (including the quadrature error), while the second term in
$\errba$ measures the error induced by reducing the problem to a
bounded domain.

\def\Ctr{\gamma_{\rm tr}}

\begin{theorem}
  \label{th:error}
  Let $\ya = y_0 + \ua \in \Ys$ be a strongly stable solution to
  \eqref{eq:min_atm_exact}.
  Then, there exist constants $\epsilon, \Ra_0, C > 0$, which are
  independent of $\Th$ and $\beta$,
  such that, if $\Ra \geq \Ra_0$, then there exist strongly stable
  solutions $y_h^\bqce$ to \eqref{eq:min_bqce} and $y_h^\bqcf$ to
  \eqref{eq:eqn_bqcf} satisfying
  \begin{align}
    \label{eq:bqce_general_errest}
    \| \D \barya - \D y_h^\bqce \|_{L^2} &\leq C \b( \errba(\ya) +
    \errcb(\ya) + \errbqce(\ya) \b), \qquad \text{and} \\
    \label{eq:bqcf_general_errest}
    \| \D \barya - \D y_h^\bqcf \|_{L^2} &\leq C \Ctr  \b( \errba(\ya) +
    \errcb(\ya) \b),
  \end{align}
  where $\Ctr = \sqrt{1+\log(\Ra)}$ if $d = 2$ and $\Ctr = 1$ if $d =
  3$. 
\end{theorem}

\begin{proposition}
  \label{th:bqce-energy_error}
  Under the conditions of Theorem \ref{th:error}, we have
  \begin{align}
    \label{eq:bqce-energy_error}
    \b| \Ea(y^\a) - \Ebh(y_h^\bqce) \b| &\leq C\B\{
    {\errba(y^\a)}^2 + {\errcb}(y^\a)^2 + {\errbqce(y^\a)}^2 \\
    \notag
    & \qquad +
    \| \D^2\beta \|_{L^2} \| \D\tilu^\a
    \|_{L^2(\Omc)} + \| \D\beta \|_{L^2} \| \D^2\tilu^\a \|_{L^2(\Omc)} \\
    \notag
    & \qquad \qquad \qquad
    + \errba(y^\a) \b( \| \D\tilu^\a \|_{L^2(\Omc)} + \| \D \ulin
    \|_{L^4(\Omc)}^2 \b) + H^{\rm ots} \B\},
  \end{align}
  where $C$ is independent of $\Th$ and $\beta$ and $H^{\rm ots}$ are
  ``higher order terms'' (cf. \S~\ref{sec:errest_cost}),
  \begin{align*}
   H^{\rm ots} &:=
   \| \D^3 \tilu^\a \|_{L^1(\Omc)} + \| \D^2\tilu^\a
    \|_{L^2(\Omc)}^2   \\
    \notag
    & \qquad + \B( \| h^2 \D^3 y_0 \|_{L^2(\Omc)} + \|h \D^2 y_0
    \|_{L^4(\Omc)}^2 \B) \B( \errba(y^\a) + \| \D \tilu^\a \|_{L^2(\Omc)} \B) \\
    \notag
    & \qquad + \| \D^2\tilu^\a \|_{L^2(\Omc)} \| \D^2 y_0 \|_{L^2(\Omc)}
    + \| \D\beta\|_{L^\infty} \| \D\tilu^\a \|_{L^2(\Omc)} \| \D^2 y_0 \|_{L^2(\Omc)}.
  \end{align*}
\end{proposition}

\begin{remark}
  The B-QCF error estimate seemingly has no $\beta$-dependence, but
  this is only due to the strong assumptions we made on $\beta$ in
  \eqref{eq:blend_stab:beta_bound}. Only under these assumptions are
  we able to state Theorem \ref{th:error}. However, it can be
  expected, that the result is also valid under more specialized, but
  otherwise much milder assumptions on $\beta$. In such a case, our
  intermediate results in \S~\ref{sec:intermediate} and
  \S~\ref{sec:bqcf_cons_prf} can be employed to understand the precise
  $\beta$-dependence of the error.
\end{remark}

\subsubsection{Error estimates in terms of computational cost}
\label{sec:errest_cost}
\def\dof{{\rm DOF}}
Following \cite{2012-MATHCOMP-qce.pair} we now convert the error
estimates \eqref{eq:bqce_general_errest},
\eqref{eq:bqcf_general_errest} and \eqref{eq:bqce-energy_error} into
convergence rates in terms of the number of degrees of freedom
\begin{displaymath}
  \dof := \# \Th.
\end{displaymath}
The quantity $\dof$ is directly related (but not necessarily
proportional) to the computational cost of solving the associated
problems \eqref{eq:min_bqce} and \eqref{eq:eqn_bqcf}.  The estimates
in terms of $\dof$ form the basis for the optimised implementations of
the B-QCE and B-QCF methods presented, respectively, in
\cite{2012-CMAME-optbqce, BQCFcomp}.

We introduce additional restrictions on $\Th$ and $\Omh$,
\begin{align}
  \label{eq:assumpt_meshgrad}
  |h(x)| \lesssim \max\b\{1,\smfrac{|x|}{\Rb}\b\},
  \quad \Ro \lesssim \Ri,
  \quad \text{and} \quad
  \dof
  &\lesssim (\Rb)^d \log \b(\smfrac{\Ro}{\Rb}\b)  \\
  \notag
  &\lesssim (\Ra)^d \log \Ra.
\end{align}
The second bound in \eqref{eq:assumpt_meshgrad} is a mild assumption
on the shape regularity of $\Omh$, while the last bound in
\eqref{eq:assumpt_meshgrad} is a corollary of the first one, upon
additionally requiring that \eqref{eq:bqcf_domain_bound} holds for
both B-QCE and B-QCF.

Then, using the regularity estimates \eqref{eq:regularity_ua} and
\eqref{eq:regularity_ulin} it is straightforward to prove that
\[
\begin{array}{rrl}
  \text{ Case (pPt): } &
  \errcb(y^\a) &\lesssim (\Ra)^{-d/2-2},  \\[1mm]
  & \errbqce(y^\a) &\lesssim
  (\Ra)^{d/2-2}, \\[1mm]
  & \errba(y^\a) & \lesssim (\Ra)^{-d/2-1} + (\Ri)^{-d/2},
  \\[3mm]
  \text{ Case (pDis): } & \errcb(y^\a)&\lesssim (\Ra)^{-2}, \\[1mm]
  &\errbqce(y^\a)&\lesssim (\Ra)^{-1},
  \\[1mm]
  & \errba(y^\a) &
  \lesssim (\Ra)^{-2} \log(\Ra) +
  (\Ri)^{-1}.
\end{array}
\]
(Here we used the estimate $\int_R^\infty r^t \log^s r \dr \leq
R^{t+1} \log^s(R)$ for $t < -1, s \in \N, R \geq 2$.)  We note that
the dominant term, $\errbqce(y^\a) \lesssim (\Ra)^{d/2-2}$ originates
entirely from the ``blended ghost force error'' term $\| \D^2 \beta
\|_{L^2}$.

Next, we note that in the B-QCE case any choice $\Ri \gg \Ra$ balances
the far-field contribution, $\| \D \baru^\a \|_{L^2(\Omo)} \lesssim
(\Ri)^{-d/2}$ with the ``blended ghost force error'' $\| \D^2 \beta
\|_{L^2} \lesssim (\Ra)^{d/2-2}$.

For the B-QCF case we balance the far-field error $\| \D \baru^\a
\|_{L^2(\Omo)} \lesssim (\Ri)^{-d/2}$ with the finite element
coarsening error.  Ignoring log-factors, we observe that the radius
$\Ri$ ought to be balanced against the interpolation error component
$\|h \D^2 \tilu^\a \|_{L^2(\Omc)} \lesssim (\Ra)^{-d/2-1}$, which
yields $\Ri \approx (\Ra)^{d/2+1}$ both in the (pPt) and (pDis) cases.
Hence, we obtain
\begin{align*}
  \text{Case (pPt):} \hspace{1cm}\errba(y^\a) &\lesssim (\Ra)^{-d/2-1}, \\[1mm]
  \text{Case (pDis):} \hspace{1cm} \errba(y^\a) &\lesssim
  (\Ra)^{-2} \log(\Ra).
\end{align*}

We summarise the foregoing computations in the following theorem,
using also the fact that, under the conditions of the theorem, $\Ra
\lesssim (\dof)^{1/d}$, $(\Ra)^{-1} \lesssim (\dof)^{-1/d} (\log
\dof)^{1/d}$ and $\Ctr \lesssim (\log\dof)^{1/2}$. The estimate for
the energy error can be immediately obtained from analogous
computations.

\begin{theorem}
  \label{th:errest_complexity}
  In addition to the assumptions of Theorem \ref{th:error} suppose
  that \eqref{eq:assumpt_meshgrad} holds and that $\Ri \geq c_\Omega
  (\Ra)^s$ for a constant $c_\Omega > 0$ independent of $(\beta,
  \Th)$, where $s > 1$ for the B-QCE method and $s \geq d/2+1$ for the
  B-QCF method. Then, there exists a constant $C$, independent of
  $(\beta, \Th)$, such that
  \begin{align*}
     \text{\hspace{0.8cm} for the B-QCE method, for both Cases (pPt) and (pDis), }
     \hspace{-5.5cm}  & \\
    \| \D \bary^\a - \D y_h^\bqce \|_{L^2} &\leq C
    (\dof)^{d/4-1} (\log\dof)^{-d/4+1},  \\[1mm]
    \b| \Ea(y^\a) - \Ebh(y_h^\bqce) \b| &\leq C (\dof)^{d/2-2}
    (\log\dof)^{-d/2+2}; \\[3mm]
    \text{and for the B-QCF method,} \hspace{0.1cm} & \\
    \| \D \bary^\a - \D y_h^\bqcf \|_{L^2} &\leq C \cases{
      (\dof)^{-d/4-1}  (\log\dof)^{2}, & \text{case (pPt)},  \\[1mm]
    (\dof)^{-1} (\log\dof)^{5/2}, & \text{case (pDis)}.}
  \end{align*}
\end{theorem}

\begin{remark}
  1. The construction of $\Th$ satisfying \eqref{eq:assumpt_meshgrad}
  is standard and can be found, e.g., in
  \cite{2012-MATHCOMP-qce.pair}.

  2. To construct $\beta$, we could, for example, choose $\Rb =
  C^\beta_1 \Ra$ for a given $\Ra$ and then choose $\beta$ in the form
  of a radial spline satisfying the conditions
  \eqref{eq:blend_stab:beta_bound}.  For complicated a/c interface
  geometries one could solve a bi-Laplace equation in a precomputation
  step (see \cite{2012-CMAME-optbqce}).

  3. Finally, we could allow for a stronger mesh coarseing, $h(x)
  \approx (|x|/\Rb)^\alpha$ and thereby drop the $\log$ factor in
  $\dof$ for a suitable choice of $\alpha>1$, which would slightly
  improve the estimates. In order to preserve mesh regularity
  \eqref{eq:blend_stab:tri}, one would need to impose that $h(x)
  \lesssim |x|$. Note that this does not violate any of our foregoing
  assumptions for suitable choices of $\alpha$; see \cite{olson:2013b}
  for further discussion.
\end{remark}

\subsection{Conclusion}
\label{sec:conclusion}
We have established the first error analysis of a/c coupling schemes
that is ``complete'' in the sense that it covers general interatomic
potentials, accomodates atomistic solutions containing defects, and
requires no assumption on the atomistic solution beyond its
stability.

While our results are restricted to two specific a/c coupling schemes,
we anticipate that the techniques we have developed allow extensions
to a much wider range of blending type a/c couplings. We emphasize,
however, that most of our techniques are specialised for blending type
schemes. In particular, the technique of Lemma
\ref{th:int:stab_bqce_lem2}, which is the main new technical
ingredient to prove stability of B-QCE and B-QCF, is unlikely to
generalise to sharp-interface couplings. To that end the ideas present
in \cite{LuMing12} and \cite{OrtnerShapeevZhangV1} are more promising
starting points.

We remark on a seemingly immediate extension which, surprisingly,
seems not straightforward: The main assumption among those formulated
in \S~\ref{sec:approx_error} is that the finite element mesh is fully
refined in the blending region. This is highly convenient from the
perspective of both analysis and implementation, but it is likely
that, in practice, a coarse mesh in the blending region would yield a
more efficient scheme; see, e.g., \cite{XiBe:2004}, where this is in
fact a crucial ingredient. Most of our results do not require this
restriction, but there are several steps (in particular in
\S~\ref{sec:bqce_coarse_err_est}) which appear to be more difficult
without it.

\section{Key Intermediate Results}
\label{sec:intermediate}
\def\G{\mathscr{G}}
The purpose of this section is to give a detailed overview of the main
steps and ideas employed in the proof of the main results, and to
state some key intermediate results that are of independent interest.

\subsection{Framework}
\label{sec:int:framework}
We adopt the analytical framework of \cite{2013-atc.acta}, which is
analogous to that of finite element methods for (regular) nonlinear
PDE, employing quasi-best approximation, consistency and
stability.

Briefly, let $\G_h = \del\Ebh + \del\Pa$ for the B-QCE scheme or $\G_h
= \Fbh + \del\Pa$ for the B-QCF scheme. Let $\Pi_h : \Us \to \Us_h$ be
a suitable ``quasi-best approximation operator'' (we define it in
\S~\ref{sec:res:bestapprox}), then we shall require that $\G_h$ is
{\em consistent},
\begin{equation}
  \label{eq:int:consistency}
  \b\< \G_h(\Pi_h \ya), v_h \b\> \leq \eta \|\D v_h \|_{L^2} \qquad
  \forall v_h \in \Us_h,
\end{equation}
for some ``small'' consistency error $\eta$ that depends on $\ya, \Th$
and $\beta$; and {\em stable},
\begin{equation}
  \label{eq:int:stability}
  \b\< \del\G_h(\Pi_h \ya) v_h, v_h \b\> \geq c_0 \|\D v_h \|_{L^2}^2
  \qquad \forall v_h \in \Us_h.
\end{equation}
We then empoy the Inverse Function Theorem to prove that, if $\eta /
c_0$ is sufficiently small (adding some technical assumptions), then
there exists $w_h \in \Us_h$ such that $\| \D w_h \|_{L^2} \leq 2 \eta
/ c_0$ and $\G_h(\Pi_h \ya + w_h) = 0$.

The condition that $\eta / c_0$ is sufficiently small corresponds to
the assumption that $\Ra$ is sufficiently large in Theorem
\ref{th:error}.

Thus, we have constructed a B-QC solution $y_h^{\rm bqc} := \Pi_h \ya
+ w_h$ satisfying
\begin{equation}
  \label{eq:int:abstract_errest}
  \| \D \barya - \D y_h^{\rm bqc} \|_{L^2} \leq 2 \frac{\eta}{c_0} +
  \| \D \barya - \D \Pi_h \ya \|_{L^2}.
\end{equation}
The second term on the right-hand side is the quasi-best approximation
error.

In the present section we shall make this generic outline concrete. We
shall present the key ideas in our analysis but postpone the technical
aspects of the proofs to later sections.

\subsection{Further Preliminaries}
Here we introduce additional ingredients that we require to motivate
and state the key intermediate results.

\subsubsection{Expansion of discrete strain}
\label{sec:int:strain}
Let $y \in \Ys$ be a deformation. Much of our analysis depends on
Taylor expansions of finite differences within the a neighbourhood
\begin{equation}
  \label{eq:defn_nux}
  \nu_x := B_{2\rcut+\sqrt{d}}(x)
\end{equation}
of some $x \in \R^d$, containing all those lattice points $\xi$ for
which $\Sa(y; x)$ depends on $D_\rho y(\xi), \rho \in \Rg$ ($\Sa$ is
the atomistic stress defined in \S~\ref{sec:int:stress_at_cb}) and an
additional $\sqrt{d}$ buffer, which we require in view of the
``convolution trick'' \eqref{eq:stress:5}.

\begin{lemma}
  \label{th:inter:strain_lemma}
  Let $z \in C^{2,1}(\nu_x)$ and $|x-\xi| \leq \rcut+\sqrt{d}, \rho
  \in \Rg$, then
  \begin{align}
    \label{eq:inter:strain_exp_low}
    \b| D_\rho z(\xi) - \D_\rho z(x) \b| & \leq C \| \D^2 z
    \|_{L^\infty(\nu_x)}, \\
    \label{eq:inter:strain_expansion}
    \b| D_\rho z(\xi) - \b[ \D_\rho z(x) + \D_\rho \D_{\xi-x} z(x) + \smfrac12
    \D_{\rho}^2 z(x) \b] \b| &\leq C \| \D^3 z \|_{L^\infty(\nu_x)},
  \end{align}
  where $C$ is a generic constant.
\end{lemma}
\begin{proof}
  The results are obtained by straightforward Taylor expansions about
  $x$.
\end{proof}

Normally, we would like to perform the expansions
\eqref{eq:inter:strain_exp_low}, \eqref{eq:inter:strain_expansion}
with $z = \tily$, but this is only possible if $\tily$ is smooth in
$\nu_x$, which fails in the dislocation case when $\nu_x$ intersects
the branch-cut.  To still use these Taylor expansions, we therefore
construct {\em equivalent local deformations} that are smooth in
$\nu_x$: for $x \in \R^d$, and $|x' - x| < |x|$, let
\begin{equation}
  \label{eq:defn_yx}
  y^x(x') := \int_{t = 0}^1 \D \tily\b( (1-t)x+tx'\b) (x'-x) \dt,
\end{equation}
then $y^x \in C^{2,1}$ in its domain of definition, with $\D^j y^x =
\D^j \tily$, $j \geq 1$, and $y^x - \tily \in b \Z$. The latter
property, together with (A.Vper) ensures that, for $|x| >
2\rcut+\sqrt{d}$,
\begin{equation}
  \label{eq:slip_invariance_yx}
  V_{,\bfrho}(Dy(\xi)) = V_{,\bfrho}(D y^x(\xi)) \qquad \text{ for all
  $\xi \in \L$, $|x - \xi| \leq \rcut+\sqrt{d}$.}
\end{equation}
We will employ \eqref{eq:slip_invariance_yx} in the consistency
proofs in an ad-hoc fashion whenever we need to replace a finite
difference stencil $Dy(\xi)$ with a stencil $D y^x(\xi)$ in order to
then perform a Taylor expansion.

\subsubsection{Expansion of the potential}
\label{sec:int:V_bounds}
Since our analysis is based on local arguments, we require bounds on
the interatomic potential in the neighbourhood of some given discrete
deformation. Let $y \in \Ys$ be such a deformation, and let $\epsilon
> 0$, then we define
\begin{equation}
  \label{eq:defn_Mrho}
  M^{(\bfrho)}_\epsilon(y) := \sup_{\xi \in \L} \sup_{\substack{ \bfg \in (\R^m)^\Rg \\
     \max_{\rho \in \Rg} \frac{|D_\rho y(\xi)- g_\rho|}{|\rho|} \leq \epsilon}} \sup_{\substack{ \bfh = (h_i)_{i
        = 1}^j \in (\R^m)^j \\ |h_1| = \dots = |h_j| = 1}} V_{,\bfrho}(\bfg) :
  \otimes_{i = 1}^j h_i \qquad \text{for } \bfrho \in \Rg^j.
\end{equation}
Our assumptions on $V$ and $y_0$ ensure that
$M^{(\bfrho)}_\epsilon(y)$ is finite for all $\epsilon > 0$ and $y \in
\Ys$.

\begin{lemma}
  Let $y \in \Ys, Dy \in \ell^\infty$, and $\epsilon > 0$ then, for $z
  \in \Ys, \|\D \barz - \D \bary \|_{L^\infty} \leq \epsilon$, $|x| >
  2 \rcut+\sqrt{d}$ and $|x-\xi| \leq \rcut+\sqrt{d}$,
  \begin{align}
    \label{eq:cons:expansion_V_1ord}
    & \b| V_{,\rho}(Dz(\xi)) -  V_{,\rho}(\D_\Rg \tilz(x)) \b| \leq
    C_2 \| \D^2
    \tilz \|_{L^\infty(\nu_x)}, \qquad \text{and} \\[1mm]
    \label{eq:cons:expansion_V_2ord}
    &\b| V_{,\rho}(Dz(\xi)) - \b[ V_{,\rho}(\D_\Rg \tilz(x))
    + \sum_{\vsig\in\Rg} V_{,\rho\vsig}(\D_\Rg\tilz(x)) \b(\D_\vsig
    \tilz(x) - D_\vsig z^x(\xi) \b) \b] \b|
    \\ \notag & \hspace{5cm}
    \leq C_3 \| \D^2 \tilz \|_{L^\infty(\nu_x)}^2,
  \end{align}
  where the constants $C_j$ depend on $M^{(\bfrho)}_{\epsilon}(y)$,
  $\bfrho \in \Rg^j$.
\end{lemma}
\begin{proof}
  Using the definition of $z^x$ according to \eqref{eq:defn_yx} and
  \eqref{eq:defn_Mrho} the estimates follow from Taylor expansions of
  $V_{,\rho}$.
\end{proof}

\subsubsection{Atomistic stress}
\label{sec:int:stress_at_cb}
To prove consistency we will employ ``weak forms'' of the atomistic and
the B-QC formulations that are local in the test function
gradient. The first step is to derive first Piola--Kirchhoff stresses
for the three models and estimate their discrepancy in terms of the
local regularity of the underlying deformation. This analysis is based
on the atomistic stress function analyzed in \cite{2012-ARMA-cb},
which is closely related to Hardy stress \cite{Hardy1982}.

A canonical representation of $\del\Ea$ is
\begin{equation}
  \label{eq:stress:delEa_v1}
  \b\<\del\Ea(y), v \b\> = \sum_{\xi\in\L}\sum_{\rho \in\Rg} V_{\xi,\rho}
  \cdot D_\rho v(\xi), \quad \text{where} \quad
  V_{\xi,\rho} := V_{,\rho}(Dy(\xi)).
\end{equation}
To convert $\del\Ea$ into a ``weak form'' that is local in $\D v$ we
replace $v$ with
\begin{equation}
  \label{eq:defn_convoluted_v}
  v^* := \zz \ast \barv
\end{equation}
and rewrite the finite differences $D_\rho v^*(\xi)$ as follows:
\begin{align}
  \notag D_\rho v^*(\xi) &= \int_{s = 0}^1 \Dc{\rho} v^*(\xi+s\rho)
  \ds = \int_\Om\int_{s=0}^1 \zz(\xi+s\rho - x) \Dc{\rho} \barv(x) \ds
  \dx \\
  \label{eq:stress:5}
  &= \int_\Om \ww_\rho(\xi-x) \D_\rho \barv \dx \qquad
  \text{ where } \quad \ww_\rho(x) := \int_{s = 0}^1 \zz(x + s \rho) \ds,
\end{align}
to obtain
\begin{align*}
 \b\<\del\Ea(y), v^* \b\>
 \notag
 =~& \sum_{\xi\in\L}\sum_{\rho \in\Rg}  V_{\xi,\rho} \cdot \int_{\R^d}
 \ww_{\rho}(\xi-x) \, \D_\rho\barv(x) \dx \\
 =~& \int_{\R^d} \bg\{ \sum_{\xi \in \L} \sum_{\rho \in \Rg}
 \b[V_{\xi,\rho} \otimes \rho\b] \ww_\rho(\xi-x) \bg\} : \D\barv \dx.
\end{align*}
Thus, we have shown that, for $y \in \Ys$ and $v \in \Usc$,
\begin{align}
  \label{eq:stress:delEa_v2}
  \b\< \del\Ea(y), v^* \> &= \int_{\R^d} \Sa(y; x) : \D \barv(x)
  \dx, \qquad \text{where} \\
  \notag
  \Sa(y; x) &:= \sum_{\xi \in \Z^d} \sum_{\rho \in \Rg}
  \b[V_{\xi,\rho} \otimes \rho\b] \ww_\rho(\xi-x).
\end{align}
(The representation \eqref{eq:stress:delEa_v2} is of course equivalent
to \eqref{eq:stress:delEa_v1} since neither require any regularity on
$y$. We use the term ``weak form'' only in analogy with the continuum
theory.)

Note that \eqref{eq:stress:delEa_v2} is in close analogy to the first
Piola--Kirchhoff stress of the Cauchy--Born model,
\begin{equation}
  \label{eq:stress:delEc_v1}
  \b\< \del\Ec(y), v \b\> = \int_\Om \Scb(y) : \D v \dx, \qquad
  \text{where } \Scb(y; x) = \partial W(\D y).
\end{equation}

To see the connection between the atomistic and Cauchy--Born stress we
replace $Dy(\xi)$ with $Dy^x(\xi)$ and expand analogously to
\eqref{eq:inter:strain_expansion} and $V_{,\rho}$ analogously to
\eqref{eq:cons:expansion_V_2ord}, to obtain
\begin{equation}
  \label{eq:inter:Sa-Sc-formal}
  \Sa(y; x) - \Sc(y; x) \sim \bbC_2(x) : \D^3 y(x) + \bbC_3(x) : \b(\D^2 y(x)
  \otimes \D^2 y(x)\b) + {\rm HOTs},
\end{equation}
where $\bbC_2(x)$ is a sixth order tensor depending on
$V_{,\bfrho}(\D_\Rg y(x))$, $\bfrho \in \Rg^2$, $\bbC_3(x)$ is an
eighth order tensor depending on $V_{,\bfrho}(\D_\Rg y(x)), \bfrho \in
\Rg^3$, and ${\rm HOTs}$ are formally higher-order terms, such as
$O(|\D^2 y|^3)$ or $O(|\D^4 y|)$.

The calculation \eqref{eq:inter:Sa-Sc-formal} exploits the fact that
we can write $V_{\xi,\bfrho} = V_{,\bfrho}(Dy(\xi)) =
V_{,\bfrho}(Dy^x(\xi))$ for $|x-\xi| \leq \rcut+\sqrt{d}$, which
removes discontinuities from $y$, as well as the following two
identities: \cite[Lemma 4.4]{2012-ARMA-cb}
\begin{align}
  \label{eq:stress:phirho_Prop0}
  \sum_{\xi \in \Z^d} \ww_\rho(\xi-x) =~& 1, \qquad \text{and} \\
  \label{eq:stress:phirho_Prop1}
  \sum_{\xi \in \Z^d} \ww_\rho(\xi-x) \,(\xi-x) =~& -\smfrac12 \rho.
\end{align}

The following lemma provides a rigorous estimate along the lines of
\eqref{eq:inter:Sa-Sc-formal}.


\begin{lemma}
  \label{th:cb_model_err}
  Suppose that $y \in \Ys$ and $\epsilon > 0$, then for $z \in \Ys, \|
  \D \barz - \D \bary \|_{L^\infty} \leq \epsilon$,
  \begin{displaymath}
    \b| \Sa(z; x) - \Sc(\tilz; x) \b| \leq C \b( \| \D^3 \tilz \|_{L^\infty(\nu_x)}
    + \| \D^2 \tilz \|_{L^\infty(\nu_x)}^2 \b),
  \end{displaymath}
  where $\nu_x$ is defined in \eqref{eq:defn_nux} and $C$ depends on
  $M^{(\bfrho)}_\epsilon(y),\bfrho \in \Rg^j, j = 2, 3$.
\end{lemma}
\begin{proof}
  This result is essentially contained in
  \cite[Thm. 4.3]{2012-ARMA-cb}. The only modification required is to
  replace the expansion of $D_\rho \tilz(\xi)$ with that of $D_\rho
  z^x(\xi)$ as detailed in \S~\ref{sec:int:strain}. It is also a
  simplified case of Lemma \ref{th:bqce_moderr_Rbest_est_pointwise}.
\end{proof}

\subsubsection{Best approximation operator}
\label{sec:res:bestapprox}
We construct a quasi-best approximation operator $\Pi_h : \Ys \to
\Ys_h$. With slight abuse of notation, we write $\Pi_h y = y_0 +
\Pi_h u$, where $y = y_0 + u$, $u \in \Us^{1,2}$, and $\Pi_h$ is also
understood as an operator from $\Us \to \Us_h$.

Given $u \in \Us^{1,2}$ we define $\Pi_h u := I_h T_R u$, where $I_h$
is the nodal interpolation operator defined in
\S~\ref{sec:results:coarsening} and $T_R$ is a truncation operator
defined as follows: we fix some arbitrary $\eta \in C^3(0, \infty)$
(e.g. a quintic spline) with $\eta(t) = 1$ in $[0, 1/2)$ and $\eta =
0$ in $[1, \infty)$, and define
\begin{equation}
  \label{eq:truncation_op}
  T_R u(\xi) := \eta\b( \smfrac{|\xi|}{\Ri} \b) \b( u(\xi) -
  {\textstyle \mint_{B_{\Ri} \setminus B_{\Ri/2}} \baru \dx} \b).
\end{equation}
Clearly, $T_R u \in \Usc$ with ${\rm supp}(T_R u) \subset \Omh$ and
hence $\Pi_h u \in \Us_h$.

\begin{lemma}
  \label{th:res:bestapx_est}
  There exists a constant $C$ such that,
  \begin{displaymath}
    \| \D \Pi_h y - \D \bary \|_{L^2} \leq C \errba(y) \qquad \text{
      for } y \in \Ys,
  \end{displaymath}
  where $\errba$ is defined in \eqref{eq:defn_errorterms}.
\end{lemma}
\begin{proof}
  The result follows immediately upon combining \cite[Lemma
  4.3]{EhrOrtSha:defects}, Lemma \ref{th:prelims:defn_smoothint}, and
  standard interpolation error estimates.
\end{proof}

\subsection{B-QCE consistency error}
\label{sec:inter:bqce_stress}
\def\GEb{\mathscr{G}_{\mathscr{E}}^\beta}
We have now assembled the prerequisites to define and estimate the
B-QCE consistency error. The first variation of $\Ebh$ is given by
\begin{displaymath}
  \b\< \del\Ebh(y_h), v_h \b\> = \sum_{\xi \in \L}(1-\beta(\xi)) \b\< \del
  V(Dy_h(\xi)), Dv_h(\xi) \b\> + \int_{\Om} Q_h \b[ \beta \dW(\D y_h)
  : \D v_h \b] \dx,
\end{displaymath}
for $y_h \in \Ys_h, v_h \in \Us_h$.  Since $v_h$ cannot be immediately
replaced with a function $v^*$ (to apply the convolution trick
\eqref{eq:stress:5})
we shall not convert this directly to a
``weak formulation''. Instead, suppose that $y \in \Ys, v \in \Usc$ such
that $y_h(\xi) = y(\xi)$ and $v^*(\xi) = v_h(\xi)$ for all $\xi \in
\La$. Then, arguing analogously as in \S~\ref{sec:int:stress_at_cb} we
can compute
\begin{align*}
  \b\< \del\Ea(y), v^* \b\> &=
  \sum_{\xi \in \L} (1-\beta(\xi)) \b\< \del V(Dy(\xi)), Dv^*(\xi)
  \b\>
  + \sum_{\xi \in \L} \beta(\xi) \b\< \del V(Dy(\xi)), D v^*(\xi)
  \b\> \\
  &= \sum_{\xi \in \L} (1-\beta(\xi)) \b\< \del V(Dy_h(\xi)), Dv_h(\xi)
  \b\>  \\
  & \qquad \qquad \qquad + \int_{\Om} \bg\{ \sum_{\xi \in \L} \beta(\xi) \sum_{\rho \in
    \Rg} \b[ V_{\xi,\rho} \otimes \rho \b] \ww_\rho(\xi-x) \bg\} : \D
  \barv \dx,
\end{align*}
where $V_{\xi,\rho} = V_{,\rho}(Dy(\xi))$. Thus, we obtain
\begin{align}
  \label{eq:inter:conserr_bqce:1}
  & \b\< \del\Ebh(y_h), v_h \b\> -  \b\< \del\Ea(y), v^* \b\> \\
  \notag
  & \quad = \int_{\Om} Q_h \b[ \beta \dW(\D y_h)
  : \D v_h \b] \dx - \int_{\Om} \bg\{ \sum_{\xi \in \L} \beta(\xi) \sum_{\rho \in
    \Rg} \b[ V_{\xi,\rho} \otimes \rho \b] \ww_\rho(\xi-x) \bg\} : \D
  \barv \dx,
\end{align}
with obvious analogies between the two groups on the right-hand
side. To complete the definition of the atomistic test function, we
take $v = \Pi_h' v_h$, where $\Pi_h' : \Us_h \to \Usc$ is a dual
approximation operator given by the conditions
\begin{equation}
  \label{eq:defn_Pih'}
  \begin{split}
    (\Pi_h' v_h)^*(\xi) &= v_h(\xi),\text{ for } \xi \in \La, \quad
    \text{and} \\
   \Pi_h' v_h(\xi) &= (\zz \ast v_h)(\xi), \quad \text{for } \xi \in \L
   \setminus \La.
  \end{split}
\end{equation}
We prove in Lemma \ref{th:inter:Pih'-lemma} that $\Pi_h'$ is
well-defined.

In order to estimate the consistency error we must estimate (1) the
quadrature error, which is standard; (2) the conformity error encoded
in the usage of two different test functions, which requires a
specific non-standard choice of $v$,
cf. \S~\ref{sec:bqce_coarse_err_est}; and (3) the modelling error
encoded in the difference between the two ``stresses''.

To indicate how we estimate the latter, we consider the simplified
``stress error''
\begin{equation}
  \label{eq:inter:bqce_stress_error}
  \mR^\beta(y; x) := \beta(x) \dW(\D y(x)) - \sum_{\xi \in \L}
  \beta(\xi) \sum_{\rho \in \Rg}
  \b[ V_{\xi,\rho} \otimes \rho \b] \ww_\rho(\xi-x),
\end{equation}
where $y$ is now a smooth function and $V_{\xi,\rho} =
V_{,\rho}(Dy(\xi))$. A formal Taylor expansion, similar as the one
leading to \eqref{eq:inter:Sa-Sc-formal}, but also expanding
$\beta(\xi)$ in terms of $\D^j \beta(x)$, yields
\begin{align}
  \label{eq:inter:Rb-est-formal}
  \mR^\beta(y; x) &\sim  \bbD_1(x) : \D^2\beta(x) + \bbD_2(x) : \b(
  \D\beta(x) \otimes \D^2 y(x) \b) \\
  \notag
  & \qquad \qquad
  + \beta(x) \B(  \bbC_2(x) : \D^3 y(x) + \bbC_3(x) : \b(\D^2 y(x)
  \otimes \D^2 y(x)\b) \B)
  + {\rm HOTs},
\end{align}
where $\bbD_1(x)$ is a fourth order tensor that depends on
$V_{,\rho}(\D_\Rg y(x))$, $\rho \in \Rg$, $\bbD_2(x)$ is a sixth order
tensor that depends on $V_{,\bfrho}(\D_\Rg y(x)), \bfrho \in \Rg^2$,
$\bbC_2, \bbC_3$ are the same tensors as in
\eqref{eq:inter:Sa-Sc-formal} and ${\rm HOTs}$ are formally higher
order terms.

\begin{theorem}[Consistency of B-QCE]
  \label{th:cons:bqce}
  Suppose that $y \in \Ys$, then there exist $\epsilon =
  \epsilon(\errba(y))$ such that, for all $v_h \in \Us_h$,
  \begin{align*}
    \b\< \del\Ebh(\Pi_h y), v_h \b\> - \b\< \del\Ea(y), \Pi_h' v_h \b\> \leq
    C\b( \errba(y) + \errcb(y) + \errbqce(y) \b) \, \| \D v_h
    \|_{L^{2}},
  \end{align*}
  where $C$ depends on $M^{(\epsilon)}_\bfrho(y)$, $\bfrho \in \Rg^j$,
  $1 \leq j \leq 4$.
\end{theorem}




\subsection{B-QCF consistency error}
\label{sec:int:bqcf_cons}
The consistency analysis of the B-QCF scheme faces different
challenges than that of the B-QCE scheme. Consider again $y \in \Ys,
y_h \in \Ys_h, v_h \in \Us_h$ and a microscopic test function $v \in
\Usc$, then we need to estimate
\begin{displaymath}
  \< \Fbh(y_h), v_h \> - \< \del\Ea(y), v \> =
  \< \del\Ea(y_h), (1-\beta) v_h \> + \< \del\Ech(y_h), I_h[\beta
  v_h] \> - \< \del\Ea(y), v \>.
\end{displaymath}
Choosing $v := \Pi_h'' v_h$, where $\Pi_h'' : \Us_h \to \Usc$ is
another dual approximation operator defined through
\begin{equation}
  \label{eq:int:cons_bqcf:defn_v}
  \Pi_h'' v_h := (1-\beta) v_h|_{\Z^d} + w^*, \qquad \text{where} \qquad w(\xi) = (\zz \ast
  I_h[\beta v_h])(\xi),
\end{equation}
we obtain
\begin{align*}
  \< \Fbh(y_h), v_h \> - \< \del\Ea(y), v \>
  &= \< \del\Ech(y_h), I_h[\beta v_h] \> - \< \del\Ea(y), w^* \>,
\end{align*}
from which we can estimate (see \S~\ref{sec:bqcf_cons_prf_1} for the
details)
\begin{equation}
  \label{eq:int:bqcf_cons_firststep}
  \< \Fbh(\Pi_h y), v_h \> - \< \del\Ea(y), v \> \leq C
  \b( \errba(y) + \errcb(y)\b) \| \D I_h[\beta v_h] \|_{L^2}.
\end{equation}

Thus, we need to estimate $\| \D I_h[\beta v_h] \|_{L^2}$ in terms of
$\| \D v_h \|_{L^2}$, which is provided in the following lemma.  The
key technical ingredient in its proof is a sharp trace inequality.

\begin{lemma}
  \label{th:int:est_Dbv_Dv}
  Suppose that the blending function $\beta$ satisfies
  \eqref{eq:blend_stab:radii}, then there exists a generic constant
  $C$, such that
  \begin{align}
    \label{eq:int:est_Dbv_Dv}
    \| \D I_h[\beta v_h] \|_{L^2} &\leq C\, \Ctr \| \D v_h \|_{L^2} \qquad \forall v_h \in
    \Us_h, \\[2mm]
    \label{eq:int:trace_ineq_const}
    & \text{where } \Ctr = \cases{
      \sqrt{1+\log(\Ro / \Ra)}, & d = 2, \\
      1, & d = 3.
    }
  \end{align}
\end{lemma}
\begin{proof}
  The proof is given in \S~\ref{sec:trace}.
\end{proof}

Based on the previous lemma we can establish the following B-QCF
consistency estimate.

\begin{theorem}[Consistency of B-QCF]
  \label{th:cons:bqcf}
  Suppose that $y \in \Ys$, then there exists $\epsilon =
  \epsilon(\errba(y)) > 0$ such that, for all $v_h \in \Us_h$,
  \begin{align*}
    &\b\< \Fbh(\Pi_h y), v_h \b\> - \b\< \del\Ea(y), \Pi_h'' v_h \b\>
    \leq C \, \Ctr \b( \errba(y) + \errcb(y) \b) \, \| \D v_h \|_{L^2}
  \end{align*}
  where $C$ depends on $M^{(\epsilon)}_\bfrho(y)$, $\bfrho \in \Rg^j, j
  = 2, 3$.
\end{theorem}
\begin{proof}
  The result immediately follows from
  \eqref{eq:int:bqcf_cons_firststep}, which is proven in
  \S~\ref{sec:bqcf_cons_prf_1}, and from
  Lemma~\ref{th:int:est_Dbv_Dv}.
\end{proof}

\subsection{Stability of B-QCE}
\label{sec:int:stab_bqce}
The aim of our stability result is to show that, if $y$ is a stable
equilibrium of the atomistic model, then choosing sufficiently large
atomistic and blending regions, we ensure that $\Pi_h y$ is stable in
the B-QCE model.

\begin{theorem} \label{th:int:stab_bqce}
Suppose $y\in\Ys$ is a stable atomistic configuration, i.e.,
\begin{equation}\label{eq:int:stab:stab_ass}
  0 < \gamma^\a(y) := \inf_{\substack{v \in \Usc \setminus \{0\} }}
  \frac{\< \ddel \Ea(y) v + \ddel \Pa(y) v, v \>}{\| \D \barv \|_{L^2}^2 }
  ,
\end{equation}
and denote
\[
	\gamma^\beta_h(y_h) := \inf_{\substack{v_h \in \Us_h \setminus \{0\} }}
	\frac{\< \ddel \Ebh(y_h) v_h + \ddel \Pa(y_h) v_h, v_h \>}{\| \D v_h \|_{L^2}^2 }
.
\]
Then there exists $\Delta\gamma(\Ra) \to 0$ as $\Ra\to \infty$ such that $\gamma^\beta_h(\Pi_h y) \geq \gamma^\a(y) - \Delta\gamma(\Ra)$.
\end{theorem}
\medskip

Positivity of $\gamma^\a$ is a property of the interatomic potential
and of the defect that we are aiming to compute, hence we postulated
this as an {\em assumption}.

The idea of the stability proof is to take a sequence of approximation
parameters $(\beta_j, \Ts_{h,j})$ with $\Ra_j \uparrow \infty$ and of
minimising test functions $v_j \in \Us_{h,j}$ (the space is now
indexed by $j$) such that $\|\D v_{j}\|_{L^2} = 1$ and $\b\< \b(\ddel
\Ebh (\Pi_{h,j} y) + \ddel \Pa(\Pi_{h,j} y) \b) v_j, v_j\b\> =
\gamma^\beta_h$. Due to the bound $\| \D v_j \|_{L^2} = 1$, we can
extract a weakly convergent subsequence (still denoted by $v_j$). This
sequence is then decomposed into three components (scales): $v_j =
v_j^{{\rm a}} + v_j^{\rm b} + v_j^{\rm c}$, for each of which we use a
different stability argument:
  \begin{itemize}
  \item $\D v^{\rm a}_j$ converges strongly at the atomic scale. It is
    concentrated near the defect core, hence for a sufficiently large
    atomistic region stability of the defect implies stability for
    this test function.

  \item $\D v^{\rm b}_j$ converges weakly to zero at the atomic scale
    but strongly at the ``interfacial scale''; i.e., after a rescaling
    $w_j^{\rm b}(x) = \delta v_j^{\rm b}(x / \epsilon)$, where
    $\epsilon \approx (\Ra)^{-1}$ and $\delta$ is chosen so that $\|\D
    w_j^{\rm b}\|_{L^2} = \|\D v_j^{\rm b}\|_{L^2}$.  This scaling
    keeps the interface (i.e., $\supp(\nabla \beta)$) near $|x|=1$ as
    $\eps\to 0$.  Consistency of B-QCE implies that the action of the
    B-QCE hessian on this test function is approximately the same as
    that of the Cauchy--Born hessian, hence stability of the continuum
    model implies stability for this component of the test function.

  \item $\D v^{\rm c}_j$ converges weakly to zero both at the atomic
    and ``interfacial scale'' (which means that it is not concentrated
    near a defect or interface).  We can then exploit that, for a
    subsequence, $v^{\rm c}_j \to 0$ strongly in $L^2(B_{\Rb})$ to
    reduce the action of the B-QCE hessian on this test function to
    the independent actions of the linearized atomistic and continuum
    operators which are both stable.

  \item All cross-terms can be neglected in the limit as $j \to
    \infty$ due to an approximate orthogonality between the three
    components.
  \end{itemize}


In practice, the idea outlined above is carried out in two
steps. First, we reduce the question to stability of a homogeneous
deformation, by only splitting $v_j = v_j^{\rm a} + (v_j^{\rm b} +
v_j^{\rm c})$.

\begin{lemma} \label{th:int:stab_bqce_lem1} Under assumptions and
  notation of Theorem \ref{th:int:stab_bqce}, there exists
  $\Delta\gamma(\Ra) \to 0$ as $\Ra\to \infty$ such that
  $\gamma^\beta_h(\Pi_h y) \geq \min\big\{\gamma^\a(y),
  \gamma^\beta_h(\mA x) \big\} - \Delta\gamma(\Ra)$.
\end{lemma}
\medskip

Thus, we are left to establish positivity of $\gamma^\beta_h(\mA
x)$. We will use the fact that positivity of $\gamma^\a(\mA x)$
follows from the positivity of $\gamma^\a(y)$.

\begin{lemma}
  \label{th:int:stab_bqce_lem2}
	Under assumptions and notation of Theorem \ref{th:int:stab_bqce}, there exists $\Delta\gamma(\Ra) \to 0$ as $\Ra\to \infty$ such that $\gamma^\beta_h(\mA x) \geq \gamma^\a(\mA x) - \Delta\gamma(\Ra)$.
\end{lemma}
\medskip

Both Lemma \ref{th:int:stab_bqce_lem1} and Lemma
\ref{th:int:stab_bqce_lem2} are proven in \S~\ref{sec:prf_bqce_stab}.

\begin{proof}[Proof of Theorem \ref{th:int:stab_bqce}.]
  In view of Lemmas \ref{th:int:stab_bqce_lem1} and
  \ref{th:int:stab_bqce_lem2} we only need to note that $\gamma^\a(\mA
  x) \geq \gamma^\a(y)$ which is proved in \cite{EhrOrtSha:defects}.
\end{proof}

We remark that our arguments to obtain convergence of the stability
constants employ compactness principles and do not yield convergence
rates as in 1D \cite{2013-atc.acta}.

\subsection{Stability of B-QCF}
\label{sec:int:stab_bqcf}
The B-QCF stability result is analogous to the B-QCE stability
result. Unlike in the B-QCE case we state the result only for stable
equilibria (rather than general deformations) since we require some
regularity of the underlying deformation in the proof.

\begin{theorem} \label{th:int:stab_bqcf} Suppose $\ya \in\Ys$ is a
  strongly stable solution of \eqref{eq:min_atm_exact}, i.e.,
  \eqref{eq:int:stab:stab_ass} holds,
  and let
\[
	\mu^{\beta}_h := \inf_{\substack{v_h \in \Us_h \setminus \{0\} }}
	\frac{\<\del \Fbh(\Pi_h \ya) v_h + \ddel \Pa(\Pi_h \ya) v_h, v_h \>}{\| \D v_h \|_{L^2}^2 }
.
\]
Then there exists $\Delta\gamma(\Ra) \to 0$ as $\Ra\to \infty$ such that $\mu^\beta_h(\Pi_h y) \geq \gamma^\a(y) - \Delta\gamma(\Ra)$.
\end{theorem}

\medskip

It is possible to adapt the proof of Theorem \ref{th:int:stab_bqce}
to prove this result, however, we obtain it via an alternative route
using an auxiliary result that it interesting in its own right: We
modify a result from \cite{BQCF}, which shows in a
simplified case that the B-QCE hessian and B-QCF jacobian are
``close''. Here, we only establish that their stability constants
converge to the same limit as $\Ra \to \infty$.

\begin{lemma}
  \label{th:int:stab_bqcf_lem1}
  Under the assumptions and notation of Theorem
  \ref{th:int:stab_bqcf}, there exists a constant $C$ such that
  \begin{align*}
    \b|\mu_h^\beta - \gamma_h^\beta(\Pi_h \ya)\b| \leq C\,\cases{
      (\Ra)^{-1} (\log \Ra)^{1/2}, & \text{ if $d = 2$}, \\
      (\Ra)^{-1}, & \text{ if $d = 3$}.
    }
\end{align*}
\end{lemma}

The proof of Lemma \ref{th:int:stab_bqcf_lem1} is given in
\S~\ref{sec:prf_bqcf_stab}.

\begin{proof}[Proof of Theorem \ref{th:int:stab_bqcf}]
  The result is an immediate corollary of Theorem
  \ref{th:int:stab_bqce} and Lemma~\ref{th:int:stab_bqcf_lem1}.
\end{proof}

\subsection{Proofs of the error estimates}
\label{sec:int:proofs}
\def\Fi{\mathscr{R}_h}
 We have now assembled all required auxiliary results to complete the
proof of Theorem \ref{th:error}.

\begin{proof}[Proof of Theorem \ref{th:error}]
  Let $y^\a$ be a fixed strongly stable atomistic equilibrium.  Using
  the notation established in \S~\ref{sec:int:framework}, we define
  $\Fi : \Us_h \to \Us_h^*$,
  \begin{displaymath}
    \< \Fi(w_h), v_h \> := \< \G_h(\Pi_h y^\a + w_h), v_h \> \quad
    \forall v_h \in \Us_h.
  \end{displaymath}

  {\it 1. Stability: } Theorems \ref{th:int:stab_bqce} and
  \ref{th:int:stab_bqcf} show that there exists $R^\a_1$ such that,
  for $\Ra \geq R^\a_1$, we have \eqref{eq:int:stability} for a
  constant $c_0 > 0$ that depends on $R^\a_1$, but is independent of
  $\Ra$. This implies that $\| \del \Fi(0)^{-1} \|_{L(\Us_h^*, \Us_h)} \leq
  c_0^{-1}$.

  {\it 2. Consistency: } Theorems \ref{th:cons:bqce} and
  \ref{th:cons:bqcf} imply that
  \begin{displaymath}
    \| \Fi(0) \|_{\Us_h^*} = \| \G_h(\Pi_h y^\a) \|_{\Us_h^*} \to 0, \quad \text{as } \Ra
    \to 0,
  \end{displaymath}
  uniformly in all choices of $(\beta, \Th)$. In particular, for any
  $\epsilon > 0$ we can choose a constant $R^\a_0 \geq R^\a_1$ such
  that $\| \Fi(0) \|_{\Us_h^*} \leq \epsilon$ whenever
  $\Ra \geq R^\a_0$.

  (In the B-QCF case, due to the logarithmic prefactor $\Ctr$ in the consistency error estimates, this requires the
  regularity estimates \eqref{eq:regularity_ua}.)

  {\it 3. Inverse function theorem: } Our assumptions on $V$ and the fact that
  $\errba(y^\a) \leq \epsilon$ for $\Ra \geq R^\a_0$ implies that $\|
  \del \G_h(y_h) - \del \G_h(z_h) \|_{L(\Us_h, \Us_h^*)} \leq L \|
  \D y_h - \D z_h \|_{L^2}$ for all $y_h, z_h \in \Us_h$, or,
  equivalently,
  \begin{displaymath}
    \| \del \Fi(w_h) - \del \Fi(z_h) \|_{L(\Us_h, \Us_h^*)} \leq L
    \| \D w_h - \D z_h \|_{L^2} \qquad \forall w_h, z_h \in \Us_h.
  \end{displaymath}
  The inverse function (see, e.g., \cite{2013-atc.acta}) states that,
  if $\| \Fi(0) \|_{\Us_h^*} L c_0^{-2} < 1$, then there exists $w_h
  \in \Us_h$ such that $\Fi(w_h) = 0$ and $\| \D w_h \|_{L^2} \leq 2
  c_0^{-1} \| \Fi(0) \|_{\Us_h^*}$. This can clearly achieved by
  setting $\epsilon$ sufficiently small.  Setting $y_h^{\rm bqc} :=
  \Pi_h y^\a + w_h$ we therefore obtain that
  \begin{displaymath}
    \|\D \Pi_h^\a y^\a - \D y_h^{\rm bqc}\|_{L^2} \leq 2 c_0^{-1} \| \Fi(0) \|_{\Us_h^*}.
  \end{displaymath}
  Inserting the estimates for $\| \Fi(0) \|_{\Us_h^*}$ from Theorems
  \ref{th:cons:bqce} and \ref{th:cons:bqcf}, and the fact that $\| \D
  \Pi_h^\a - \D \bary^\a \|_{L^2} \lesssim \errba(y^\a)$, we obtain
  the two error estimates \eqref{eq:bqce_general_errest} and
  \eqref{eq:bqcf_general_errest}.
\end{proof}

\section{Proofs of Interpolation and Approximation Results}
\label{sec:aux}
\subsection{Analysis of the quasi-interpolant}
Recall the definitions of $\barv$ from~\eqref{eq:interp:S1_interp} and
of $v^* := \zz \ast \barv$ from \eqref{eq:defn_convoluted_v}. To
summarize results concerning $v^*$ we first need the following lemma.

\begin{lemma}
  \label{th:reflection_symm}
  The partition $\TT$ is invariant under reflections about all lattice
  points $\xi \in \Z^d$. In particular, we have $\zz(\xi-x) = \zz(\xi+x)$ for
  all $\xi \in \Z^d, x \in \R^d$.
\end{lemma}
\begin{proof}
  In 2D the result is geometrically evident.

  In 3D, one first observes that the partition $\{\hat T_1, \dots,
  \hat T_6\}$ of the unit cube $[0, 1]^3$, shown in
  Figure~\ref{fig:tet}, is invariant under the map $x \mapsto (1,1,1)
  - x$ (which is the reflection about $(1/2, 1/2, 1/2)$). Moreover,
  since $\TT$ is translation invariant by construction, we obtain for
  $\xi \in \Z^d$,
  \begin{displaymath}
    \xi - T = \b[ \xi - (1,1,1) \b] + \b[(1,1,1) - T\b] \in \TT. \qedhere
  \end{displaymath}
\end{proof}

Based on Lemma \ref{th:reflection_symm} the analysis in
\cite{OrtShap:Interp} allows us to deduce the following statements:
Let $v \in \Us$, then $\barv \in W^{1,\infty}_\loc(\R^d; \R^m)$ and
$v^* \in W^{3,\infty}_\loc(\R^d; \R^m)$
\cite[Lemma~1]{OrtShap:Interp}. Further, there exists a constant $c$,
independent of $p$, such that, for all $u \in \Us$ and $p \in [1,
\infty]$, \cite[Theorem~2]{OrtShap:Interp}
\begin{align}
  \label{eq:interp_stab}
  & c \| \D \baru \|_{\LL^p} \leq \| \D u^* \|_{\LL^p} \leq \| \D
    \baru \|_{\LL^p}.
\end{align}

\subsection{Analysis of the smooth nodal interpolant}
\label{sec:analysis_smoothint}
\newcommand{\Ns}{\mathscr N}
\newcommand{\Ms}{\mathscr M}
\newcommand{\Nbar}{\bar{\mathscr N}}
Let $n\in\Z_+$.
For each multi-index $\alpha\in\Z_+^d$, $|\alpha|_\infty \leq n$, denote by $\partial_\alpha$ the respective partial derivative and let $D_\alpha$ be a finite difference approximation to $\partial_\alpha$.
We assume that each $D_\alpha$ is exact on polynomials of degree $n$ and is supported on
\[
\Ns_n = \{\xi\in\Z^d : |\xi| \leq \big\lceil\smfrac{n}{2}\big\rceil\}
.
\]

Next, for a lattice function $u$, introduce a $d$-dimensional Hermite interpolation based on derivatives $\partial_\alpha$, $|\alpha|_\infty \leq n$.
Namely, in each cell $\xi+B_d$, where
\[
B_d=\{x : 0\leq x_i < 1, ~i=1,2,\ldots,d\}
\]
is the $d$-dimensional unit cube, define a $Q_{2n+1}(\R^d)$
polynomial, i.e., a polynomial in $x_1$, $x_2$, \ldots, $x_d$, of
degree at most $2n+1$ in each variable (and thus of degree at most $d(2n+1)$) $P_{u,\xi}(x)$ such that
\begin{equation}\label{eq:smooth_interp:eqP}
\partial_\alpha P_{u,\xi}(x) = D_\alpha u(x)
\quad \text{for all $2^d$ verticies $x$ of the cell $\xi+B_d$}
\end{equation}
and define
\begin{equation} \label{eq:smooth_interp:I_def}
\tilu(x) := P_{u,\xi}(x)
\quad \text{if $x\in \xi+B_d$}.
\end{equation}

\begin{lemma}\label{lem:smooth_interp:well_posed}
The relation \eqref{eq:smooth_interp:I_def} uniquely defines $\tilu$ for any lattice function $u:\Z^d\to\R$.
\begin{proof}
For $\mu\in\Z_+^d$, $|\mu|_\infty \leq 2n+1$, let $B_{2n+1,\mu}(x) := \prod_{i=1}^d x_i^{2n+1-\mu_i} (1-x_i)^{\mu_i}$ be the multivariate Bernstein polynomial.
These polynomials form a basis of $Q_{2n+1}(\R^d)$ and on the other hand upper-triangularize the linear system \eqref{eq:smooth_interp:eqP}.
Hence the solution $P_{u,\xi}$ to \eqref{eq:smooth_interp:eqP} exists and is unique.
\end{proof}
\end{lemma}

\begin{lemma}[Regularity]\label{lem:smooth_interp:regularity}
For any lattice function $u:\Z^d\to\R$, $\tilu\in C^{n,1}_\loc(\R^d)$.
\end{lemma}
\begin{proof}
It is enough to prove that across any face shared by two cells, the function and normal derivatives up to order $n$ are continuous.

Indeed, without loss of generality, consider two adjacent cells, $B_d$ and $B_d-e_d$, where $e_d = (0,\ldots,0,1)\in\R^d$.
For $y\in\R^{d-1}$ denote $\bar{y} = (y_1,\ldots,y_{d-1},0)\in\R^d$.
Let $m\in\{0,1,\ldots,n\}$ be the order of the normal derivative and consider the polynomial $p(y) = \big(\frac{\partial}{\partial x_d}\big)^m (P_{u,0}(\bar{y})-P_{u,-e_d}(\bar{y}))$.
By construction of $P_{u,0}$ and $P_{u,-e_d}$, we have that $p\in Q_{2n+1}(\R^{d-1})$ and satisfies
\[
\partial_\beta p(y) = 0
\quad \text{for all $\beta\in\Z_+^{d-1}$ such that $|\beta|_{\infty} \leq n$ and all vertices $y$ of $B_{d-1}$.}
\]
Due to Lemma \ref{lem:smooth_interp:well_posed} such a polynomial is unique, hence we obtain $p(y) \equiv 0$, which implies continuity of $\tilu$ and its derivatives.
\end{proof}

\begin{lemma}[Stability]\label{lem:smooth_interp:stab}
For any $u:\Z^d\to\R$ and $\beta\in\Z_+^d:\, |\beta|_1 \leq n+1$,
\begin{equation}\label{eq:smooth_interp:stab}
\|\partial_\beta \tilu\|_{L^p(B_d)} \leq C \|D^{\cn,|\beta|_1} u\|_{\ell^p}
\end{equation}
for some constant $C$ independent of $u$, where $D^{\cn,m}$ is the collection of all finite differences of order $m$ whose stencil lies within
\[
\Nbar_n = \{\xi\in\Z^d : |\xi+\smfrac12| \leq \big\lceil\smfrac{n}{2}\big\rceil+\smfrac12\}
.
\]
\end{lemma}
\begin{proof}
Since both $\|\partial_\beta \tilu\|_{L^p(B_d)}$ and $\|D^{|\beta|_1}_{\Nbar_n} u\|_{\ell^p}$ are seminorms on the finite dimensional space $\{u:\Nbar_n\to\R\}$, \eqref{eq:smooth_interp:stab} may fail to hold only if there exists $u^\dagger:\Nbar_n\to\R$ such that $\partial_\beta \tilu^\dagger \not\equiv 0$ on $B_d$, but $D^{|\beta|_1}_{\Nbar_n} u^\dagger = 0$.
The latter may happen only if $u^\dagger$ is a polynomial of degree $|\beta|_1-1$.
Then, since $D_\alpha$ are exact on such polynomials (note that $|\beta|_1-1 \leq n$), $D_\alpha u^\dagger(\xi) = \partial_\alpha u^\dagger(\xi)$ for any vertex $\xi$ of $B_d$, therefore $P_{u^\dagger,0}(x) = u^\dagger(x)$ for all $x\in\R^d$, and hence $\partial_\beta \tilu^\dagger = \partial_\beta P_{u^\dagger,0} = 0$ on $B_d$.
\end{proof}

\begin{proof}[Proof of Lemma \ref{th:prelims:defn_smoothint}]
Applying Lemmas \ref{lem:smooth_interp:well_posed} and \ref{lem:smooth_interp:regularity} with $n=2$ proves part (a).
To show part (b), we apply Lemma \ref{lem:smooth_interp:stab} and note that any finite difference entering \eqref{eq:smooth_interp:stab} also enters \eqref{eq:prelims:defn_smoothint}.
\end{proof}



\subsection{Dual interpolant for B-QCE}
\label{sec:proofs_dual_int}
Recall the definition of $\Pi_h'$ from \eqref{eq:defn_Pih'}.

\begin{lemma}
  \label{th:inter:Pih'-lemma}
  The operator $\Pi_h' : \Us_h \to \Usc$ is well-defined. Moreover, it
  satisfies the estimates
  \begin{align}
    \| \D (\Pi_h' v_h)^* \|_{\LL^2} \leq \| \D \overline{\Pi_h' v_h} \|_{\LL^2} \leq~& C \| \D v_h
    \|_{\LL^2}, \quad \text{and} \\
    \| v_h - \overline{\Pi_h' v_h} \|_{\LL^2} \leq~& C \| \D v_h \|_{\LL^2},
  \end{align}
  where $C$ is a generic constant.
\end{lemma}
\begin{proof}
  To see that $v := \Pi_h' v_h$ is well-defined by
  \eqref{eq:defn_Pih'}, we first define $w \in \Usc$, $w(\xi) := (\zz
  \ast v_h)(\xi)$. From standard quasi-interpolation arguments (see,
  e.g., \cite{Verfurth, Verfurth:m2an}) we can deduce that
  \begin{displaymath}
    \| \barw - v_h \|_{L^2} \leq C \| \D v_h \|_{L^2}.
  \end{displaymath}

  Writing $v := w + z$, \eqref{eq:defn_Pih'} becomes
  \begin{align*}
    (\zz \ast z)(\xi) &= g(\xi), \qquad  \xi \in \La, \\
    z(\xi) &= 0, \qquad \quad \xi \in \L \setminus \La,
  \end{align*}
  where $g(\xi) = v_h(\xi)-\zz\ast{w}(\xi) = v_h(\xi) - (\zz\ast\zz\ast
  v_h)(\xi)$.
  Testing the first line with a test function $\varphi \in \Usc$,
  $\varphi = 0$ in $\L \setminus \La$, and using the fact that
  \begin{align*}
    \sum_{\xi \in \L} (\zz \ast z)(\xi) \varphi(\xi) &= \int_{\R^d}
    \sum_{\xi \in \L} \zz(x - \xi) \barz(x) \varphi(\xi) \dx
    = \int_{\R^d} \barz(x) \cdot \bar\varphi(x) \dx,
  \end{align*}
  we obtain the variational form
  \begin{displaymath}
    \int_{\R^d} \barz(x) \cdot \bar\varphi(x) \dx = \sum_{\xi \in \La}
    g(\xi) \cdot \varphi(\xi) \qquad \text{ for all } \varphi \in \Usc,
    \varphi|_{\L\setminus\La} = 0,
  \end{displaymath}
  from which it is now obvious that a unique solution
  exists.

  Testing with $\varphi = z$, we obtain that
  \begin{displaymath}
    \| \barz \|_{L^2} \leq C \|g \|_{\ell^2}.
  \end{displaymath}
  Exploiting the assumption that $\Nh$ and $\La$ coincide in $\Oma$ it
  is straightforward to show that
  \begin{displaymath}
    \| g \|_{\ell^2} \leq C \| \D v_h \|_{L^2},
  \end{displaymath}
  and we further obtain that
  \begin{align*}
    \| \D (\barv - v_h) \|_{L^2} &\leq C_1 \| \barv - v_h \|_{L^2}
    \leq C_1 \b( \| \barz \|_{L^2} + \| \bar w - v_h \|_{L^2} \b)
    \leq C_2 \| \D v_h\|.
  \end{align*}
  In particular, $\| \D \barv \|_{L^2} \leq C \| \D v_h
  \|_{L^2}$. This completes the proof of Lemma
  \ref{th:inter:Pih'-lemma}.
\end{proof}

\subsection{Inverse estimates}
\label{eq:inverse_ests}
Before we embark on the proof of the consistency estimates, we another
technical tool that allows us to convert local $L^\infty$ bounds into
$L^p$ bounds. This is motivated by the form of the estimate in Lemma
\ref{th:cb_model_err}.

Performing such conversions are standard norm-equivalence arguments
if the functions involved are piecewise polynomial:
\begin{equation}
  \label{eq:inv_est_poly}
  \| \D^j \tilv \|_{L^\infty(T)} \lesssim \| \D^j \tilv \|_{L^p(T)}
  \qquad \forall v \in \Us, T \in \TT, j = 0, \dots, 3, p \in [1, \infty].
\end{equation}
In the point defect case, this also extends to $y = y_0+u$, where $y_0
= \mA x$.

However, we will also need to perform such estimates for $y_0 = \mA x
+ \ulin$. To that end, we now construct a piecewise polynomial
interpolant of $y_0$ that takes into account the structure of
$\ulin$. For $y \in \Ys$, $x \in \R^d, |x| > 2 \rcut + 2\sqrt{d}$, we
define
\begin{equation}
  \label{eq:defn_haty0}
  \hat{y}^x(x') := \widetilde{y^x}(x') \qquad \text{for $x' \in \nu_x$},
\end{equation}
where $\widetilde{y^x}$ is the $C^{2,1}$-conforming piecewise
polynomial interpolant defined through
Lemma~\ref{th:prelims:defn_smoothint}. (Since $\hat{y}^x$ is piecewise
polynomial, it is {\em not} of the form $y_0 + \tilde{u}$ for any $u
\in \Us^{1,2}$.)

The interpolant is clearly well-defined and we obtain the following
bounds from standard interpolation error estimate arguments (e.g., see
\cite{Ciarlet:1978}): for $Q = \xi+(0,1)^d \subset
B_{2\rcut+2\sqrt{d}}(x)$ and $\omega_Q = \xi+(-1,2)^d$, $q \in [1,
\infty]$, we have
\begin{align}
  \label{eq:haty_errest}
  \| \D \tily - \D \hat{y}^x \|_{L^\infty(Q)} &\leq C_1 \| \D^3 \tily
  \|_{L^q(\omega_Q)}, \\
  \label{eq:haty_invest}
  \| \D^j \hat{y}^x \|_{L^\infty(Q)} &\leq C_2 \| \D^j \hat{y}^x
  \|_{L^q(Q)} \text{ for } j = 1, 2, 3, \quad \text{and}\\
  \label{eq:haty0_to_y0}
  \| \D^j \hat{y}^x \|_{L^q(Q)} &\leq C_3 \| \D^j \tily
  \|_{L^q(\omega_Q)} \text{ for } j = 2, 3,
\end{align}
where the constants $C_1, C_2, C_3$ are generic. While
\eqref{eq:haty_invest} is obvious, the two other estimates require
some comments.

\begin{proof}[Proof of \eqref{eq:haty_errest}]
  Since, for $d = 2$, $W^{3,1}$ is embedded in $C$, standard
  interpolation error arguments yield
  \begin{displaymath}
    \| \D \tily - \D \hat{y}^x \|_{L^\infty(Q)} = \| \D y^x - \D
    \widetilde{y^x} \|_{L^\infty(Q)} \lesssim \| \D^3 y^x
    \|_{L^1(\omega_Q)} = \| \D^3 \tily
    \|_{L^1(\omega_Q)}.
  \end{displaymath}

  For $d = 3$ the embedding fails, however, in this case $\tily$ is
  piecewise polynomial; that is, $\tily = \hat{y}^x$, hence the result
  is true in this case as well.
\end{proof}

\begin{proof}[Proof of \eqref{eq:haty0_to_y0}]
  Let $p$ be an arbitrary polynomial of degree $j-1$, then
  \begin{align*}
    \| \D^j \hat{y}^x \|_{L^q(Q)} &=   \| \D^j (\hat{y}^x - p)
    \|_{L^q(Q)} 
    \lesssim \| \D (\hat{y}^x - p) \|_{L^q(Q)} \\
    &\lesssim \| \D (\hat{y}^x - y^x) \|_{L^q(Q)} + C \| \D (y^x - p) \|_{L^q(Q)}.
  \end{align*}
    From \eqref{eq:haty_errest} and the Bramble-Hilbert Lemma,
    we obtain \eqref{eq:haty0_to_y0}.
\end{proof}

\section{Consistency Proofs}
\label{sec:cons}

\subsection{B-QCE coarsening error}
\label{sec:bqce_coarse_err_est}
Throughout this section and the next we assume the conditions of
Theorem \ref{th:cons:bqce}. Thus, let $y = y_0+u \in \Ys$ be fixed,
let $y_h = y_0+u_h := \Pi_h y^\a$ be its quasi-best approximation and
let $v_h \in \Usc$ be an arbitrary test function. We choose $v :=
\Pi_h' v_h \in \Usc$, where $\Pi_h'$ is defined in
\eqref{eq:defn_Pih'} and analysed in \S~\ref{sec:proofs_dual_int}, and
estimate the B-QCE consistency error
\begin{displaymath}
  \< \del\Ebh(y_h), v_h \> - \< \del\Ea(y^\a), v^* \>.
\end{displaymath}

Using the fact that $v^*(\xi) = v_h(\xi)$ for all $\xi \in \La$, and
employing \eqref{eq:inter:bqce_stress_error} we split the error as follows,
\begin{align*}
  & \< \del\Ebh(y_h), v_h \> - \< \del\Ea(y), v^* \> \\
  &= \int_{\R^d} Q_h \b[ \beta \b( \partial W(\D y_h) - \pp W(\D\tily)
  \b): \D
  v_h \b] \dx + \int_{\R^d} (Q_h-{\rm Id})\b[\beta \partial
  W(\D \tily) : \D v_h \b] \dx \\
  & \quad  + \int_{\R^d} \beta\partial W(\D \tily) : (\D v_h - \D\barv) \dx
  + \int_{\R^d} \mR^\beta(\tily; x) : \D \barv \dx \\
  &:= {\rm T}_1 + {\rm T}_2 + {\rm T}_3 + {\rm T}_4,
\end{align*}
where $\mR^\beta$ is defined in \eqref{eq:inter:bqce_stress_error}. In
the consistency error analysis of the B-QCF method in
\S~\ref{sec:bqcf_cons_prf} we use an analogous splitting, hence the
following estimates for the terms ${\rm T}_1, {\rm T}_2, {\rm T}_3$
will be used there as well.


\begin{lemma}
  \label{th:bqce_coarse_err_T1_T2}
  Under the conditions of Theorem \ref{th:cons:bqce}, the terms ${\rm
    T}_1$ and ${\rm T}_2$ are bounded by
  \begin{align*}
    {\rm T}_1 &\lesssim \B( \| \D u_h - \D \tilu \|_{L^2(\Omc)} + \| h^2 \D^3
    \tilu \|_{L^2(\Omc)}\B) \| \D v_h \|_{L^2} \quad \text{and} \\
    {\rm T}_2 &\lesssim \B(\b\| h^2 \D^2 \b[ \beta \pp W(\D \tily) \b]
    \b\|_{L^2}\B) \| \D v_h \|_{L^2}.
  \end{align*}
\end{lemma}
\vspace{-3mm}
\begin{proof}
  {\it 1. Estimate of ${\rm T}_1$: } Let $T \in \T$ such that $\beta \neq
  0$ on $T$, then
  \begin{align*}
    & \hspace{-1cm}\int_T Q_h\b[ \beta \b( \partial W(\D y_h) - \pp W(\D\tily) \b):
    \D v_h \b] \dx \\
    &\leq |T|^{1/2} \| \partial W(\D y_h) - \pp
    W(\D\tily) \|_{L^\infty(T)}  \| \D v_h \|_{L^2(T)} \\
    &\lesssim |T|^{1/2} \| \D u_h - \D \tilu \|_{L^\infty(T)} \| \D
    v_h \|_{L^2(T)},
  \end{align*}
  where $C$ depends on $\pp^2 W$ in a neighbourhood of $\D \tily$ and
  hence on $M_\epsilon^{(\bfrho)}(y), \bfrho \in \Rg^2$. Employing the embedding $H^2 \subset C$,
    \begin{align*}
    |T|^{1/2} \| \D u_h - \D \tilu \|_{L^\infty(T)} &\lesssim \B(\| \D u_h -
    \D \tilu \|_{L^2(T)} +  \| h^2 \D^3 \tilu \|_{L^2(T)}\B),
  \end{align*}
    where $C$ depends only on the shape regularity of the mesh.
  Summing over all $T$, we obtain the stated result. 

  {\it 2. Estimate of ${\rm T}_2$: } For any piecewise linear  (not
  necessarily continuous) $\psi_h$ we have
  \begin{align*}
    {\rm T}_2 &= \int_{\R^d} (Q_h - {\rm Id}) \b[ \beta \pp W(\D
    \tily) - \psi_h \b] : \D v_h \dx \\
    &\leq \b\| (Q_h - {\rm Id}) \b[ \beta \pp W(\D
    \tily) - \psi_h \b]\b\|_{L^2} \| \D v_h \|_{L^2}.
  \end{align*}
  Therefore, by the Bramble-Hilbert Lemma,
  \begin{displaymath}
    {\rm T}_2 \lesssim \b\| h^2 \D^2 \b[ \beta \pp W(\D\tily) \b] \b\|_{L^2} \|\D v_h\|_{L^2},
  \end{displaymath}
  where the constant depends again on the shape regularity of $\T$.
\end{proof}

\begin{lemma}
  \label{th:bqce_coarse_err_T3}
  Under the conditions of Theorem \ref{th:cons:bqce}, the term ${\rm
    T}_3$ is bounded above by
  \begin{align*}
    {\rm T}_3 &\lesssim \B(\b\| \D^2 \b[ \beta \pp W(\D\tily)\b]\b\|_{L^2} + \| h
    \D^2 \tilu \|_{L^2(\Omc)} + \| h \D^2 y_0 \otimes (\D y_0 - \mA)
    \|_{L^2(\Omc)}\B) \| \D v_h \|_{L^2}.
  \end{align*}
\end{lemma}
\vspace{-5mm}
\begin{proof}
  Let $\zeta_\nu$ be the nodal basis function associated with a node
  $\nu \in \Nh$, with support $\omega_\nu$, and let $f := -\divv\b[
  \beta \pp W(\D\tily) \b]$. We integrate the term ${\rm T}_3$ by
  parts, and then use the fact that $\zeta_\nu$ form a partition of
  unity, to obtain
  \begin{align*}
    {\rm T}_3 &= - \int_{\R^d} \divv\b[ \beta \pp W(\D\tily) \b] \cdot
    (v_h - \barv) \dx \\
    &= \sum_{\nu \in \Nh} \int_{\R^d} f(x) \cdot (v_h(\nu) - \barv(x))
    \zeta_\nu(x) \dx.
  \end{align*}
  \textit{ Case 1:}
   If $\nu \in \La \cap \Nh$, then $\zeta_\nu(x) = \zz(x-\nu)$
  and hence
  \begin{equation}
    \label{eq:coarse_T3:100:mean_zero_condition}
    \int_{\R^d} (v_h(\nu) - \barv(x))
    \zeta_\nu(x)  \dx = v_h(\nu) - v^*(\nu) = 0,
  \end{equation}
  by definition of $v^*$ and $v$.  Therefore,
  \begin{displaymath}
    \int_{\R^d} f(x) \cdot (v_h(\nu) - \barv(x))
    \zeta_\nu(x) \dx \lesssim \| \D f \|_{L^2(\omega_\nu)}
     \|\zeta_\nu^{1/2} (v_h(\nu)
    - \barv(x)) \|_{L^2(\omega_\nu)}
  \end{displaymath}
  where $\omega_\nu = \text{supp}\,\zeta_\nu$. Exploiting again
  \eqref{eq:coarse_T3:100:mean_zero_condition} we can estimate
  \begin{align*}
    \|
    \zeta_\nu^{1/2} (v_h(\nu)
    - \barv) \|_{L^2(\omega_\nu)}^2 &= \int_{\omega_\nu} (v_h(\nu) -
    \barv) \cdot \b[ (v_h(\nu) - \barv) \zeta_\nu\b] \dx \\
    &=   \int_{\omega_\nu} \big( (\barv)_{\omega_\nu} -
    \barv\big) \cdot \b[ (v_h(\nu) - \barv) \zeta_\nu\b] \dx  \\
    &\leq \| (\barv)_{\omega_\nu} - \barv \|_{L^2(\omega_\nu)} \|
    \zeta_\nu (v_h(\nu)
    - \barv) \|_{L^2(\omega_\nu)} \\
    &\lesssim \| \D\barv \|_{L^2(\omega_\nu)} \|
    \zeta_\nu^{1/2} (v_h(\nu)
    - \barv) \|_{L^2(\omega_\nu)},
  \end{align*}
  and hence we arrive at
  \begin{equation}
    \label{eq:coarse_T3:110}
    \int_{\R^d} f(x) \cdot (v_h(\nu) - \barv(x))
    \zeta_\nu(x) \dx \lesssim \| \D f \|_{L^2(\omega_\nu)} \| \D v_h
    \|_{L^2(\omega_\nu)} \qquad \text{ for } \nu \in \Nh \cap \La.
  \end{equation}
  \textit{Case 2:}
  Because of the way $v$ is defined, we do not have
  \eqref{eq:coarse_T3:100:mean_zero_condition} for $\nu \in \Nh
  \setminus \La$,
  but on the other hand $\beta \equiv 1$ in this case, which means
  that the second-order estimate is not crucial. In this case, using
  elementary interpolation error estimates, we obtain only
  \begin{align*}
    \int_{\R^d} f(x) \cdot (v_h(\nu) - \barv(x))
    \zeta_\nu(x) \dx &\leq \| h f \|_{L^2(\omega_\nu)} \b( \| h^{-1} (v_h(\nu) -
    v_h) \|_{L^2(\omega_\nu)} + \| h^{-1}(v_h - \barv) \|_{L^2(\omega_\nu)} \b)
    \\
    &\leq \| h f \|_{L^2(\omega_\nu)} \b( \| \D v_h \|_{L^2(\omega_\nu)}
    + \| v_h - \barv \|_{L^2(\omega_\nu)} \b).
  \end{align*}

  Summing the estimates over all $\nu$ and estimating the overlaps of
  the patches (the shape regularity of the mesh enters again here;
  this is a standard argument from a posteriori error analysis), we
  deduce that
  \begin{displaymath}
    {\rm T}_3 \lesssim \B( \b\| \D \divv (\beta \pp W(\D \tily))
    \b\|_{L^2} + \b\| \beta h \divv \pp W(\D \tily) \b\|_{L^2} \B) \,
    \b( \| \D v_h\|_{L^2} + \| v_h - \barv
    \|_{L^2} \b).
  \end{displaymath}
  and, finally, employing Lemma \ref{th:inter:Pih'-lemma},
  \begin{equation}
    \label{eq:coarse_T2:130}
    {\rm T}_3 \lesssim \B( \b\| \D \divv (\beta \pp W(\D \tily))
    \b\|_{L^2} + \b\| \beta h \divv \pp W(\D \tily) \b\|_{L^2} \B) \, \|
    \D v_h \|_{L^2}.
  \end{equation}
  Note that we have inserted $\beta$ in $\b\| \beta h \divv \pp W(\D
  \tily) \b\|_{L^2}$ merely to indicate that it is restricted to the
  continuum region.  Inserting the estimate
  \begin{align*}
     |\divv \pp W(\D
  \tily)| &\leq \b| \divv \pp W(\D \tily) - \divv \pp W(\D y_0)\b|  \\
  & \qquad +
  \b| \divv \pp W(\D y_0) - \divv \pp^2 W(\mA) : (\D y_0 - \mA) \b| \\
  &\lesssim | \D^2 \tilu| + |\D^2 y_0|\, |\D y_0 - \mA|.
  \end{align*}
  into \eqref{eq:coarse_T2:130} yields the stated result.
\end{proof}

We can now combine the foregoing results to arrive at the complete
coarsening error estimate.

\begin{lemma}[B-QCE coarsening error]
  \label{th:bqce_final_coarse_err_est}
  Under the conditions of Theorem \ref{th:cons:bqce},
  \begin{align}
    \notag
    \< \del\Ebh(\Pi_h y_h), v_h \> - \< \del\Ea(y), (\Pi_h' v_h)^\ast \>
    &\lesssim
    \B( \errba(y) + \errcb(y) + \errbqce(y) + \|
    \mR^\beta(\tily) \|_{L^2} \B) \| \D v_h \|_{L^2} \\
    &\qquad \text{ for all } v_h \in \Us_h.
    \label{eq:bqce_final_coarse_err_est}
  \end{align}
\end{lemma}
\vspace{-4mm}
\begin{proof}
  Using Lemma \ref{th:bqce_coarse_err_T1_T2} and Lemma
  \ref{th:bqce_coarse_err_T3}, the bound
  \begin{displaymath}
    |\D^2(\beta \pp W(\D \tily))| \lesssim
    |\beta \D^3 \tily| + |\D\beta \D^2 \tily| + | \D^2 \beta|,
  \end{displaymath}
  and the estimate
    \begin{displaymath}
    {\rm T}_4 \leq \| R^\beta \|_{L^2} \| \D \barv \|_{L^2} \lesssim \| R^\beta \|_{L^2} \| \D v_h \|_{L^2},
  \end{displaymath}
  where we employed Lemma~\ref{th:inter:Pih'-lemma}, we obtain the
  result.
\end{proof}

\subsection{B-QCE modelling error estimate}
\label{sec:moderr:bqce_moderr}
To complete the B-QCE consistency error analysis it remains to provide
a sharp bound on the B-QCE stress error $\mR^\beta$, which is defined
in \eqref{eq:inter:bqce_stress_error}.

\begin{lemma}
  \label{th:bqce_moderr_Rbest_est_pointwise}
  Let $\epsilon > 0$ and $z \in C^{2,1}(\nu_x)$ with $\|\D z - \D
  \tily \|_{L^\infty} \leq \epsilon$, then
  \begin{equation}
    \label{eq:bace_moderr_Rbeta_est}
    \b|\mR^\beta(z; x)\b| \leq C \b( \| \D^2\beta
    \|_{L^\infty(\nu_x)} + | \D\beta(x)|\,|\D^2 z(x)| + \| \D^3
    z \|_{L^\infty(\nu_x)} + \| \D^2 z \|_{L^\infty(\nu_x)}^2\b),
  \end{equation}
  where $C$ depends on $M_{\epsilon}^{(\bfrho)}(y), \bfrho \in \Rg^j, j
  = 1, \dots, 3$.
\end{lemma}
\begin{proof}
  Throughout the proof we define $V_{\xi,\rho} := V_{,\rho}(Dz(\xi))$
  and $\bar V_{,\rho} := V_{,\rho}(\D_{\Rg} z(x) )$. Further, we
  define $\beta \equiv \beta(x)$ and $\D\beta \equiv
  \D\beta(x)$. Finally, we denote
  \begin{displaymath}
    \epsilon_j := \| \D^j z \|_{L^\infty(\nu_x)}, \quad \text{and}
    \quad
    \delta_j := \| \D^j \beta \|_{L^\infty(\nu_x)}.
  \end{displaymath}

  We begin by noting that, since $\Rg$
  and the support of $\omega_\rho$ are both bounded, the sum over
  $\xi$ in the definition of $\mR^\beta$
  \begin{align*}
    \mR^\beta(z; x) &:= \beta(x) \dW(\D z(x)) - \sum_{\xi \in \L}
    \beta(\xi) \sum_{\rho \in \Rg} \b[ V_{\xi,\rho} \otimes \rho \b]
    \ww_\rho(\xi-x)
  \end{align*}
  is only over a bounded set. Therefore, we can insert the
  expansion~\eqref{eq:cons:expansion_V_2ord} to obtain
  \begin{align}
    \notag
    \mR^\beta(z; x) &:= \beta(x) \dW(\D z(x))  \\
    \notag
    & \qquad - \sum_{\xi \in
      \L} \beta(\xi) \sum_{\rho \in \Rg} \B[\B( \bar{V}_{,\rho} +
    \sum_{\vsig \in \Rg} \bar{V}_{,\rho\vsig} (D_\vsig z - \D_\vsig
    z) + O(\epsilon_2^2) \B) \otimes \rho \B]
    \ww_\rho(\xi-x) \\
    \notag
    &= \beta(x) \dW(\D z(x)) - \sum_{\xi \in \L} \beta(\xi)
    \sum_{\rho \in \Rg} \b[ \bar{V}_{,\rho} \otimes \rho \b] \ww_\rho(\xi-x)
     \\
    \notag
     & \qquad - \sum_{\xi \in \L} \beta(\xi) \sum_{\rho,\vsig \in \Rg} \b[\bar{V}_{,\rho\vsig} (D_\vsig z -
    \D_\vsig z) \b]\otimes \rho \, \ww_\rho(\xi-x) +
    O(\epsilon_2^2) \\
    \label{eq:bqce_moderr:100}
    &=: {\rm T}_1 - {\rm T}_2 + O(\epsilon_2^2).
  \end{align}

  We expand $\beta(\xi) = \beta + \D \beta \cdot (\xi-x) +
  O(\delta_2)$, and employ \eqref{eq:stress:phirho_Prop0}
   to estimate
  \begin{align*}
    {\rm T}_1 &=  \beta \dW(\D z(x)) - \sum_{\xi \in \L} \B(
    \beta + \D \beta \cdot (\xi-x) \B)
    \sum_{\rho \in \Rg} \b[ \bar{V}_{,\rho} \otimes \rho \b]
    \ww_\rho(\xi-x) + O(\delta_2) \\
    &= \beta \dW(\D z(x)) - \beta \sum_{\rho \in \Rg} \b[
    \bar{V}_{,\rho} \otimes \rho \b]
     \\
     & \qquad - \D\beta \cdot
    \sum_{\rho \in \Rg} \b[ \bar{V}_{,\rho} \otimes \rho \b] \sum_{\xi
      \in \L} (\xi-x)
    \ww_\rho(\xi-x)
    + O(\delta_2).
  \end{align*}
  Since $\sum_{\rho \in \Rg} \b[ \bar{V}_{,\rho} \otimes \rho \b] =
  \dW(\D z(x))$ and $\sum_{\xi \in \L} (\xi-x) \ww_\rho(\xi-x) = -
  \smfrac12 \rho$ by~\eqref{eq:stress:phirho_Prop1}, we further obtain
  \begin{align*}
    {\rm T}_1 &= -\frac12 \sum_{\rho \in \Rg} \b[ \bar{V}_{,\rho}
    \otimes \rho \b] (\D\beta \cdot \rho) + O(\delta_2) = O(\delta_2),
  \end{align*}
  where the sum over $\Rg$ cancels due to the point symmetry
  assumption \eqref{eq:point_symmetry}.

  To estimate ${\rm T}_2$ we expand $\beta$ and use
  expansion \eqref{eq:inter:strain_expansion},
\eqref{eq:stress:phirho_Prop0}, and
  \eqref{eq:stress:phirho_Prop1} to obtain
  \begin{align*}
    {\rm T}_2 &= \sum_{\xi \in \L} \B( \beta + \D\beta \cdot
    (\xi-x)\B) \sum_{\rho,\vsig \in \Rg} \b[\bar{V}_{,\rho\vsig}
    (\D_\vsig \D_{\xi-x} z + \smfrac12 \D_\vsig^2 z) \b]
    \otimes \rho \, \ww_\rho(\xi-x) + O(\delta_2 + \eps_3)
    \\
    &= \beta \sum_{\rho,\vsig \in \Rg}
    \b[\bar{V}_{,\rho\vsig} (-\smfrac12 \D_\vsig \D_{\rho} z + \smfrac12
    \D_\vsig^2 z) \b] \otimes \rho + O\b(| \D\beta |\, |\D^2
    z|\b) + O(\delta_2+\eps_3)
  \end{align*}
  Using again \eqref{eq:point_symmetry} we observe that the sum over
  $\rho,\vsig \in \Rg$ cancels, and hence we obtain that $|{\rm T}_2|
  \lesssim | \D\beta |\, |\D^2 z| + \delta_2 + \eps_3$.

  Combining this with the estimate for ${\rm T}_1$, we obtain the
  stated result.
\end{proof}




We now convert the pointwise estimate \eqref{eq:bace_moderr_Rbeta_est}
into a global estimate.

\begin{lemma}
  \label{th:bqce_moderr_Rbest_est_glob}
  Under the conditions of Theorem \ref{th:cons:bqce}, we have
  \begin{equation}
    \label{eq:bace_moderr_Rbeta_est_glob}
    \b\|\mR^\beta(\tily; x)\b\|_{L^2} \leq C\b( \| \D^2\beta \|_{L^2} +
    \|\D\beta\|_{L^\infty} \|\D^2\tily\|_{L^2(\Omb)} + \| \D^3 \tily
    \|_{L^2(\Omc)} + \| \D^2 \tily \|_{L^4(\Omc)}^2 \b),
  \end{equation}
  where $C$ depends on $M_{\epsilon}^{(\bfrho)} (y), \bfrho \in \Rg^j, j
  = 1, \dots, 3$.
\end{lemma}

\begin{proof}
  %
  The main point of this proof is to use the inverse estimates from
  \S~\ref{eq:inverse_ests} to obtain $L^q$-type bounds from the
  $L^\infty$ bounds provided by Lemma
  \ref{th:bqce_moderr_Rbest_est_pointwise}.

  Let $r(x) := \mR^{\tilde{\beta}}(\hat{y}^x; x)$ and $\mF(x) := \D
  \hat{y}^x(x)$, then we begin by estimating
  \begin{align}
    \notag
    \| \mR^\beta(\tily) - r \|_{L^2}
    &\leq \| \beta \pp W(\D \tily) - \tilde\beta\pp W(\mF)
    \|_{\LL^2} \\
    \notag
    &\lesssim \b( \|\beta - \tilde\beta \|_{L^2} \b) + \b\| \beta \b( \pp W(\D \tily) - \pp
    W(\mF) \b\|_{L^2} \\
    \label{eq:Rbest_glob_10}
    &\lesssim \b( \| \D^2 \beta \|_{L^2} + \| \beta (\D\tily - \mF) \|_{L^2} \b),
  \end{align}
  where we used $\|\beta - \tilde\beta \|_{L^2} \leq C \| \D^2 \beta
  \|_{L^2}$.  Let $x \in {\rm supp}(\beta) \cap Q^x$ where $Q^x = \xi
  + [0, 1)^d$, and let $\xi \in \L$. Then
  \begin{align*}
    |\D\tily - \mF(x)| &\leq \| \D\tily - \D \hat{y}^x \|_{L^2(Q)} 
    \lesssim \| \D^3 \tily \|_{L^2(Q)}.
  \end{align*}
  Integrating over $x$ yields
  \begin{equation}
    \label{eq:Rbest_glob_aux_15}
    \int \beta(x)^2 |\D\tily - \mF(x)|^2 \dx \leq \int \beta(x) \|
    \D^3 \tily \|_{L^2(Q^x)}^2 \dx \lesssim \| \D^3 \tily \|_{L^2(\Omc)}^2,
  \end{equation}
  using an argument analogous to the one in
  \cite[App.A]{2012-ARMA-cb}. Together with \eqref{eq:Rbest_glob_10}
  we obtain
  \begin{equation}
    \label{eq:Rbest_glob_20}
    \| \mR^\beta(\tily) - r \|_{L^2} \lesssim \b( \| \D^2 \beta \|_{L^2}
    + \| \D^3 \tily \|_{L^2(\Omc)} \b).
  \end{equation}

  Using Lemma \ref{th:bqce_moderr_Rbest_est_pointwise} with $\beta$ replaced
  with $\tilde\beta$ and $z = \hat{y}^x$, defining $\nu_x^\beta :=
  \nu_x \cap {\rm supp}(\nabla\tilde{\beta})$, and recalling
  \eqref{eq:haty_invest}, we obtain
  \begin{align*}
    |r(x)|^2 &\lesssim \B(\| \D^2
    \tilde\beta \|_{L^\infty(\nu_x)}^2 + \|\D\tilde\beta\|_{L^\infty}^2
    \|\D^2\hat{y}^x\|_{L^\infty(\nu_x^\beta)}^2 + \| \D^3 \hat{y}^x \|_{L^\infty(\nu_x)}^2 + \|
    \D^2 \hat{y}^x \|_{L^\infty(\nu_x)}^4  \B) \\
    &\lesssim \B(\| \D^2
    \tilde\beta \|_{L^2(\nu_x)}^2 + \|\D\tilde\beta\|^2_{L^\infty}
    \|\D^2\hat{y}^x\|^2_{L^2(\nu_x^\beta)} + \| \D^3 \hat{y}^x \|_{L^2(\nu_x)}^2 + \|
    \D^2 \hat{y}^x \|_{L^4(\nu_x)}^4  \B).
  \end{align*}
  Using Lemma~\ref{th:prelims:defn_smoothint}, the results of \S~\ref{eq:inverse_ests},
  and techniques similar to those used to prove \eqref{eq:Rbest_glob_aux_15},
  we deduce that
    \begin{displaymath}
    \|r\|_{L^2} \lesssim \| \D^2 \beta \|_{L^2} +
    \|\D\beta\|_{L^\infty} \|\D^2\tily\|_{L^2(\Omb)} + \| \D^3 \tily
    \|_{L^2(\Omc)} + \| \D^2 \tily \|_{L^4(\Omc)}^2.
  \end{displaymath}
  Together with \eqref{eq:Rbest_glob_20}, this yields the desired
  result.
\end{proof}

\begin{proof}[Proof of Theorem \ref{th:cons:bqce}]
  The result follows upon combining Lemma
  \ref{th:bqce_final_coarse_err_est} and Lemma
  \ref{th:bqce_moderr_Rbest_est_glob}.
\end{proof}

\subsection{B-QCE energy error estimate}
\label{sec:bqce_en_err}
We assume that all conditions of Theorem \ref{th:error} hold.  Let
$\ya$ be a solution to \eqref{eq:min_atm_exact} and let $y^\bqce_h$ be
the solution to \eqref{eq:min_bqce} guaranteed by
Theorem~\ref{th:error}. For ease of notation, we write $y := y^\a$ and
$y_h := y^\bqce_h$. Further, we define
\begin{displaymath}
  V'(Dy) := V(Dy) - V(Dy_0) \qquad \text{and} \qquad
  W'(\D y) := W(\D y ) - W(\D y_0).
\end{displaymath}

Let
\begin{displaymath}
  \tilde\E := \sum_{\ell \in \L} (1-\beta(\ell)) V'(Dy(\ell)) + \int_{\R^d} [Q_h
  \beta] W'(\D \tily) \dx,
\end{displaymath}
then we split the energy error into
\begin{align*}
  \Ea(y) - \Ebh(y_h) &= \b[ \Ea(y) - \tilde\E \b] + \b[ \tilde\E - \Ebh(\Pi_h
  y) \b] + \b[ \Ebh(\Pi_h y) - \Ebh(y_h) \b] \\
  &=: T_1 + T_2 + T_3.
\end{align*}
Since $y_h$ is a minimiser we obtain
\begin{equation}
  \label{eq:bqce_enerr_T3}
  |T_3| \lesssim \| \D \Pi_h u - \D u_h \|_{L^2}^2,
\end{equation}
which we already estimated in Theorem \ref{th:error}.

\subsubsection{Estimate for $T_{1}$}
The term $T_1$ contains the main ``modelling error'' contribution. For
$f : \R^d \to \R$ let $I_1 f := \bar{f}$ denote the P1 nodal interpolant
with respect to the atomistic mesh $\TT$. Then, using the fact that
\begin{displaymath}
  \sum_{\xi \in \L} \beta(\xi) W'(\D\tily(\xi)) = \int_{\R^d}
I_1[\beta W'(\D \tily)] \dx,
\end{displaymath}
we rewrite $T_1$ as
\begin{align*}
  T_1 &= \sum_{\xi \in \L} \beta(\xi) V'\b( Dy(\xi) \b) - \int [Q_h
  \beta] W'(\D \tily) \dx \\
  &= \sum_{\xi \in\L} \beta(\xi) \B( V'\b(Dy(\xi) \b) - W'\b(\D
  \tily(\xi)\b)  \B) + \int \B( [Q_h \beta] W'(\D \tily) - I_1[ \beta W'(\D
  \tily) ] \B) \dx \\
  &=: T_{1,1} + T_{1,2}.
\end{align*}

$T_{1,2}$ is essentially a quadrature error estimate, since both the
integrals $\int Q_h [\beta W'(\D \tily)] \dx$ and $\int I_1 [\beta W'(\D
\tily)] \dx$ are second-order quadrature approximations to $\int \beta
W'(\D \tily) \dx$:
\begin{align}
  \notag
  |T_{1,2}| &\lesssim \| \D^2\beta\|_{L^2}  \|\D \tilu \|_{L^2(\Omc)}
  + \|\D\beta \|_{L^\infty} \| \D\tilu \|_{L^2(\Omc)} \|\D^2
  y_0\|_{L^2(\Omc)} + \| \D\beta\|_{L^2} \| \D^2 \tilu \|_{L^2(\Omc)} \\
  \label{eq:bqce_enerr_T12_est}
  & \qquad \quad + \| \D\tilu \|_{L^2(\Omc)} \b( \| \D^3 y_0
  \|_{L^2(\Omc)} + \| \D^2 y_0 \|_{L^4(\Omc)}^2 \b) + \| \D^2
  \tilu
  \|_{L^2(\Omc)}^2  \\
  \notag
  & \qquad \quad + \| \D^2 \tilu \|_{L^2} \| \D^2 y_0 \|_{L^2(\Omc)}
  + \| \D^3 \tilu \|_{L^1(\Omc)}.
\end{align}
We will later see that most of these terms are dominated by other
terms occuring in the energy error estimate.

\begin{proof}[Proof of \eqref{eq:bqce_enerr_T12_est}]
  Fix an atomistic element $T \in \TT$.
  If $\beta \not\equiv 1$ in $T$, then $T \in \T$ as well,
  so $Q_h\beta = Q_1\beta$, where $Q_1$ denotes the P0 midpoint nodal interpolant
  with respect to the atomistic mesh $\TT$.
  In the other case, where $\beta \equiv 1$ in $T$, we also have $Q_h\beta = Q_1\beta = 1$.

  We estimate the integral defining ${\rm T}_{1,2}$ restricted to
  $T$; call it
  \begin{displaymath}
    {\rm T}_{1,2}^T := \int_T \B( [Q_1\beta] W'(\D\tily) - I_1[\beta
    W'(\D\tily)] \B)\dx.
  \end{displaymath}
  First, we replace $\tily$ with $\hat{y}^T := \hat{y}^{x_T}$, $y_0$
  with $\hat{y}_0^T := \hat{y}_0^{x_T}$ where $x_T$ is the barycentre
  of $T$, and $\beta$ with $\tilde\beta$. Also, let $y_0^T :=
  y_0^{x_T}$ and $W'(\D \hat{y}^T) = W(\D\hat{y}^T) -
  W(\D\hat{y}^T_0)$. Then, a brief computation shows that
  \begin{align*}
    \b\| W'(\D y^T) - W'(\D \hat{y}^T) \b\|_{L^\infty} &\lesssim \b\| |\D y_0^T - \D
    \hat{y}_0^T| \, |\D \tilu| \b\|_{L^\infty(T)} \lesssim \| \D \tilu
    \|_{L^2(T)}\| \D^3 y_0
    \|_{L^2(T)}, \\
    \| (\beta-\tilde\beta) W'(\D \tily) \|_{L^\infty(T)} &\leq \|
    \D^2\beta \|_{L^2(T)} \| \D \tilu \|_{L^2(T)},
  \end{align*}
  and hence,
  \begin{align*}
    \b|{\rm T}_{1,2}^T\b| &\lesssim \bg|\int_T \B( [Q_1\tilde\beta]
    W'(\D\hat{y}^T) - I_1[\tilde\beta W'(\D\hat{y}^T)] \B)\dx\bg| \\
    & \qquad \qquad
    + \| \D \tilu \|_{L^2(T)} ( \| \D^3 y_0
    \|_{L^2(T)} + \|\D^2\beta\|_{L^2(T)} ) \\
    &=: \b| \hat{\rm T}_{1,2}^T \b| +
     \| \D \tilu \|_{L^2(T)} ( \| \D^3 y_0
    \|_{L^2(T)} + \|\D^2\beta\|_{L^2(T)}),
  \end{align*}

  We estimate the term $\hat{\rm T}_{1,2}^T$ as follows:
  \begin{align*}
    \b|\hat{\rm T}_{1,2}^T\b|
    &\leq \bg|\int_T [Q_1\tilde\beta] \b( W'(\D y^T) - Q_1[W'(\D y^T)]
    \b) \dx\bg| + \bg|\int_T  (Q_1 - I_1)[\tilde\beta  W'(\D \tily)] \dx\bg| \\
    &\lesssim \| \D^2 [\tilde\beta W'(\D \hat{y}^T)]  \|_{L^\infty(T)} \\
    &\lesssim \| (\D^2 \tilde\beta) W'(\D\hat{y}^T) \|_{L^\infty(T)} + \|
    \D\tilde\beta \otimes\D W'(\D\hat{y}^T) \|_{L^\infty(T)} + \| \tilde\beta \D^2
    W'(\D\hat{y}^T) \|_{L^\infty(T)} \\
    &\lesssim \| \D^2\tilde\beta\|_{L^\infty} \| \D\tilu \|_{L^\infty(T)} +
    \|\D\tilde\beta\|_{L^\infty} \B(\| \D\tilu \D^2
    \hat{y}_0 \|_{L^\infty(T)} + \| \D^2 \tilu \|_{L^\infty(T)} \B) \\
    & \qquad \qquad +  \B(
    \b\| |\D \tilu|\, |\D^2 \hat{y}_0|^2
    \b\|_{L^\infty(T)}
    + \b\| |\D^2 \tilu| \, |\D^2 \hat{y}_0| \b\|_{L^\infty(T)}
    + \b\| |\D^2 \tilu|^2 \b\|_{L^\infty(T)} \\
    & \qquad \qquad \qquad
    + \| \D^3 \tilu \|_{L^\infty(T)} + \b\| |\D\tilu|\,|\D^3 \hat{y}_0| \b\|_{L^\infty(T)}
    \B),
  \end{align*}
  where we used the fact that $\D\hat{y}^T - \D\hat{y}_0^T = \D\tilu$
  and identities along the lines of
  \begin{align*}
    \D W'(\D \hat{y}^T) &= \D \b( W(\D \hat{y}^T) - W(\D \hat{y}_0^T)
    \b) = \dW(\D\hat{y}^T) : \D^2 \hat{y}^T - \dW(\D\hat{y}_0^T) : \D^2
    \hat{y}_0^T \\
    &= \b(\dW(\D\hat{y}^T) - \dW(\D\hat{y}_0^T) \b): \D^2 \hat{y}_0^T
    + \dW(\D\hat{y}^T) : \b( \D^2 \hat{y}^T - \D^2 \hat{y}_0^T \b),
  \end{align*}
  and its lower and higher order analogues.

  Applying suitable inverse inequalities
  (cf. \S~\ref{eq:inverse_ests}), summing over $T \in \TT$, and being
  careful to only collect those terms for that actually occur in a
  given element yields \eqref{eq:bqce_enerr_T12_est}.
\end{proof}

To estimate $T_{1,1}$ we perform a basic Taylor expansion, using the
tools developed in \S~\ref{sec:int:strain} and
\S~\ref{sec:int:V_bounds}.

\begin{lemma}
  \label{th:potential_expansion}
  Let $\xi \in \L \setminus B_{\Ra}$, and $y \in \Ys$, then
  \begin{equation}
    \label{eq:enerr_potential_expansion}
    \begin{split}
      \b| V'(Dy(\xi)) - W'(\D \tily(\xi)) \b| &\lesssim \B( \| \D^3 \tilu
      \|_{L^1(\nu_\xi)} + \| \D^2 \tilu \|_{L^2(\nu_\xi)}^2 +
      \| \D^2 \tilu \|_{L^2(\nu_\xi)} \| \D^2 y_0\|_{L^2(\nu_\xi)} \\
      & \qquad\qquad + \| \D\tilu \|_{L^2(\nu_\xi)} \b( \| \D^3 y_0
      \|_{L^2(\nu_\xi)} + \| \D^2 y_0 \|_{L^4(\nu_\xi)}^2 \b) \B).
    \end{split}
  \end{equation}
\end{lemma}
\begin{proof}
   All derivatives and finite differences below are evaluated at $\xi$,
   so we omit the argument, writing $Du$ for $Du(\xi)$, for example.
  Let $z := \hat{y}^\xi$ and $z_\theta := \hat{y}_0^\xi + \theta
  \tilu$, so that $z = z_1$,  $V'(Dz) = V(Dz) - V(Dz_0)$ and $W'(\D z)
  = W(\D z) - W(\D z_0)$. Then
  \begin{align*}
    V'(Dz) - W'(\D z) &= \int_{\theta = 0}^1 \B(\< \del V(D z_\theta),
    Du \> - \< \dW(\D z_\theta), \D \tilu \> \B)
    \, \rm{d} \theta.
  \end{align*}
  Expanding $\< \del V(D z_\theta), Du \>$ analogously to the proof of
  Lemma \ref{th:bqce_moderr_Rbest_est_pointwise}, with $\eps_j := \|
  \D^j z_\theta \|_{L^\infty(\nu_\xi)}$ and $\tilde{\eps}_j := \| \D^j
  \tilu \|_{L^\infty(\nu_\xi)}$,  we obtain
  \begin{align*}
    \< \del V(D z_\theta), Du \> &= \sum_{\rho \in \Rg} V_{,\rho}(D z_\theta)
    D_\rho u \\
    &= \sum_{\rho \in \Rg} \bg( V_{,\rho} + \sum_{\vsig \in \Rg}
    V_{,\rho\vsig} \smfrac12 \D_\vsig^2 z_\theta + O(\eps_3+\eps_2^2)
    \bg) D_\rho u \\
    &= \sum_{\rho \in \Rg} V_{,\rho} \b(\D_\rho \tilu + \smfrac12
    \D_\rho^2 \tilu\b)
    + \sum_{\rho,\vsig \in \Rg}
    \< V_{,\rho\vsig}  \D_\rho \tilu, \smfrac12 \D_\vsig^2 z_\theta\>
    \\
    & \qquad \qquad
    + O(\tilde{\eps}_3) + O(\eps_2 \tilde{\eps}_2) + O\b(\tilde{\eps}_1(\eps_3+\eps_2^2)\b).
  \end{align*}
  We now observe that $\sum_{\rho \in \Rg} V_{,\rho} \D_\rho\tilu =
  \dW(\D z_\theta) : \D \tilu$, and that, due to the point symmetry
  \eqref{eq:point_symmetry}, both
  \begin{align*}
    \sum_{\rho \in \Rg} V_{,\rho} \smfrac12 \D_\rho^2 \tilu = 0 \qquad
    \text{and} \qquad
    \sum_{\rho,\vsig \in \Rg} \<V_{,\rho\vsig} \D_\rho\tilu, \D_\vsig^2
    z_\theta \> = 0.
  \end{align*}

  We combine the foregoing calculations to obtain
  \begin{align*}
    \b|V'(Dz) - W'(\D z)\b| &\leq C\B( \| \D^3 \tilu \|_{L^\infty(\nu_\xi)} +
    \| \D^2 \tilu \|_{L^\infty(\nu_\xi)} \| \D^2 z_0
    \|_{L^\infty(\nu_\xi)} + \| \D^2 \tilu \|_{L^\infty(\nu_\xi)}^2  \\
    & \qquad + \| \D \tilu \|_{L^\infty(\nu_\xi)} \b( \| \D^3 z_0
    \|_{L^\infty(\nu_\xi)} + \| \D^2 z_0 \|_{L^\infty(\nu_\xi)}^2 \b)\B).
  \end{align*}

  Using appropriate inverse estimates, and incorporating the error $\D
  z_\theta - \D\tily$, similarly (e.g.) as in the proof of
  \eqref{eq:bqce_enerr_T12_est} (this yields additional $\| \D \tilu
  \|_{L^2} \| \D^3 \tily_0 \|_{L^2}$ terms), we obtain the stated
  result.
\end{proof}

Summing \eqref{eq:enerr_potential_expansion} over all $\xi \in \L$
with $\beta(\xi) > 0$ it is straightforward now to prove that
\begin{equation}
  \label{eq:enerr:estimate_T11}
    \begin{split}
          |T_{1,1}|
          &\leq C \B( \| \D^3 \tilu
      \|_{L^1(\Omc)} + \| \D^2 \tilu \|_{L^2(\Omc)}^2 + \| \D^2 \tily
      \|_{L^2(\Omc)} \| \D^2 \tily \|_{L^2(\Omc)}  \\
      & \qquad\qquad + \| \D\tilu \|_{L^2(\Omc)} \b(\| \D^3 y_0
      \|_{L^2(\Omc)} + \| \D^2 y_0 \|_{L^4(\Omc)}^2 \b) \B).
    \end{split}
\end{equation}

This completes the estimate for ${\rm T}_1$

\subsubsection{Estimate for ${\rm T}_2$}
We begin by recalling that $\T$ and $\Pi_h$ are defined in such a way
that $\Pi_h y(\xi) = y(\xi)$ in a sufficiently large neighbourhood
so that
\begin{align*}
  {\rm T}_2 &= \int \B([Q_h \beta] W'(\D \tily) - Q_h\b[\beta W'(\D\Pi_h
  y)\b] \B) \dx \\
  &= \int [Q_h\beta] \b(  W'(\D \tily) - Q_h W'(\D\Pi_hy ) \b) \dx \\
  &= \int [Q_h\beta] \b(  W'(\D \tily) - W'(\D \Pi_hy ) \b) \dx \\
  & \qquad +  \int [Q_h\beta] \b(  W'(\D \Pi_h y) - Q_h W'(\D\Pi_hy )
  \b) \dx \\
  &=: {\rm T}_{2,1} + {\rm T}_{2,2}.
\end{align*}
The term $T_{2,1}$ is an approximation error, while $T_{2,2}$ is a
quadrature error.

First, we prove that
\begin{equation}
  \label{eq:bqce_enerr_T21_est}
  |T_{2,1}| \lesssim \b\| \b( |\D\tilu| + |\D\ulin|^2\b) (\D \tilu - \D \Pi_h u)
  \b\|_{L^1(\Omc)} + \|\D^3\tilu\|_{L^1(\Omc)} + \| \D \tilu - \D
  \Pi_h u \|_{L^2(\Omc)},
\end{equation}
where we set $\ulin \equiv 0$ in the case (pPt).

\begin{proof}[Proof of \eqref{eq:bqce_enerr_T21_est}]
  We first note that, with $e_h := \D\tilu - \D\Pi_h u$ we have
  \begin{align*}
    |T_{2,1}| &= \bg|\int [Q_h\beta] \b( W(\D \tily) - W(\D \tily - \D e_h)\b) \dx\bg| \\
    &\leq \bg|\int [Q_h\beta] \b( \partial W(\D\tily) \D e_h \b) \dx\bg| + C \|
    \D e_h \|_{L^2}^2 \\
    &=: |T_{2,1}'| +  C \|
    \D e_h \|_{L^2}^2.
  \end{align*}

  Let $e_h = [\D\tilu - \D\baru]+[\D\baru-\D\Pi_h u] =: e_h' + e_h''$,
  then $\|[Q_h\beta] \D e_h'\|_{L^1} \lesssim \|\D^3\tilu\|_{L^1(\Omc)}$
  and $e_h'' = 0$ in $\Omb$, hence
  \begin{align*}
    |T_{2,1}'|&\leq
    \bg|\int [Q_h\beta] \b( \partial W(\D\tily) \D e_h'' \b) \dx\bg| + C
    \| \D^3 \tilu \|_{L^1(\Omc)} \\
    &= \bg|\int [\dW(\D\tily) - \mS^{\rm lin} \b] \D e_h'' \dx\bg| + C
    \| \D^3 \tilu \|_{L^1(\Omc)},
  \end{align*}
  where, setting $\ulin \equiv 0$ in the case (pPt),
  \begin{displaymath}
    \mS^{\rm lin} = \dW(\mA) + \bbC:\D\ulin,
  \end{displaymath}
  We can now estimate
  \begin{displaymath}
    \b|\dW(\D\tily) - \mS^{\rm lin} \b| \lesssim 
      |\D\tilu| + |\D\ulin|^2,
  \end{displaymath}
  which yields
  \begin{align*}
    |T_{2,1}'| &\lesssim \B(\b\| \b(|\D\tilu| + |\D\ulin|^2\b) \D e_h''
    \b\|_{L^1(\Omc)}^2 + \| \D^3 \tilu \|_{L^1(\Omc)} \\
    &\lesssim \b\| \b(|\D\tilu| + |\D\ulin|^2\b) \D e_h
    \b\|_{L^1(\Omc)}^2 + \| \D^3 \tilu \|_{L^1(\Omc)},
  \end{align*}
  estimating again that $\| \D e_h' \|_{L^1(\Omc)} \leq \| \D^3 \tilu
  \|_{L^1(\Omc)}$.

  This completes the proof of \eqref{eq:bqce_enerr_T21_est}.
\end{proof}

The final term to complete the estimate for the B-QCE energy error is
$T_{2,2}$, which we can bound by
\begin{align}
  |T_{2,2}| &\lesssim \B( \| h^2 \D^3 y_0 \|_{L^2(\Omc)} + \|h \D^2 y_0
  \|_{L^4(\Omc)}^2 \B) \| \D \Pi_h u \|_{L^2(\Omc)} \\
  &\lesssim \B(\| h^2 \D^3 y_0 \|_{L^2(\Omc)} + \|h \D^2 y_0
  \|_{L^4(\Omc)}^2 \B) \B( \errba(y) + \| \D \tilu \|_{L^2(\Omc)} \B)
  \label{eq:bqce_enerr_T22_est}
\end{align}
The proof of this estimate follows much along the same lines as that
of \eqref{eq:bqce_enerr_T12_est}, exploiting the fact that
\begin{align*}
  &\hspace{-5mm} W'(\D \Pi_h y) - Q_h W'(\D \Pi_h y) \\
  &= \int_{\theta = 0}^1 \B(\dW( \D
  y_0 + \theta \D \Pi_h u) - \dW\b(Q_h[\D y_0]+\theta \D \Pi_h u\b)
  \B) \dd\theta : \D \Pi_h u.
\end{align*}

\subsubsection{Completing the energy error estimate}
\label{sec:bqce_enerr_final}
Combining the estimates \eqref{eq:bqce_enerr_T12_est},
\eqref{eq:enerr:estimate_T11}, \eqref{eq:bqce_enerr_T21_est} and
\eqref{eq:bqce_enerr_T22_est}, ignoring any terms that are dominated
by others, we arrive at
\begin{align}
  \label{eq:bqce_enerr_final:1}
  \b| \Ea(y) - \Ebh(y_h) \b| &\leq C\B\{
  {\errba(y)}^2 + {\errcb}(y)^2 + {\errbqce(y)}^2 \\
  \notag
  & \qquad +
  \| \D^2\beta \|_{L^2} \| \D\tilu
  \|_{L^2(\Omc)} + \| \D\beta \|_{L^2} \| \D^2\tilu \|_{L^2(\Omc)} \\
  \notag
  & \qquad \qquad \qquad
  + \errba(y) \b( \| \D\tilu \|_{L^2(\Omc)} + \| \D \ulin \|_{L^4(\Omc)}^2 \b)
\\
  \notag
  & \qquad + \| \D^3 \tilu \|_{L^1(\Omc)} + \| \D^2\tilu
  \|_{L^2(\Omc)}^2   \\
  \notag
  & \qquad + \B( \| h^2 \D^3 y_0 \|_{L^2(\Omc)} + \|h \D^2 y_0
  \|_{L^4(\Omc)}^2 \B) \B( \errba(y) + \| \D \tilu \|_{L^2(\Omc)} \B) \\
  \notag
  & \qquad + \| \D^2\tilu \|_{L^2(\Omc)} \| \D^2 y_0 \|_{L^2(\Omc)}
     + \| \D\beta\|_{L^\infty} \| \D\tilu \|_{L^2(\Omc)} \| \D^2 y_0 \|_{L^2(\Omc)}
	\B\}.
\end{align}
A slight rearrangement yields the statement of Proposition
\ref{th:bqce-energy_error}.

\clearpage

\subsection{B-QCF Consistency analysis}
\label{sec:bqcf_cons_prf}
\subsubsection{Consistency error estimate, part 1}
\label{sec:bqcf_cons_prf_1}
Recall the definition of the B-QCF operator \eqref{eq:defn_Fbh} and
assume that $y_h(\xi) = y(\xi)$ for $\xi \in \La$, then we have
\begin{align*}
  \< \Fbh(y_h), v_h \> - \< \del\Ea(y), v \> &= \< \del\Ea(y),
  (1-\beta) v_h - v \> + \< \del\Ech(y_h), I_h[\beta v_h] \>.
\end{align*}
Similar to the B-QCE case in \S~\ref{sec:moderr:bqce_moderr}, we
choose a specially adapted test function $v := \Pi_h'' v_h$, as
defined in \eqref{eq:int:cons_bqcf:defn_v},
for the ``weak form'' of the
atomistic force $\< \del\Ea(y), v\>$. That is,
\begin{equation}
  \label{eq:defn_dual_interp_bqcf}
   v = (1-\beta) v_h|_{\L} + w^*, \quad \text{where}
  \quad w(\xi) := \b(\zz \ast w_h\b)(\xi) \text{ and } w_h := I_h[\beta v_h].
\end{equation}
Standard quasi-interpolation error estimates (see
e.g. \cite{Carstensen1999} for an analogous result) yield
\begin{equation}
  \label{eq:estimate_dualint_bqcf}
  \| \D \barw - \D w_h \|_{L^2} + \| \barw - w_h \|_{L^2} \lesssim \| \D
  w_h \|_{L^2}.
\end{equation}

Applying the stress form of $\del\Ea(y)$ in
\eqref{eq:stress:delEa_v2}, with $\Sa = \Sa(y; x)$, we can now compute
\begin{align*}
  & \hspace{-5mm} \< \Fbh(y_h), v_h \> -  \< \del\Ea(y), v \>\\[1mm]
   &= \< \del\Ea(y),
  (1-\beta) v_h - v \> + \< \del\Ech(y_h), I_h[\beta v_h] \>  \\[1mm]
  &= - \< \del \Ea(y), w^* \> +\< \del\Ech(y_h), w_h \> = \< \del\Ech(y_h), w_h \> - \< \del \Ea(y), w^* \>  \\
  &= \int \B[ Q_h\b[\pp W(\D y_h) : \D w_h\b] - \Sa : \D \barw \B] \dx \\
  &= \int Q_h \B[ \pp W(\D y_h) - \pp W(\D \tily) \B] : \D w_h \dx
  +\int \big(Q_h-{\rm Id}\big)\b[\pp W(\D \tily) : \D w_h\b] \dx\\
 &\qquad+\int \pp  W(\D \tily) : \b( \D w_h- \D \barw \b) \dx
 + \int \b( \pp W(\D\tily) - \Sa \b) : \D \barw \dx\\
 &=: T_1+T_2+T_3+T_4.
\end{align*}

Applying Lemma \ref{th:bqce_coarse_err_T1_T2} and
Lemma~\ref{th:bqce_coarse_err_T3} with $\beta \equiv 1$, and
exploiting the fact that ${\rm supp}(w_h), {\rm supp}(\bar{w}) \subset
\Omc$ we obtain
%
\begin{align*}
  |T_1| &\lesssim \| \D u_h - \D\tilu \|_{L^2(\Omc)} \| \D w_h \|_{L^2},
  \\[1mm]
  |T_2| &\lesssim \| h^2 \D^2 \dW(\tily) \|_{L^2(\Omc)} \|\D
  w_h\|_{L^2}, \quad \text{and} \\[1mm]
  |T_3| &\lesssim \| \D \divv \dW(\tily)  \|_{L^2(\Omc)} \| \D w_h \|_{L^2}.
\end{align*}
%

Finally, the fourth term is the Cauchy--Born modelling error estimated
in Lemma \ref{th:cb_model_err} combined with the quasi-interpolation
error estimates in \eqref{eq:estimate_dualint_bqcf}. Applying Lemma
\ref{th:bqce_moderr_Rbest_est_glob} with $\beta \equiv 1$ and
exploiting again that ${\rm supp}(\barw) \subset \Omc$, we obtain
\begin{displaymath}
  |T_4| \lesssim \b( \| \D^3 \tily \|_{L^2(\Omc)} + \| \D^2 \tily
  \|_{L^4(\Omc)}^2 \b) \| \D w_h \|_{L^2}.
\end{displaymath}

Combining the estimates for the terms ${\rm T}_1, \dots, {\rm T}_4$
and then arguing as in Lemma \ref{th:bqce_final_coarse_err_est} we
arrive at
\begin{align}\label{eq:bqcf_consis_error}
  \b| \< \Fbh(y_h), v_h \> - \< \del\Ea(y), v \> \b| 
  \lesssim \b( \errba(y) + \errcb(y) \b) \| \D w_h \|_{L^2}.
\end{align}
In particular, we have proven \eqref{eq:int:bqcf_cons_firststep}.

It now remains to estimate $\| \D w_h \|_{L^2}$, where $w_h =
I_h[\beta v_h]$ in terms of $\| \D v_h \|_{L^2}$.

\subsubsection{The trace inequality}
\label{sec:trace}
Our aim is to prove \eqref{eq:int:est_Dbv_Dv}. For the sake of
argument, suppose $w_h \sim \beta v_h$ (we dropped the interpolant),
so that $\D w_h \sim \beta \D v_h + v_h \otimes \D\beta$. Thus, we
need to estimate $v_h$ in the support of $\D\beta$ (i.e., in the
blending region) in terms of $\D v_h$ in $\Omh$. The key ingredient to
obtain such an estimate is the following trace inequality.

\begin{lemma}
  \label{th:trace}
  Let $d \geq 2$ and $0 < r_0 < r_1$, then
  \begin{align}
    \label{eq:trace_dirichlet}
    \| u \|_{L^2(\partial B_{r_0})}^2 \leq
    C_1 \| \D u \|_{L^2(B_{r_1}\setminus
      B_{r_0})}^2
    \quad \text{for all } u \in H^1(B_{r_1}\setminus B_{r_0}),
    u|_{\partial B_{r_1}} = 0, & \\[1mm]
     \text{where }
    C_1 = \cases{ 2 r_0 \b|\log \smfrac{r_1}{r_0} \b|, & d = 2, \\
      2 r_0 / (d-2), & d \geq 3.
    } &
  \end{align}
\end{lemma}
\begin{proof}
  %
  The result follows from minor modifications to remove the constraint
  $r_1 < 1$ of the proof of \cite[Lemma 5.1]{BQCF}; up
  to \cite[Eq. (5.4)]{BQCF} and choosing $s = r_1$ from
  the beginning.
\end{proof}

\begin{corollary}
  \label{th:vol_trace_ineq}
  Under the conditions of Lemma \ref{th:trace}, we have
  \begin{align*}
    \| v_h \|_{L^2(\Omb)}^2 \leq \b( (\Rb)^2 - (\Ra)^2 \b) C_1' \| \D v_h \|_{L^2(\Omh)}^2 \qquad
    \forall v_h \in \Us_h,& \\
    \qquad \text{where } \qquad C_1' = \cases{ \log\b|\frac{\Ro}{\Ra}
      \b|, &  d = 2, \\
      1, & d = 3.
    }
  \end{align*}
\end{corollary}
\vspace{-5mm}
\begin{proof}
  Recalling that $\Omb \subset B_{\Rb} \setminus B_{\Ra}$ we write
  \begin{displaymath}
    \|v_h\|_{L^2(\Omb)}^2 \leq \int_{r = \Ra}^{\Rb} \| v_h
    \|_{L^2(\partial B_r)}^2 \dr.
  \end{displaymath}
  Applying \eqref{eq:trace_dirichlet} yields the stated result.
\end{proof}

\begin{proof}[Proof of Lemma \ref{th:int:est_Dbv_Dv}]
  If $T \in \mathcal{T}_h$ with $\beta|_T \equiv 1$, then $I_h[\beta
  v_h] = v_h$, and hence $\| \D I_h[\beta v_h]\|_{L^2(T)} \leq \| \D
  v_h \|_{L^2(T)}$.

  Conversely, if $\beta|_T \not\equiv 1$, then $h_T \lesssim 1$ and
  hence standard nodal interpolation error
  estimates~\cite{Ciarlet:1978} imply
  \begin{align*}
      \| \D I_h[\beta v_h] \|_{L^2(T)} &\le \|\D I_h[\beta v_h]-\D \left(\beta v_h\right)\|_{L^2(T)}+\|\D \left(\beta v_h\right)\|_{L^2(T)}\\
      &\lesssim \|\D^2 \left(\beta v_h\right)\|_{L^2(T)}+\|\D \left(\beta v_h\right)\|_{L^2(T)}.
  \end{align*}
  Since $v_h|_{T}\in \rm{P}1(T)$, so $\D^2 v_h=0$, for each such
  element $T$ we have
  \begin{align*}
    \| \D I_h[\beta v_h] \|_{L^2(T)}
      &\lesssim \|\D^2 \beta \|_{L^{\infty}(T)}\|v_h\|_{L^2(T)}
      +2\|\D \beta \|_{L^{\infty}(T)}\|\D v_h\|_{L^2(T)} \\
      & \qquad \qquad \qquad + \|\beta \D  v_h\|_{L^2(T)} + \|v_h \otimes \D
\beta\|_{L^2(T)} \\
      & \lesssim \|\D \beta\|_{W^{1,\infty}(T)} \|v_h\|_{L^2(T)}
      +(1+ \|\D \beta \|_{L^{\infty}(T)}) \|\D v_h\|_{L^2(T)} .
  \end{align*}
  Recall that $\Omb$ is constructed in such a way that ${\rm
    supp}\nabla \beta \cap T \neq\emptyset$ implies that $T \subset
  \Omb$. Thus, summing over all $T \subset \Omb$, and also recalling
  that $\|\D\beta\|_{L^\infty} \lesssim 1$ and then applying Corollary
  \ref{th:vol_trace_ineq}, we obtain
  \begin{align}%
    \notag
    \| \D I_h[\beta v_h] \|_{L^2}
    &\lesssim \|\D \beta\|_{W^{1,\infty}}
    \|v_h\|_{L^2(\Omega^\beta)}
    + \|\D v_h\|_{L^2}  \\
    \label{eq3:bqcf_wh_est}
    &\lesssim \B( C_1' \b[(\Rb)^2-(\Ra)^2\b]^{1/2}  \|
    \D\beta\|_{W^{1,\infty}} + 1 \B) \| \D v_h \|_{L^2},
  \end{align}
  where $C_1'$ is the constant from Lemma \ref{th:vol_trace_ineq}.

  Recall now that in \eqref{eq:blend_stab:radii} we assumed that the
  blending function $\beta$ satisfies $\| \D^{j} \beta \|_{L^\infty}
  \lesssim (\Rb)^{-j}$ for $j = 1, 2$. Inserting this assumption into
  \eqref{eq3:bqcf_wh_est} finally completes the proof of Lemma
  \ref{th:int:est_Dbv_Dv}.
\end{proof}

\section{Stability Proofs}
\label{sec:stab_prfs}
\subsection{BQCE stability}
\label{sec:prf_bqce_stab}
\begin{proof}[Proof of Lemma \ref{th:int:stab_bqce_lem1}]
  Assume, for contradiction, that there exists a sequence of B-QCE
  approximations, characterized by $\beta_n$, $\Ts_{h,n}$, $v_{h,n}
  \in \Us_{h,n}$, etc., with $\Ra_n\to\infty$, as well as test
  functions $v_{h,n}$ satisfying $\|\nabla v_{h,n}\|_{L^2}^2=1$ and
\[
\lim_{n\to\infty} \< \ddel [\E^{\beta}_{h,n}(\Pi_{h,n} y)+ \Pa(\Pi_{h,n} y)] v_{h,n}, v_{h,n} \> =: \tilde{\gamma}^\beta_h
<
\min\big\{ \gamma^\a(y),
    \gamma^\beta_h(\mA x) \big\}
.
\]

In what follows, we will drop the index $h$ in $\Us_{h,n}$,
$\Ts_{h,n}$, $\Pi_{h,n}$, $\E^{\beta}_{h,n}$, and so forth.

Upon extracting a subsequence (which is still denoted by $v_n$), we
have $\nabla v_n\weakto\nabla \bar{v}_0$ in $L^2$ for some lattice
function $v_0 : \Z^d \to \R^m$.  Further, similarly to \cite[Lemma
4.9]{EhrOrtSha:defectsV1}, there exists a sequence $\check{r}_n \to
\infty$, $\check{r}_n < \frac12 \Ra_n$, such that, defining $w_n :=
\overline{\eta_n v_n}$, where $\eta_n$ is a smooth cut-off function
satisfying
\[
\eta_n(\xi) = 1
\quad(|\xi| \leq \check{r}_n+2\rcut)
\qquad\text{and}\qquad
\eta_n(\xi) = 0
\quad(|\xi|\geq 2\check{r}_n-2\rcut)
,
\]
(cf. the definition of the truncation operator $T_R$ in
\eqref{eq:truncation_op}) and $z_n := v_n - w_n$, then
\begin{align*}
D w_n \to D v_0 \text{ in $\ell^2$,}
\qquad&
\text{\phantom{and}}\qquad
\nabla w_n \to \nabla \bar{v}_0 \text{ in $L^2$,}
\\
D z_n \weakto 0 \text{ in $\ell^2$,}
\qquad&
\text{\phantom{and}}\qquad
\nabla z_n \weakto 0 \text{ in $L^2$,}
\\
D w_n(\xi) = \cases{
D v_n(\xi), & |\xi| \leq \check{r}_n, \\
0, & |\xi| \geq 2 \check{r}_n, \\
}
\qquad&
\text{and}\qquad
\nabla w_n(x) = \cases{
\nabla v_n(x), & |x| \leq \check{r}_n, \\
0, & |x| \geq 2 \check{r}_n. \\
}
\end{align*}
We note that $w_n = 0$ on $\Omc$ and hence $w_n$ is an admissible
displacement, $w_n \in \Us_{n}$, which also ensures that $z_n \in
\Us_{n}$.  The statement that $D z_n \rightharpoonup 0$ follows from
the fact that, for any fixed $\varphi \in \Usc$, $\< D z_n, D \varphi
\> \to 0$ as $\La$ will eventually enclose the support of $\varphi$
for sufficiently large $n$.

Hence we have
\begin{align*}
\< \ddel [\E^{\beta}_{n}+\Pa](\Pi_n y) v_n, v_n \>
=~&
\< \ddel \E^{\beta}_{n}(\Pi_n y) w_n + \ddel\Pa(y) w_n, w_n \>
\\~&
+
2 \< \ddel \E^{\beta}_{n}(\Pi_n y) w_n, z_n \>
+
\< \ddel \E^{\beta}_{n}(\Pi_n y) z_n, z_n \>
\\
=:~&
a_n
+
2 b_n
+
c_n.
\end{align*}
Here we used the fact that, for $n$ large enough, $\Pa(\Pi_n y) =
\Pa(y)$ and is supported outside $\supp(D z_n)$ or $\supp(\nabla
z_n)$.

Due to $\check{r}_n < \frac12 \Ra_n$ and the stability assumption \eqref{eq:int:stab:stab_ass} we have that
\begin{align*}
a_n
=~&
\< \ddel [\Ea + \Pa](\Pi_n y) w_n, w_n \>
= 
\< \ddel [\Ea + \Pa](y) w_n, w_n \>
\geq \gamma^\a(y) \|\nabla w_n\|_{L^2}^2
.
\end{align*}

Similarly, since $D w_n(\xi)$ can be nonzero only for $\xi$ such that $\beta(\xi)=1$, we have that
\begin{align*}
b_n
&=
	\< \ddel \E^{\beta}_{n}(\Pi_n y) w_n, z_n \>
=
	\< \ddel \Ea(y) w_n, z_n \> \\
&=
\sum_{\xi\in\L}  \< \ddel V(Dy(\xi)) D w_n(\xi), D z_n(\xi) \>.
\end{align*}
Since $\ddel V(Dy) D w_n \to \ddel V(Dy) D v_0$ in $\ell^2$ and $D z_n
\weakto 0$ in $\ell^2$ it follows that $b_n \to 0$.

Finally, the fact that $\|\nabla \Pi_n y - \mA
\|_{L^\infty(\R^d\setminus B_{\check{r}_n})} \to 0$ as $\check{r}_n\to\infty$ and the
Lipschitz regularity of $\ddel V$ and $\partial^2 W$ imply that
\begin{align*}
\|\ddel V(D \Pi_n y) - \ddel V(\mA \Rg)\|_{\ell^\infty(\supp(D z_n))}
\leq
\|\ddel V(D \Pi_n y) - \ddel V(\mA \Rg)\|_{\ell^\infty(\Z^d\setminus B_{\check{r}_n})}
\to~& 0
,\quad\text{and} \\
\|\partial^2 W(\nabla \Pi_n y) - \partial^2W(\mA)\|_{L^\infty(\supp(\nabla z_n))}
\leq
\|\partial^2W(\nabla \Pi_n y) - \partial^2W(\mA)\|_{L^\infty(\R^d\setminus B_{\check{r}_n})}
\to~& 0
\end{align*}
which, upon writing out $\ddel \E^{\beta}_{n}(\Pi_n y)-\ddel \E^{\beta}_{n}(\mA x)$ and estimating $\|D z_n\|_{\ell^2}$ by $\|\nabla z_n\|_{L^2}$, allow us to conclude that
\[
c_n
=
\< \ddel \E^\beta_{n}(\mA x) z_n, z_n \>
+ o(1) \|\nabla z_n\|_{L^2}^2
\geq (\gamma^\beta_h(\mA x) + o(1)) \|\nabla z_n\|_{L^2}^2,
\]
where $o(1)$ denotes a sequence that converges to $0$ as $n\to\infty$.

It remains only to observe that
\begin{align*}
&
\hspace{-1.5cm} \gamma^\a(y) \|\nabla w_n\|_{L^2}^2
+
(\gamma^\beta_h(\mA x) + o(1)) \|\nabla z_n\|_{L^2}^2
\\ \geq~&
	\min\{\gamma^\a(y), \gamma^\beta_h(\mA x) + o(1)\} \,
	(\|\nabla w_n\|_{L^2}^2 + \|\nabla z_n\|_{L^2}^2)
\\ =~&
\min\{\gamma^\a(y), \gamma^\beta_h(\mA x) + o(1)\} \, (\|\nabla v_n\|_{L^2}^2
	-2 (\nabla w_n, \nabla z_n)_{L^2} )
\\ =~&
\min\{\gamma^\a(y), \gamma^\beta_h(\mA x)\} + o(1),
\end{align*}
where we used again that fact that $\D w_n$ converges strongly while
$\D z_n \rightharpoonup 0$.

Thus, we have arrived at a contradiction to our original assumption,
and have therefore established the result.
\end{proof}

In the proof of Lemma \ref{th:int:stab_bqce_lem2} we will use the
following auxiliary result.

\begin{lemma}\label{lem:stab:sqrt}
If $\beta:\R^d\to\R$ satisfies \eqref{eq:blend_stab:beta_bound} then $\sqrt{\beta},\sqrt{1-\beta}\in W^{1,\infty}$ and
\[
\max\big\{
	\|\nabla \sqrt{\beta}\|_{L^\infty},\|\nabla \sqrt{1-\beta}\|_{L^\infty}
\big\}
\leq \sqrt{\|\nabla^2 \beta\|_{L^\infty}/2}
.
\]
\end{lemma}
\begin{proof}
  It is proved in \cite{Glaeser} that $\sqrt{\beta}$ is continuously
  differentiable on $\Omega = \{x \sep \beta(x)\ne 0 \text{ or }
  \nabla^2\beta(x)= 0\}$ and $\|\nabla
  \sqrt{\beta}\|_{L^\infty(\Omega)} \leq \sqrt{\|\nabla^2
    \beta\|_{L^\infty}/2}$.  It remains to notice that $\Omega$ is
  everywhere dense, hence $\sqrt{\beta}$ is Lipschitz everywhere,
  i.e., $\sqrt{\beta}\in W^{1,\infty}$.  The result for
  $\sqrt{1-\beta}$ follows similarly.
\end{proof}

\begin{proof}[Proof of Lemma \ref{th:int:stab_bqce_lem2}.]
  As in the proof of Lemma \ref{th:int:stab_bqce_lem1} we assume, for
  contradiction, that there exists a sequence $\beta_n$, $\Ts_n$, $v_n
  \in \Us_n$, $\|\nabla v_n\|_{L^2}=1$ etc.\ (again, we omit the
  subscript $h$) such
  that \begin{equation}\label{eq:blend_stab:contrary}
    \lim_{n\to\infty} \b\< \ddel \E^{\beta}_{n}(\mA x) v_n, v_n \b\> <
    \gamma^\a = \gamma^\a(\mA x).
\end{equation}

We introduce the parameter $\eps_n=1/\Ra_n \to 0$, rescale
variables,
\[
x\mapsto \eps_n x,
\quad
\xi\mapsto \eps_n \xi,
\quad
v_n\mapsto \eps_n^{1-d/2} v_n,
\quad
\Omega_n\mapsto \eps_n \Omega_n,
\]
and define $\|w_n\|_{\ell^2(\eps_n \Z^d)}^2 := \eps_n^d
\sum_{\xi\in\eps_n\Z^d} |w_n(\xi)|^2$. We observe that $\|\nabla
v_n\|_{L^2}=1$ is preserved under this rescaling, while
\eqref{eq:blend_stab:contrary} now reads $\lim_{n\to\infty} \<H_n v_n,
v_n\> < \gamma^\a $, where
\begin{align}
\notag
&\<H_n v_n, v_n\>
:=
	\eps_n^d \sum_{\xi\in\eps_n\Z^d} (1-\beta_n(\xi)) \<A \, D_n v_n(\xi), D_n v_n(\xi)\>
	+
	\int_{\Omega_n} (Q_n \beta_{n}) (\C \!:\! \nabla v_n) \!:\! \nabla v_n \dx,
 \\
\label{notation:bqce_stab}
&A := \ddel V(\mA\Rg), \quad \C := \partial^2 W(\mA), \quad \text{ and
} 
\quad D_n w(\xi) =
\bg(\frac{w(\xi+\eps_n\rho)-w(\xi)}{\eps_n}\bg)_{\rho\in\Rg}.
\end{align}

Upon extracting a subsequence we have that $\nabla v_n \weakto \nabla
v_0$ in $L^2$ for some $v_0\in H^1_\loc(\R^d)$.  Hence we define $w_n
:= \Pi_h(\eta_{r_n} \ast v_0) \in \Us_n$ and split $v_n = w_n + z_n$,
where $\eta_r\in C^\infty(\R^d)$ is a family of mollifiers, and the
sequence $r_n\to 0$ will be chosen later.  Since $\nabla w_n \to
\nabla v_0$ in $L^2$, we have that $\nabla z_n \weakto 0$ in $L^2$.

{\it Step 1: estimating $\<H_n z_n, z_n\>$.}

{\it Step 1.1: continuum contribution. }  We start by bounding the
continuum contribution from $\<H_n z_n, z_n\>$,
\[
	\int_{\Omega_n} (Q_n \beta_{n}) (\C \!:\! \nabla z_n) \!:\! \nabla z_n \dx
.
\]
Due to rescaling $\beta_n(x) \mapsto \beta_n (\eps_n^{-1} x)$ and
$\eps_n=1/\Ra_n$, we now have a uniform bound $|\nabla^2 \beta_n| \leq
C^\beta_2$.  Hence, the error of interpolation of $\beta_n$ tends to
zero due to the assumption \eqref{eq:blend_stab:tri}, i.e., $\|Q_n
\beta_{n} - \beta_{n}\|_{L^{\infty}} \to 0$, which enables us to
replace $Q_n \beta_{n}$ by $\beta_n$ while making at most $o(1)$ error
as $n\to\infty$.

\def\Ras{\hat{R}}
\def\Bas{\hat{B}}

For ease of notation, let $\Ras := C^\beta_1$, so that $\eps_n
R^\beta_n = R^\beta_n / R^\a_n \leq \Ras$, and $\Bas := B_{\Ras}$;
cf. \eqref{eq:blend_stab:radii}.

Upon shifting the test function we may assume that $\mint_{\Bas} v_n
\dx = 0$. (Note that the shifted test function does not satisfy the
homogeneous Dirichlet boundary condition, but this is irrelevant for
the following estimates.)  Therefore, due to (i) norm equivalence
$\|v_n\|_{H^1(\Bas)} \lesssim \|\nabla v_n\|_{L^2(\Bas)}$ and (ii) the
compactness of the embedding $L^2(\Bas) \subset H^1(\Bas)$, we have
that $\|z_n\|_{L^2(\Bas)} \to 0$.  Further,
\eqref{eq:blend_stab:beta_bound} implies that $\sqrt{\beta_n} =:
\varphi_n \in W^{1,\infty}$ and that it satisfies the bound $\|\nabla
\varphi_n\|_{L^\infty} \leq \sqrt{\|\nabla^2 \beta_n\|_{L^\infty}/2}
\leq \sqrt{C^\beta_2/2}$, as we have proved in Lemma
\ref{lem:stab:sqrt}.

Noting that $\supp(\varphi_n) \subset \Bas$, we have that
\begin{align*}
& \bigg|
	\int_{\Omega_n} \beta_n (\C \!:\! \nabla z_n) \!:\! \nabla z_n
-
	\int_{\Omega_n} (\C \!:\! \nabla (\varphi_n z_n)) \!:\! \nabla (\varphi_n z_n)
\bigg|
\\=~&
\bigg|
	2 \int_{\Omega_n} (\C \!:\! (\nabla\varphi_n\otimes z_n)) \!:\! \nabla (\varphi_n z_n)
	+
	\int_{\Omega_n} (\C \!:\! (\nabla\varphi_n\otimes z_n)) \!:\! (\nabla\varphi_n\otimes z_n)
\bigg|
\\\lesssim~& 2 \|z_n\|_{L^2(\Bas)} \|\nabla z_n\|_{L^2(\Bas)} +
\|z_n\|_{L^2(\Bas)}^2 \to 0 \quad \text{ as } n \to \infty.
\end{align*}
Thus,
\begin{align}
	\int_{\Omega_n} (Q_n \beta_n) (\C \!:\! \nabla z_n) \!:\! \nabla z_n
=~&
	\int_{\Omega_n} \beta_n (\C \!:\! \nabla z_n) \!:\! \nabla z_n
	+ o(1)
  \notag
\\ =~&
	\int_{\Omega_n} (\C \!:\! \nabla (\varphi_n z_n)) \!:\! \nabla (\varphi_n z_n)
	+ o(1)
  \notag
\\ \geq~&
	\gamma_{\rm c} \|\nabla (\varphi_n z_n)\|_{L^2}^2
	+ o(1)
  \notag
\\ =~&
	\gamma_{\rm c} \big\|\nabla \big(\sqrt{\beta_n} z_n\big)\big\|_{L^2}^2
	+ o(1)
.
\label{eq:stab_bqce_step1.1_final}
\end{align}

In the last estimate we used two facts: (i) the stability
\eqref{eq:strong_stab_eq} of the exact solution implies the stability
of the far-field \cite{EhrOrtSha:defects}, that is,
\[
\eps_n^d \sum_{\xi\in\eps_n\Z^d} \b\< \ddel V(\mA \Rg) D_n w_n(\xi),
D_n w_n(\xi) \b\>
\geq \gamma^\a \|\nabla w_n\|_{L^2}^2
;
\]
and (ii) that atomistic stability implies continuum stability
\cite{Hudson:stab}, that is,
\[
	\int (\C \!:\! \nabla w_n) \!:\! \nabla w_n
	\geq \gamma^\c \|\nabla w_n\|_{L^2}^2 \qquad \text{where}
        \quad \gamma^\c = \gamma^\c(\mA) \geq \gamma^\a(\mA x).
\]

{\it Step 1.2.}  A similar argument can be applied to the atomistic
contribution to $\<H_n z_n, z_n\>$.  We introduce the translation
operator $T_n w_n(\xi) := (w(\xi+\eps_n \rho))_{\rho\in\Rg}$ and the
product $D_n\varphi_n T_n z_n := ( (D_n \varphi_n)_\rho (T_nz_n)_\rho
)_{\rho \in \Rg}$.  Then, redefining $\varphi_n := \sqrt{1-\beta_n}
\in W^{1,\infty}$, $\|\nabla \varphi_n\| \leq \sqrt{C^\beta_2/2}$, we
obtain
\begin{align}
& \notag
	 \bigg|
	\eps_n^d \sum_{\xi\in\eps_n\Z^d} (1-\beta_n(\xi)) \B\<A \,D_n z_n(\xi), D_n z_n(\xi)\B\>
	-
	\eps_n^d \sum_{\xi\in\eps_n\Z^d} \B\<A \,D_n (\varphi_n
        z_n)(\xi), D_n (\varphi_n z_n)(\xi)\B\>
	\bigg|
\\ & \qquad = \notag
	\bg| \eps_n^d \sum_{\xi\in\eps_n\Z^d} \B\<
		A \,
		\b(D_n \varphi_n T_nz_n\b)(\xi) ,
		2 D_n (\varphi_n z_n)(\xi) + \b(D_n
                \varphi_n T_n z_n\b)(\xi)\B\> \bg|
\\ & \qquad \lesssim \label{eq:int:stab_bqce_lem1:intrmed}
	\big(2 \|\nabla \barz_n\|_{L^2(\supp(\varphi_n))}
	+ 	\|\barz_n\|_{L^2(\supp(\varphi_n))}\big)
	\|\barz_n\|_{L^2(\supp(\varphi_n))},
\end{align}
where we used rescaled versions of the local norm-equivalence and
inverse estimates \eqref{eq:inv_est_poly}.

Next, we notice that the mesh $\TT_n$ is fully refined on
$\supp(\varphi_n)$ (cf.\ the assumption \eqref{eq:blend_stab:tri}),
hence $\barz_n=z_n$ on $\supp(\varphi_n)$, and therefore
\eqref{eq:int:stab_bqce_lem1:intrmed} tends to zero as $n \to \infty$.
Thus,
\begin{align*}
  & \hspace{-2cm} \eps_n^d \sum_{\xi\in\eps_n \Z^d} (1-\beta_n(\xi))
  \B\< A
  \, D_n z_n(\xi), D_n z_n(\xi) \B\> \\[-2mm]
&=
	\eps_n^d \sum_{\xi\in\eps_n \Z^d} \B\< A D_n
        \big(\varphi_n(\xi) z_n(\xi)\big), D_n \big(\varphi_n(\xi)
        z_n(\xi)\big) \B\>
	+ o(1) \hspace{-2cm}
\\ &\geq
	\gamma_{\rm a} \big\|\nabla \big(\overline{\varphi_n z_n}\big) \big\|_{L^2}^2
	+ o(1)
.
\end{align*}

Next, we need to prove that $\|\nabla(\overline{\varphi_n z_n} -
\varphi_n z_n)\|_{L^2} \to 0$.  Indeed, $\nabla(\overline{\varphi_n
  z_n} - \varphi_n z_n)$ can be nonzero only in those $T\in\TT_n$
where $\varphi_n$ is not constant, and all such triangles are
contained in $\Bas$, which implies
\[
\|\nabla(\overline{\varphi_n z_n} - \varphi_n z_n)\|_{L^2(\R^d)}
=
\|\nabla(\overline{\varphi_n z_n} - \varphi_n z_n)\|_{L^2(\Bas)}.
\]
Upon defining the oscillation operator $\osc_T(f) := \sup_{x,y\in T}
|f(x)-f(y)|$ we can estimate the right-hand side, for any $T \in
\Ts_n$, by
 \begin{align*}
 \|\nabla(\overline{\varphi_n z_n} - \varphi_n z_n)\|_{L^\infty(T)}
 \leq~& \osc_T (\nabla(\varphi_n z_n))
 \\ \leq~& \osc_T (\nabla\varphi_n \, z_n) + \osc_T (\varphi_n \nabla z_n)
 \\ \leq~& 2 \|\nabla\varphi_n\|_{L^\infty(T)} \|z_n\|_{L^\infty(T)}
 	+ \osc_T (\varphi_n) \big|\nabla z_n|_T\big|
 \\ \lesssim ~& 2 \|\nabla\varphi_n\|_{L^\infty(T)} \|z_n\|_{L^2(T)}
 	+ \eps_n \|\nabla \varphi_n\|_{L^\infty(T)} \big|\nabla z_n|_T\big|,
 \end{align*}
 where in the last step we used the fact that ${\rm diam}(T) \lesssim
 \eps_n$ and that $\|z_n\|_{L^\infty(T)} \lesssim \|z_n\|_{L^2(T)}$
 since $z_n$ is a linear function on $T$.

 Then summing the contributions over all $T \subset \Bas$, we obtain
\[
\|\nabla(\overline{\varphi_n z_n} - \varphi_n z_n)\|_{L^2(\Bas)}
\lesssim 2 \|\nabla\varphi_n\|_{L^\infty} \|z_n\|_{L^2(\Bas)} + \eps_n
\|\varphi_n\|_{L^\infty} \|\nabla z_n\|_{L^2(\Bas)} \to 0,
\]
since $\|z_n\|_{L^2(\Bas)}\to 0$ and $\eps_n\to 0$ as $n \to \infty$.

Thus,
\begin{align}
	\eps_n^d \sum_{\xi\in\eps_n \Z^d} (1-\beta_n(\xi)) \B\<A \,
        D_n z_n(\xi), D_n z_n(\xi) \B\>
\geq ~&
	\gamma_{\rm a} \|\nabla (\varphi_n z_n) \|_{L^2}^2
	+ o(1)
\notag
\\[-2mm] =~&
	\gamma_{\rm a} \big\|\nabla \big(\sqrt{1-\beta_n} z_n\big) \big\|_{L^2}^2
	+ o(1)
.
\label{eq:stab_bqce_step1.2_final}
\end{align}

{\it Step 1.3.} Combining \eqref{eq:stab_bqce_step1.1_final} and
\eqref{eq:stab_bqce_step1.2_final},and using $\gamma^\c \geq
\gamma^\a$, we obtain
\begin{align*}
	\<H_n z_n,z_n\>
\geq~&
	\gamma_{\rm a} \int_{\Omega_h} \Big(\big|\nabla \big(\sqrt{1-\beta_n}\, z_n\big)\big|^2 + \big|\nabla \big(\sqrt{\beta_n}\, z_n\big)\big|^2\Big) \dx
	+ o(1)
.
\end{align*}
Then arguing similarly to the above (expanding the gradient of a
product and exploiting the fact that $\|z_n\|_{L^2(\Bas)}\to 0$) we
conclude that
\begin{align*}
	\int_{\Omega_h} \Big(
		\big|\nabla \big(\sqrt{1-\beta_n}\, z_n\big)\big|^2 + \big|\nabla \big(\sqrt{\beta_n}\, z_n\big)\big|^2
	\Big) \dx
=~&
	\int_{\Omega_h} \big((1-\beta_n) |\nabla z_n|^2 + \beta_n |\nabla z_n|^2\big) \dx
	+ o(1)
\\ =~&
	\int_{\Omega_h} |\nabla z_n|^2 \dx
	+ o(1)
.
\end{align*}

Summarizing, in Step 1 we proved that
\begin{equation}
  \label{eq:stab_bqce_step1_final}
  \<H_n z_n,z_n\> \geq \gamma_{\rm a} \|\nabla z_n\|_{L^2}^2 + o(1).
\end{equation}

{\it Step 2: estimating $\<H_n w_n, w_n\>$.}

Since $\supp(\beta_n)$ is contained in $\Bas$ and $\nabla^2\beta_n$ is
uniformly bounded, we have that, up to extracting a subsequence,
$\beta_n \to \beta_0$ in $C^1$ for some $\beta_0\in C^1(\R^d)$.  Due
to the strong convergence $Q_n \beta_n\to\beta_0$ in $L^\infty$ and
$w_n\to v_0$ in $L^2$, it is straightforward to evaluate the limit of
the continuum contribution to $\<H_n w_n, w_n\>$:
\begin{equation}
  \label{eq:stab_bqce_step2_continuum}
	\int_{\Omega_n} (Q_n \beta_n) (\C \!:\! \nabla w_n) \!:\! \nabla w_n \dx
=
	\int_{\Omega_n} \beta_0 (\C \!:\! \nabla v_0) \!:\! \nabla v_0 \dx + o(1)
.
\end{equation}

To evaluate the limit of the atomistic contribution to $\<H_n w_n,
w_n\>$, recall the definition \eqref{notation:bqce_stab} of $D_n$ and
notice that for a fixed $r>0$, $\|D_n (\eta_{r}*v_0) - \nabla_\Rg
(\eta_{r}*v_0)\|_{\ell^2(\eps_n\Z^d)} \to 0$ as $D_n (\eta_{r}*v_0)$
is a finite difference approximation to the derivative of a smooth
function, $\nabla_\Rg (\eta_{r}*v_0)$. Hence, since
$\|\beta_n-\beta_0\|_{\ell^\infty(\eps_n\Z^d)} \leq
\|\beta_n-\beta_0\|_{L^\infty (\mathbb{R}^d)} \to 0$, we obtain
\begin{align*}
&
\lim_{n\to\infty} \eps_n^d
\sum_{\xi\in\eps_n\Z^d}
	(1-\beta_n(\xi)) \B\< A D_n (\eta_{r}*v_0)(\xi), D_n (\eta_{r}*v_0)(\xi)\B\>
\\ =~&
\lim_{n\to\infty} \eps_n^d\sum_{\xi\in\eps_n\Z^d}
	(1-\beta_0(\xi)) \B\< A \nabla_\Rg (\eta_{r}*v_0)(\xi),
        \nabla_\Rg (\eta_{r}*v_0)(\xi) \B\>
\\ =~&
\int_{\R^d}
	(1-\beta_0) \b\< A \nabla_\Rg (\eta_{r}*v_0), \nabla_\Rg
        (\eta_{r}*v_0)\b\> \dx
.
\end{align*}
In the last step we used the fact that a summation rule applied to a
smooth function converges to its integral.

Next, we notice that $\nabla (\eta_{r}\ast v_0) \to \nabla v_0$ in
$L^2$, as $r\to0$, hence
\begin{align*}
  & \lim_{r\to0}\lim_{n\to\infty}
\eps_n^d
\sum_{\xi\in\eps_n\Z^d}
	(1-\beta_n(\xi)) \B\< A D_n (\eta_{r}*v_0)(\xi), D_n
        (\eta_{r}*v_0)(\xi) \B\> \\
        & \qquad \qquad =
\int_{\R^d}
	(1-\beta_0) \b\< A \nabla_\Rg v_0, \nabla_\Rg v_0\b\> \dx
.
\end{align*}
Therefore there exists a sequence $r_n \downarrow 0$ (sufficiently
slowly) such that
\begin{align*}
  & \lim_{n\to\infty}
\eps_n^d
\sum_{\xi\in\eps_n\Z^d}
	(1-\beta_n(\xi)) \B\< A D_n (\eta_{r_n}*v_0)(\xi), D_n
        (\eta_{r_n}*v_0)(\xi) \B\> \\
& \qquad \qquad =
\int_{\R^d}
	(1-\beta_0) \b\< A \nabla_\Rg v_0, \nabla_\Rg v_0\b\> \dx
\\ & \qquad \qquad =
\int_{\Omega_n}
	(1-\beta_0) (\C \!:\! \nabla v_0) \!:\! \nabla v_0 \dx
.
\end{align*}

Finally it remains to notice that due to the full refinement of
$\TT_n$ on $\supp(1-\beta_n)$, $w_n := \Pi_n (\eta_{r_n}*v_0) =
\eta_{r_n}*v_0$, hence
\begin{equation}
  \label{eq:bqce_stab_step2_atm}
	\eps_n^d \sum_{\xi\in\eps_n \Z^d} (1-\beta_n(\xi)) \b\< A \,
        D_n w_n(\xi), D_n w_n(\xi) \b\>
	\to
	\int_{\Omega_n} (1-\beta_0) \, (\C \!:\! \nabla v_0) \!:\! \nabla v_0 \dx
.
\end{equation}

Combining the estimates for the atomistic contribution
\eqref{eq:bqce_stab_step2_atm} with that for the continuum
contribution \eqref{eq:stab_bqce_step2_continuum} we finally deduce
that
\begin{align}
  \notag
\<H_n w_n, w_n\>
=~&
	\int (\C \!:\! \nabla v_0) \!:\! \nabla v_0 + o(1)
\\ \geq ~&
	\gamma_{\rm c} \|\nabla v_0\|_{L^2}^2 + o(1)
 =
	\gamma_{\rm c} \|\nabla w_n\|_{L^2}^2 + o(1)
.
\label{eq:bqce_stab_step2_final}
\end{align}

{\it Step 3: estimating the cross terms $\<H_n w_n, z_n\>$. }
%
%

Since $\nabla z_n \weakto 0$ and $\nabla w_n\to \nabla v_0$ in $L^2$,
and $Q_n \beta_{n} \to \beta_0$ in $L^\infty$, we trivially have that
\[
	\int_{\Omega_n} (Q_n \beta_{n}) (\C \!:\! \nabla w_n) \!:\! \nabla z_n \dx
	\to
	\int_{\Omega_n} \beta_0 (\C \!:\! \nabla v_0) \!:\! 0 \dx
	=
	0
.
\]

To prove that
\begin{equation}
  \label{eq:stab2_step3_main}
  \eps_n^d \sum_{\xi\in\eps_n \Z^d} (1-\beta_n) \<A D_n w_n, D_n
  z_n \> = o(1)
\end{equation}
we convert the sum to stress-strain form as in
\S~\ref{sec:int:stress_at_cb}. Let $\zz_n(\xi) = \eps_n^{-d} \, \zz(\xi/\eps_n)$ be the rescaled hat function
and $\psi_n : \eps \Z^d \to \R^m$ such that $\psi_n^* := (\zz_n \ast
\bar{\psi}_n) = z_n$ on $\eps_n \L$,
then
\begin{align*}
  &\eps_n^d \sum_{\xi\in\eps_n \Z^d} (1-\beta_n) \<A D_n w_n, D_n
  z_n \> = \int \mS_n : \D \bar{\psi}_n,  \qquad \text{where} \\
  &\mS_n(x) = \eps_n^d \sum_{\xi \in \eps\L} (1-\beta_n(\xi))
  \sum_{\rho,\vsig \in \Rg} \b[ (A_{\rho\vsig} D_{n,\vsig} w_n(\xi))
  \otimes \vsig\b] \mint_{t = 0}^{\eps_n} \zz_n(\xi+t \rho) \dt.
\end{align*}
We can now argue analogously as in Step 2 to prove that
\begin{displaymath}
  \mS_n(x) \to (1-\beta_0) \bbC : \D v_0 \qquad \text{ strongly in } L^2,
\end{displaymath}
again requiring that $r_n \to 0$ sufficiently slowly (possibly at a
slower rate than in Step 2).  Thus, if we can prove that $\D\bar\psi_n
\rightharpoonup 0$ in $\Bas$, then \eqref{eq:stab2_step3_main}
follows.

To that end, let $\mu \in C^\infty_{\rm c}(\Bas; \R^m)$ be a test
function with compact support, then $\|\zz_n \ast \D \mu - \D \mu
\|_{L^\infty} \to 0$ as $n \to \infty$ and hence,
\begin{align*}
  \int \D \bar\psi_n : \D \mu \dx &=  \int \D \bar\psi_n : (\zz_n \ast
  \D \mu) \dx + o(1) = \int \D \b(\zz_n \ast \bar\psi_n\b) : \D  \mu
  \dx + o(1)  \\
  &= \int \D\psi_n^* : \D \mu \dx + o(1) = \int \psi_n^* \cdot \Delta \mu \dx + o(1).
\end{align*}
Due to local norm-equivalence in each element, we have that
\begin{displaymath}
  \| \psi_n^* \|_{L^2(\Bas)} \lesssim \| \psi_n^* \|_{\ell^2(\eps \L \cap
  \Bas)} =  \| z_n \|_{\ell^2(\eps \L \cap \Bas)} \lesssim \| z_n
\|_{L^2(\Bas)} \to 0 \quad \text{as } n \to \infty.
\end{displaymath}
Hence, it follows that $ \int \tilde\psi_n \cdot \Delta \mu \dx \to
0$, which completes the proof that $\D\bar\psi_n \rightharpoonup 0$,
and hence also the proof of \eqref{eq:stab2_step3_main}. Thus, we have
established that
\begin{equation}
  \label{eq:bqce_stab_step3_final}
  \< H_n w_n, z_n \> \to 0 \qquad \text{as } n \to \infty.
\end{equation}



{\it Step 4: conclusion of the proof. }
Combining \eqref{eq:stab_bqce_step1_final},
\eqref{eq:bqce_stab_step2_final} and \eqref{eq:bqce_stab_step3_final}
we obtain
\begin{align*}
	\<H_n v_n, v_n\>
\geq~&
	\gamma_{\rm a} \|\nabla z_n\|_{L^2}^2
	+ \gamma_{\rm a} \|\nabla w_n\|_{L^2}^2 + o(1)
\\=~&
	\gamma_{\rm a} \|\nabla v_n\|_{L^2}^2
	-2\gamma_{\rm a} (\nabla z_n, \nabla w_n)_{L^2}^2 + o(1)
\\=~&
	\gamma_{\rm a} \|\nabla v_n\|_{L^2}^2
	+ o(1)
=
	\gamma_{\rm a}
	+ o(1).
\end{align*}
Thus, we have a contradiction to our initial assumption that
$\lim_{n\to\infty} \<H_n v_n, v_n\> < \gamma^\a$.
\end{proof}

\subsection{BQCF stability}
\label{sec:prf_bqcf_stab}
The main step towards the proof of Lemma \ref{th:int:stab_bqcf_lem1}
is the following estimate.

\begin{lemma}\label{lem:BQCE-BQCF-equiv}
  There exists $C$, independent of $(\beta, \T)$ such that
\begin{align} \notag
&|\< \ddel \Ebh(y_h) v_h - \del \Fbh(y_h) v_h, v_h \>| \leq C E
\qquad \forall v_h \in \Us_h, \\[1mm]
\notag
\text{where} \quad &E :=
	\|\nabla y_h-\mA\|_{L^\infty(\Omb)}
	\big( \|\nabla v_h\|_{L^2}^2
		+ \|v_h\|_{L^2(\Omb)} \|\nabla v_h\|_{L^2} \|\nabla \beta\|_{L^{\infty}}
	\big)
\\~& \qquad\quad + \label{eq:bqce-bqcf-equiv-def_E}
\|\nabla \beta\|_{W^{1,\infty}} \|\nabla v_h\|_{L^2}^2
	+
	\|\nabla^2 \beta\|_{L^\infty} \|v_h\|_{L^2(\Omb)}
	\|\nabla v_h\|_{L^2}
.
\end{align}
\end{lemma}
{ 
  \begin{proof} {\it Step 1: reduction to the homogeneous case. } 
    Let $\ddel V_\xi := \ddel V(D y_h(\xi))$ and $\ddW := \ddW(\D
    y_h(x))$, $A := \ddel V(\mA \Rg)$ and $\bbC := \ddW(\mA)$.

    Then, the difference in the linearised operators is given
    by 
\begin{align*}
&
\hspace{-1.5cm} \< \ddel \Ebh(y_h) v_h - \del \Fbh(y_h) v_h, v_h \>
\\ =~&
	\sum_{\xi \in \L} (1-\beta(\xi)) \b\< \ddel V_\xi D v_h(\xi),
        D v_h(\xi) \b\>
	+ \int_\Omh (Q_h\beta) \big(\ddW \!:\! \nabla v_h\big) \!:\! \nabla v_h \dx,
\\~& \quad -
	\sum_{\xi \in \L} \b\< \ddel V_\xi D v_h(\xi), D ((1-\beta)
        v_h)(\xi) \b\>
	-
	\int_\Omh (\ddW \!:\! \nabla v_h) \!:\! \nabla I_h (\beta v_h) \dx
\\ =~&
	\sum_{\xi \in \L} \b\< \ddel V_\xi D v_h(\xi),
	\big( 	-\beta D v_h + D (\beta v_h) \big)(\xi) \b\>
\\~& \quad +
	\int_\Omh (\ddW \!:\! \nabla v_h) \!:\!
	\big( (Q_h \beta) \nabla v_h - \nabla I_h (\beta v_h) \big)
	\dx
\\ =~&
	\sum_{\xi \in \Lb} \b\< \ddel V_\xi D v_h(\xi),
	\big( 	-\beta D v_h + D (\beta v_h) \big)(\xi) \b\>
\\~& \quad +
	\int_\Omb (\ddW \!:\! \nabla v_h) \!:\!
	\big( (Q_h \beta) \nabla v_h - \nabla I_h (\beta v_h) \big)
	\dx.
\end{align*}
In the last step we used the fact that the summand is nonzero only if
$\xi\in\Lb$ and the integrand is nonzero only if $x\in\Omb$, where
$\Lb$ and $\Omb$ are defined in \S~\ref{sec:approx_error}.

For such $\xi$ and $x$ we can estimate $|\ddel V(D y_h(\xi))- A|
\lesssim \|D y_h-\mA\Rg\|_{\ell^\infty(\Lb)} \lesssim \|\nabla
y_h-\mA\|_{L^\infty(\Omb)}$ and $|\ddW(\nabla y_h(\xi))-\ddW(\mA)|
\lesssim \|\nabla y_h-\mA\|_{L^\infty(\Omb)}$.  Hence, we can estimate
\begin{align}
&
\hspace{-1.5cm} \big|
	\< \ddel \Ebh(y_h) v_h - \del \Fbh(y_h) v_h, v_h \>
	-
	\< \ddel \Ebh(\mA x) v_h - \del \Fbh(\mA x) v_h, v_h \>
\big|
\notag
\\ \lesssim ~&
	\|D y_h-\mA\Rg\|_{\ell^\infty(\Lb)}
	\|D v_h\|_{\ell^2(\Lb)}
	\big( \|\beta\|_{\ell^{\infty}} \|D v_h\|_{\ell^2}
		+ \|D \beta\|_{\ell^{\infty}} \|v_h\|_{\ell^2(\Lb)}
	\big)
\notag
\\~& \quad +
	\|\nabla y_h-\mA\|_{L^\infty(\Omb)}
	\|\nabla v_h\|_{L^2}
	\big( \|\nabla \beta\|_{L^{\infty}} \|v_h\|_{L^2(\Omb)} + \|\nabla v_h\|_{L^2} \|\beta\|_{L^{\infty}}
	\big) \hspace{-1cm}
\notag
\\ \lesssim ~& 
	\|\nabla y_h-\mA\|_{L^\infty(\Omb)}
	\|\nabla v_h\|_{L^2}
	\big( \|\nabla \beta\|_{L^{\infty}} \|v_h\|_{L^2(\Omb)} + \|\nabla v_h\|_{L^2}
	\big)
\notag
\\ \lesssim ~& E.
\label{eq:prf_bqcf_stab_step1_final}
\end{align}

{\it Step 2: estimate for the case $y_h = \mA x$.}  It remains to
bound $\ddel \Ebh(\mA x) - \del \Fbh(\mA x)$. To that end denote
\[
E_\xi :=
	\|\nabla \beta\|_{W^{1,\infty}(\nu_\xi)} \|\nabla v_h\|_{L^2(\nu_\xi)}^2
	+
	\|\nabla^2 \beta\|_{L^\infty(\nu_\xi)} \|v_h\|_{L^2(\nu_\xi)}
	\|\nabla v_h\|_{L^2(\nu_\xi)},
\]
where $\nu_{\xi}={B_{2\rcut+\sqrt{d}}}(\xi)$ is defined in
\eqref{eq:defn_nux}, so that $\sum_{\substack{\xi\in\Lb}} E_\xi
\lesssim E$.

Further, let $A_{\rho\sigma} = V_{,\rho\sigma}(\mA\Rg)$,
then we have
\begin{align}
&\< \ddel \Ebh(\mA x) v_h - \del \Fbh(\mA x) v_h, v_h
\>
\notag
\\ =~&
\sum_{\xi \in \Lb} \b\< A D v_h(\xi),
\big( 	-\beta D v_h + D (\beta v_h) \big)(\xi) \b\>
\notag
\\~&
\qquad  +
	\int_{\Omb} (\bbC \!:\! \nabla v_h) \!:\!
	\big( (Q_h \beta) \nabla v_h - \nabla I_h (\beta v_h) \big)
	\dx
\notag
\\ =~&
\sum_{\xi \in \Lb} \Big(\b\< A D v_h(\xi), (v_h \D\beta)(\xi) \b\> + O(E_\xi) \Big)
-
	\int_{\Omb} (\bbC \!:\! \nabla v_h) \!:\!
	\big( v_h \otimes \D\beta \big)
	\dx
\notag
\\ =~& \sum_{\rho,\sigma \in \Rg} A_{\rho\vsig} : \bg\{
\sum_{\xi \in \L} D_\rho v_h(\xi) \otimes v_h(\xi)
\D_\sigma\beta(\xi) - \int \D_\rho v_h \otimes v_h \D_\sigma \beta
\dx  \bg\} + O(E).
\label{eq:prf_bqcf_stab_step2_initialidentity}
\end{align}

Let $v(\xi) := v_h(\xi)$ for all $\xi \in \L$ and recall the
definition of $v^*$ from \eqref{eq:defn_convoluted_v}.  We observe
that the sum and integral are only taken over a region where $v_h =
\barv$ (recall that $\TT = \T$ in the blending region), hence we can
write
\begin{align}
  \notag
  \sum_{\xi \in \L} D_\rho v(\xi) \otimes v_h(\xi)
\D_\sigma\beta(\xi) &=   \sum_{\xi \in \L} D_\rho v^*(\xi) \otimes v(\xi)
\D_\sigma\beta(\xi) \\
\notag
& \qquad +   \sum_{\xi \in \L} D_\rho (v-v^*)(\xi) \otimes v(\xi)
\D_\sigma\beta(\xi)  \\
\label{eq:prf_bqcf_stab_step2_defn_ST}
&=: {\rm S}_{\rho\sigma} + {\rm T}_{\rho\sigma}.
\end{align}

{\it Step 2.1: Rewriting ${\rm S}_{\rho\sigma}$. }  Employing \eqref{eq:stress:5}
and \eqref{eq:stress:phirho_Prop0} we can write
\begin{align*}
  {\rm S}_{\rho\sigma} &= \sum_{\xi \in \L} \int_{\R^d} \D_\rho \barv(x)
  \ww_\rho(\xi-x) \dx \otimes v(\xi)
\D_\sigma\beta(\xi) \\
&= \int_{\R^d} \D_\rho \barv \otimes \bg\{ \sum_{\xi \in \L}
\ww_\rho(\xi-x) v(\xi) \D_\sigma \beta(\xi) \bg\} \\
&= \int_{\R^d} \D_\rho \barv \otimes \bg\{ \sum_{\xi \in \L}
\ww_\rho(x-\xi) \B( \barv(x) \D_\sigma \beta(x) + O\b( \|\D(\barv
\D_\sigma\beta) \|_{L^\infty(\nu_\xi)} \b) \B)\bg\}
\dx \\
&= \int_{\R^d} \D_\rho \barv \otimes \barv(x) \D_\sigma \beta(x) \dx + O(E).
\end{align*}

Thus, we observe from \eqref{eq:prf_bqcf_stab_step2_initialidentity}
and \eqref{eq:prf_bqcf_stab_step2_defn_ST} that
\begin{equation}
  \label{eq:eq:prf_bqcf_stab_step2_reduction_to_T_only}
  \< \ddel \Ebh(\mA x) v_h - \del \Fbh(\mA x) v_h, v_h
\> = \sum_{\rho,\sigma \in \Rg} A_{\rho\sigma} : {\rm T}_{\rho\sigma}
+ O(E).
\end{equation}

{\it Step 2.2: Estimating ${\rm T}_{\rho\sigma}$. } Summation by parts
yields
\begin{align*}
  \b|{\rm T}_{\rho\sigma}\b|
  &= \bg| - \sum_{\xi \in \L} (v - v_*)(\xi) \otimes D_{-\rho} \b( v
  \D_\sigma \beta\b)(\xi)\bg| \\
  & \lesssim \| v - v_* \|_{\ell^2(\Lb)} \| D_{-\rho} (v \D_\sigma
  \beta) \|_{\ell^2(\Lb)} \\
  & \lesssim \| \D v \|_{L^2(\Omb)} \b( \| \barv \|_{L^2(\Omb)} \| \D^2
  \beta \|_{L^\infty} + \| D_{-\rho} v \|_{\ell^2(\Lb)} \| \D\beta
  \|_{L^\infty} \b) \\
  & \lesssim E.
\end{align*}
Recalling \eqref{eq:eq:prf_bqcf_stab_step2_reduction_to_T_only} and
\eqref{eq:prf_bqcf_stab_step1_final}, this completes to proof of the
lemma.
\end{proof}
} 

\begin{proof}[Proof of Lemma \ref{th:int:stab_bqcf_lem1}]
  In view of Lemma \ref{lem:BQCE-BQCF-equiv} we only need to verify
  that $E \rightarrow 0 \text{ as } \Ra \rightarrow \infty, $ where
  $E$ is defined by \eqref{eq:bqce-bqcf-equiv-def_E}.  Using
    Corollary \ref{th:vol_trace_ineq} we estimate
    \[
    \begin{split}
      E\le~&
      \|\nabla y_h-\mA\|_{L^\infty(\Omb)}
      \big( \|\nabla v_h\|_{L^2}^2
      + \sqrt{C_2} \|\nabla v_h\|_{L^2}^2 \|\nabla \beta\|_{L^{\infty}}
      \big)
      \\~&+
      \|\nabla \beta\|_{W^{1,\infty}} \|\nabla v_h\|_{L^2}^2
      +
      \|\nabla^2 \beta\|_{L^\infty}\sqrt{C_2} \|\nabla v_h\|_{L^2}^2,
    \end{split}
    \]
    where $C_2 = \b( (\Rb)^2 - (\Ra)^2 \b) C_1'$ and $C_1'$ is the
    constant from Lemma \ref{th:vol_trace_ineq}.  Then, using $\|\D
    ^{j}\beta\|_{L^\infty}\lesssim (\Rb)^{-j}$, we have
\begin{align}
E\lesssim~&
	\|\nabla y_h-\mA\|_{L^\infty(\Omb)}
	\big( \|\nabla v_h\|_{L^2}^2
		+ \sqrt{C_2} \|\nabla v_h\|_{L^2}^2(\Rb)^{-1}
	\big)
\nonumber\\
~&+
	\big((\Rb)^{-2} + (\Rb)^{-1}\big) \|\nabla v_h\|_{L^2}^2
	+
	(\Rb)^{-2}\sqrt{C_2} \|\nabla v_h\|_{L^2}^2,\label{eq:bqcf_stab_bnd_E}\\
=~& I_1+I_2+I_3.\nonumber
\end{align}

To complete the estimates we note that $\sqrt{C_2} (\Rb)^{-1} \lesssim
\Ctr$ where $\Ctr$ is defined in \eqref{eq:int:est_Dbv_Dv}. Further,
the regularity estimates from Lemma \ref{th:regularity} and
\eqref{eq:regularity_ulin} yield
\begin{displaymath}
  \|\nabla y_h-\mA\|_{L^\infty(\Omb)} \lesssim \cases{ (\Ra)^{-1}, &
    \text{ case (pDis)}, \\
    (\Ra)^{-d}, & \text{ case (pPt)}.
  }
\end{displaymath}

These estimates are combined to yield
\begin{align*}
 I_1 &\lesssim \cases{ (\Ra)^{-1} (\log \Ra)^{1/2}, & \text{ case
      (pDis)}, \\
    (\Ra)^{-2} (\log \Ra)^{1/2}, & \text{ case (pPt), $d = 2$,} \\
    (\Ra)^{-3}, & \text{case (pPt), $d = 3$, }
  } \\
   I_2 &\lesssim (\Ra)^{-1},  \qquad \text{and} \\
    I_3 &\lesssim \cases{ (\Ra)^{-1} (\log\Ra)^{1/2}, & \text{ case $d =
        2$, } \\
      (\Ra)^{-1}, & \text{ case $d = 3$.}
    }
\end{align*}
This completes the proof of Lemma \ref{th:int:stab_bqcf_lem1}.
\end{proof}

\begin{remark}
The auxiliary results, Lemmas \ref{th:int:stab_bqce_lem1} and \ref{th:int:stab_bqcf_lem1}, hold under much weaker assumptions.
For instance, with extra work, Lemma \ref{th:int:stab_bqcf_lem1} can be proved for the blending width (i.e., the width of $\supp(\nabla\beta)$) scaling slower than $\Ra$ \cite{BQCF}.
However, this would not be important for the practical implementation of the method or for our error estimates.
\end{remark}

\clearpage

\bibliographystyle{plain}
\bibliography{qc}

\end{document}